\documentclass[10pt]{amsart}

\usepackage{hyperref}
\usepackage{tabmac}
\usepackage{fullpage}
\usepackage{amsfonts, amsmath, amssymb, amsthm}
\usepackage{graphicx}
\usepackage{lmodern, color} 
\usepackage[normalem]{ulem}

\newtheorem{thm}{Theorem}[section]
\newtheorem{prop}[thm]{Proposition}
\newtheorem{lemma}[thm]{Lemma}

\newtheorem{cor}[thm]{Corollary}
\theoremstyle{definition}
\newtheorem{defn}[thm]{Definition}

\theoremstyle{remark}
\newtheorem{rmk}[thm]{Remark}

\newtheorem{ex}[thm]{Example}

\newcommand{\ol}[1]{\overline{#1}}
\newcommand{\fw}[1]{{\operatorname{fw}\mathopen{}\left(#1\right)\mathclose{}}}
\newcommand{\fwC}[2]{{\operatorname{fw}_{#1}\mathopen{}\left(#2\right)\mathclose{}}}
\newcommand{\bk}[2]{{\operatorname{bk}_{#1}\mathopen{}\left(#2\right)\mathclose{}}}
\newcommand{\jbk}{\mathrm{bk}}
\newcommand{\str}[2]{{\operatorname{\mathfrak{st}}_{#1}\mathopen{}\left(#2\right)\mathclose{}}}
\def\st{{\mathfrak{st}}}
\newcommand{\Z}{\mathbb{Z}}
\newcommand{\A}{\mathcal{A}}
\newcommand{\B}{\mathcal{B}}
\newcommand{\C}{\mathcal{C}}
\newcommand{\h}{h}

\newcommand{\abs}[1]{\left\lvert#1\right\rvert}
\newcommand{\sh}[1]{\mathopen{}\left\langle#1\right\rangle\mathclose{}}

\graphicspath{{figures/}}
\newcommand{\tps}{\texorpdfstring}
\author{Michael Chmutov \and Pavlo Pylyavskyy \and Elena Yudovina}
\address{Department of Mathematics, University of Minnesota, Minneapolis, MN 55455}
\thanks{M. C. was partially supported	 by NSF grants DMS-1148634 and DMS-1503119; P. P. was partially supported by NSF grants DMS-1148634, DMS-1351590, and Sloan Fellowship.}
\dedicatory{In memory of V. A. Yasinskiy}

\begin{document}
\begin{abstract}
In his study of Kazhdan-Lusztig cells in affine type $A$, Shi has introduced an affine analog of Robinson-Schensted correspondence. We generalize the Matrix-Ball Construction of Viennot and Fulton to give a more combinatorial realization of Shi's algorithm. As a byproduct, we also give a way to realize the affine correspondence via the usual Robinson-Schensted bumping algorithm. Next, inspired by Lusztig and Xi, we extend the algorithm to a bijection between the extended affine symmetric group and collection of triples $(P, Q, \rho)$ where $P$ and $Q$ are tabloids and $\rho$ is a dominant weight. The weights $\rho$ get a natural interpretation in terms of the Affine Matrix-Ball Construction. Finally, we prove that fibers of the inverse map possess a Weyl group symmetry, explaining the dominance condition on weights.
\end{abstract}
\title[MBC for affine Robinson-Schensted correspondence]{Matrix-ball construction of affine Robinson-Schensted correspondence}
\maketitle
\tableofcontents

\section{Introduction}
\subsection{Cells in Kazhdan-Lusztig theory}
In a groundbreaking paper \cite{KL}, Kazhdan and Lusztig laid a basis for an approach to the representation theory of Hecke algebras. Since then this approach has been significantly developed, and is called \emph{Kazhdan-Lusztig theory}. Of particular importance in it is the notion of \emph{cells}. 

Briefly, the definition is as follows. The Hecke algebra is associated with a Coxeter group $W$. Kazhdan and Lusztig define a pre-order on
elements of $W$ denoted $\leq_L$. Some pairs $v,w$ of elements of $W$ satisfy both $v \leq_L w$ and $w \leq_L v$, in which case we say that they are left-equivalent, denoted $v \sim_L w$. Similarly one can define right equivalence $\sim_R$. The respective equivalence classes are called the \emph{left cells} and the \emph{right cells}. Another important notion turns out to be that of left-right, or two-sided, equivalence $\sim_{LR}$. Two elements of $W$ are left-right equivalent if they are connected by a series of moves where every move is either a left equivalence or a right equivalence. The corresponding equivalence classes are called \emph{two-sided cells}.

\subsection{Type $A$}
In type $A$, i.e., when $W$ is a symmetric group, the Kazhdan-Lusztig cell structure corresponds to something very familiar to combinatorialists, the \emph{Robinson-Schensted correspondence}. It is a bijective correspondence between elements of the symmetric group and pairs of {\it {standard Young tableaux}} of the same shape. It is well known (see \cite{BV,KL,GM,A}) that
\begin{itemize}
 \item two permutations lie in the same left cell if and only if they have the same \emph{recording tableau} $Q$;
 \item two permutations lie in the same right cell if and only if they have the same \emph{insertion tableau} $P$;
 \item two permutations lie in the same two-sided cell if and only if their insertion and recording tableaux have the same shape $\lambda$; equivalently, this happens when certain posets naturally associated with the permutations have the same \emph{Greene-Kleitman invariants} \cite{GK}.
\end{itemize}
We refer the reader to \cite{fulton_yt} for the classical formulation of Robinson-Schensted correspondence via an insertion algorithm. We will soon describe a less widely known construction, which can also be found in \cite{fulton_yt}, and which we generalize in this paper.

\subsection{Affine type $A$}
For $W$ of affine type $A$, i.e. an affine symmetric group, Shi has shown in \cite{shibook} that the left Kazhdan-Lusztig cells correspond to \emph{tabloids}. These are equivalence classes of fillings of Young diagrams with the integers (or residue classes), up to permuting elements within rows. The shape of these tabloids determines the two-sided cell. Furthermore, in \cite{Shi} Shi gave an algorithm for constructing a tabloid $P(w)$ out of an affine permutation $w \in W$. We refer to this algorithm as \emph{Shi's algorithm}. We will describe it in Section \ref{sec:Shi algorithm and KL cells}. We also refer to the affine Robinson-Schensted correspondence between left cells and tabloids as the \emph{Shi correspondence}.

A major difference between the finite and affine types is that in the latter case the map $$w \in W \mapsto (\text{insertion tabloid } P(w), \text{ recording tabloid } Q(w))$$ is {\it {not}} an injection. This is to be expected of course, as there are only finitely many pairs $(P,Q)$, while $W$ is infinite.

One can get an actual bijection if one adds a third piece to the data: weights. Namely, there exists a bijection 
$$w \in W \mapsto (\text{insertion tabloid } P(w), \text{ recording tabloid } Q(w), \text{ dominant weight } \rho).$$
This was essentially known since Lusztig's conjecture in \cite{lusztig}. A proof and an explicit description of this correspondence was given by Honeywill \cite{honey}, relying on some key ideas of Xi \cite{xibook}.

\subsection{Goals}

\subsubsection*{Goal A} \emph{Find a description of the Shi correspondence which would generalize in a natural way some known construction of the Robinson-Schensted correspondence, would yield insight into the meaning of associating a tabloid to a cell, and would be convenient in its applications to the further study of Kazhdan-Lusztig cells in affine type $A$.}

Shi's algorithm is an involved process consisting of several distinct sub-algorithms which need to be applied in specific order. It makes it challenging to develop any direct intuitive connection between the permutation and the resulting tabloids. An alternative algorithm explained in Shi's paper (the one we will describe) is more natural in that it, essentially, consists of a sequence of Knuth moves getting a permutation to a form where the tabloid can be read off. This version however requires pre-computing the Greene-Kleitman invariants of the associated poset.

It is our impression that the \emph{Matrix-Ball Construction} is the most suitable candidate to generalize. The standard version first appeared in \cite{viennot_shadow}, however we will use the terminology from the semistandard generalization in \cite{fulton_yt}. Thus, in this paper we introduce the \emph{Affine Matrix-Ball Construction}, or \emph{AMBC}. AMBC has been implemented in Java \cite{program}; we strongly suggest the program to the reader for additional examples. 

Surprisingly, as a byproduct of our analysis we get a realization of the Shi correspondence in terms of the usual Robinson-Schensted insertion, see Section \ref{sec:noproofs asymptotic rsk} for details. A reader who wishes to learn \emph{a} simple combinatorial way to realize the Shi correspondence can jump directly to this section. This realization implicitly appeared in the work of Pak \cite{pak}.

\subsubsection*{Goal B}
\emph{Describe an extension of Shi's correspondence to a bijection in terms of AMBC.}

It turns out that weights $\rho$ can be given a natural interpretation in terms of AMBC.  We also explain how to complete the bijection by describing the inverse map $(P,Q,\rho) \mapsto w$. We leave the question of the exact relationship between our weights and those of Honeywill for future research.

Throughout the paper we mostly work with a slightly larger group, namely the extended affine symmetric group. In Section \ref{sec:nonext} we show that by imposing a simple restriction $\sum_i \rho_i = 0$ on weights one recovers all of the theory for the non-extended case.

\subsubsection*{Goal C}
\emph{The inverse map $(P,Q,\rho) \mapsto w$ is defined for any (not necessarily dominant) weight; describe its fibers.}

 It turns out that the fibers are orbits of a certain Weyl group acting on the lattice of weights. This explains why orbit representatives can be chosen to be dominant weights. We are unaware of anything similar to this theorem in the existing literature.

\subsection{Paper parts}
The paper is divided into two parts. The first part is relatively short and is written to explain the main results in an accessible form. The second part contains a detailed development of the theory, including all the proofs.

\subsection{Notational preliminaries}
\label{sec:notational preliminaries}
Let $n$ be a positive integer. Let $[n] := \{1,\dots, n\}$. For each $i\in\Z$ denote by $\ol{i}$ the residue class $i+n\Z$. Let $[\ol{n}]:= \{\ol{1},\dots, \ol{n}\}$.

We will encounter several groups. Let $W$ be the symmetric group of type $A_{n-1}$, i.e. $S_n$. Let $\widetilde{W}$ be the extended affine symmetric group of type $\widetilde{A}_{n-1}$; it consists of bijections $w:\Z\to\Z$ such that 
$$w(i+n) = w(i)+ n.$$
The elements of $\widetilde{W}$ are called \emph{extended affine permutations}. Since these will be common objects we deal with, we will  shorten that to just permutations.
Let $\ol{W}$ be the affine symmetric group of type $\widetilde{A}_{n-1}$; it consists of $w\in\widetilde{W}$ such that 
$$\sum_{i=1}^n w(i) = \frac{n(n+1)}{2}.$$ 

A \emph{partial (extended affine) permutation} is a pair $(U, w)$ where $U\subseteq\Z$ has the property that $(x\in U) \Leftrightarrow (x+n\in U)$ and $w:U\to\Z$ is an injection such that $w(i+n) = w(i)+ n$. We will suppress the explicit mention of the subset $U$ in the notation and just refer to the partial permutation $w$. Any permutation is viewed as a partial permutation with $U = \Z$. 

A permutation is determined by its values on $1,\dots, n$. The \emph{window notation} for a permutation $w$ is $[a_1,\dots, a_n]$ where $w(1)=a_1,\dots, w(n)=a_n$.	A partial permutation $w$ is also determined by its values on $1,\dots, n$, except it may not be defined on some of them. The \emph{window notation} for a partial permutation $w$ is $[a_1,\dots, a_n]$ where $a_i = w(i)$ if $w$ is defined on $i$ and $a_i=\varnothing$ otherwise.

We usually think of permutations in terms of pictures such as the one in Figure \ref{fig:proper numbering}. More precisely, on the plane we draw an infinite matrix; the rows are labeled by $\Z$, increasing downward, and the columns are labeled by $\Z$, increasing to the right. If $w(i)=j$ then we place a ball in the $i$-th row and $j$-th column. To distinguish the $0$-th row, we put a solid red line between the $0$-th and $1$-st rows, and similarly for columns. We also put dashed red lines every $n$ rows and columns. 

The cells of the matrix will be referred to by their matrix coordinates, e.g. the cell $(1,4)$ in Figure \ref{fig:proper numbering} contains a ball. The balls of a partial permutation will also be referred to by their matrix coordinates. For a partial permutation $w$, we denote by $\B_w\subsetneq\Z\times\Z$ the collection of balls of $w$. 

For an integer $k$ and a ball $b=(i,j)$ we will refer to the ball $b' = (i+kn, j+kn)$ as the \emph{$k(n,n)$-translate of $b$} and denote this by $b' = b+k(n,n)$. Two balls $b$ and $b'$ are $(n,n)$-translates if for some $k$ one is a $k(n,n)$-translate of the other.

We will assign numbers to balls of permutations as well as to other cells of the matrix. For a partial permutation $w$, a \emph{numbering of $w$} is a function $d:\B_w\to\Z$. A numbering $d$ of $w$ is \emph{semi-periodic with period $m$} if for any $b\in\B_w$ we have $d(b+(n,n)) = d(b) + m$. When referring to a numbering in pictures, we will write $d(b)$ inside the ball $b$ as done in Figure \ref{fig:proper numbering}, where we show a semi-periodic numbering of period 3.

We use compass directions (north, east, etc.) inside the matrix with north corresponding to smaller row numbers and east corresponding to larger column numbers. By a composite direction (e.g. northeast) we mean north and east. The relations are weak by default: a cell $(i,j)$ is \emph{southwest of $(i',j')$} if $i\geqslant i'$and $j\leqslant j'$. Adding the modifier ``directly'' constrains one of the two coordinates: a cell $(i,j)$ is \emph{directly south} of $(i',j')$ if $i\geqslant i'$ and $j=j'$. Directions define partial orders on $\Z\times\Z$: we say $(i,j)\leqslant_{SW} (i',j')$ if $(i,j)$ is southwest of $(i',j')$. 

Sequences of cells going southeast and northwest will be particularly important to us. A $\emph{path}$ is a sequence $(b_0,\dots, b_k)$ of cells such that for each $i$, $b_{i+1}$ is strictly northwest (meaning both strictly north and strictly west) of $b_i$. A $\emph{reverse path}$ is a sequence $(b_0,\dots, b_k)$ of cells such that for each $i$, $b_{i+1}$ is strictly southeast of $b_i$. In both cases, the number $k$ is referred to as the \emph{length} or \emph{number of steps} of the (reverse) path.

\subsection{Acknowledgments}
The authors are grateful to Joel Lewis and an anonymous referee for reading the paper and providing valuable feedback, and to Darij Grinberg for helpful comments.

\part{Statements.}

\section{Matrix-Ball Construction}

In this section we briefly review the algorithm called the Matrix-Ball Construction (MBC) and its inverse. MBC is an implementation of the Robinson-Schensted correspondence which was originally described by Viennot \cite{viennot_shadow} in terms of shadows. It was later generalized by Fulton \cite{fulton_yt} to the setting of integer matrices. We will use the terminology from the second reference. 

In these algorithms (MBC, its inverse, AMBC, and its inverse) we will consider collections of cells of the following form:

\begin{defn}
\label{def:zig-zag}
A \emph{reverse zig-zag} is a non-empty sequence $(c_1, c_2, \ldots, c_k)$ of cells such that both of the following hold:
\begin{enumerate}
\item for all $1\leqslant i< k$, $c_{i+1}$ is adjacent to and either directly north or directly east of $c_i$, and
\item if $k\geqslant 2$, then $c_2$ is directly east of $c_1$ and $c_k$ is directly north of $c_{k-1}$.
\end{enumerate}
See Figure \ref{fig:pieces_of_yd} for an example. Similarly, a \emph{forward zig-zag} is a collection of cells that satisfies the first of the above conditions, starts with a step north, and ends with a step east.
\end{defn}

The forward/reverse terminology is motivated by the fact that in Section \ref{sec:weights} we will use forward zig-zags to create paths and reverse zig-zags to create reverse paths. Most of the time it will be clear from the context whether we are talking about forward zig-zags or reverse zig-zags, so we will drop the specification and just call them zig-zags.

\begin{defn}
Given a reverse zig-zag $Z = (c_1, c_2, \ldots, c_k)$, we say that
\begin{enumerate}
\item the \emph{back corner-post of $Z$} is the cell in the same column as $c_1$ and the same row as $c_k$,
\item the \emph{inner corner-posts of $Z$} are the cells of $Z$ such that no cell directly north or directly west of them is in $Z$, and
\item if $k \geqslant 2$, the \emph{outer corner-posts of $Z$} are the cells of $Z$ such that no cell directly south or directly east of them is in $Z$; if $k = 1$ then there are no outer corner-posts of $Z$.
\end{enumerate}
The definitions for forward zig-zags are obtained by reflection in the anti-diagonal (i.e. interchange ``north'' with ``east'', and ``south'' with ``west'').
\end{defn}
Note that while the inner and outer corner-posts are always parts of the zig-zag, the back corner-post is only part of the zig-zag if the zig-zag consists of just one cell.

\begin{figure}
\centering
\resizebox{.8\textwidth}{!}{\input{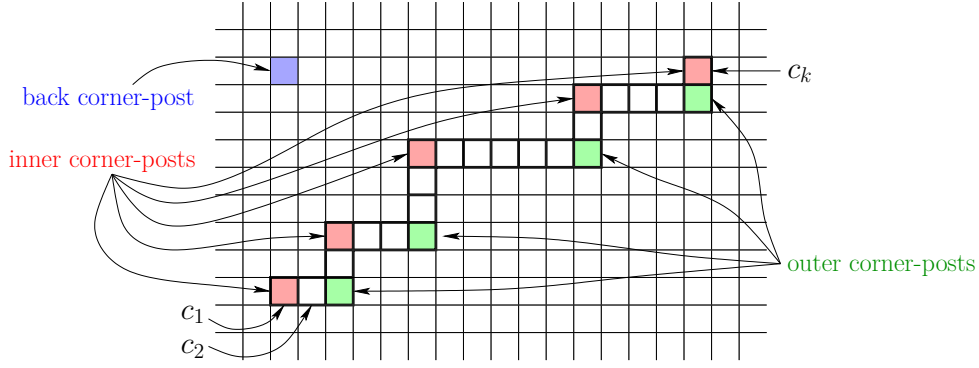}}
\caption{A reverse zig-zag and various corner-posts.}
\label{fig:pieces_of_yd}
\end{figure}

The Robinson-Schensted correspondence sends a non-affine permutation $w$ to a pair of standard Young tableaux $(P(w), Q(w))$ of the same shape. Before describing the algorithm we will discuss a way to produce a partial non-affine permutation $\fw{w}$ from a partial non-affine permutation $w$. View $w$ as an $n\times n$ matrix with balls as described in the introduction. Make a numbering $d_w$ of $w$ iteratively as follows. For any ball $b$ with no balls strictly northwest of it, let $d_w(b) = 1$. For any ball $b$ with only balls numbered $1$ strictly northwest of it, let $d_w(b) = 2$.  Continue until all balls are numbered (see Figure \ref{fig:mbc}). For each $i$, consider the unique zig-zag $Z_i$ whose inner corner-posts are the balls labeled $i$ (denoted by green lines in Figure \ref{fig:mbc}). For each $i$, let $(a_{w,i}, b_{w,i})$ be the back corner-post of $Z_i$ (denoted by $*$'s in the figure). Let $\fw{w}$ be the partial non-affine permutation whose balls are the outer corner-posts of all the $Z_i$'s (the green balls in the figure). Now we can describe the algorithm:

\begin{itemize}
 \item Input $w\in W$.
 \item Output: a pair $(P(w),Q(w))$ of standard tableaux of the same shape.
 \item Initialize $(P,Q)$ to $(\varnothing,\varnothing)$.
 \item Repeat until $w$ is the empty partial permutation:
\begin{itemize}
 \item Add a row $(b_{w,1},b_{w,2},\dots)$ to $P$ and a row $(a_{w,1},a_{w,2},\dots)$ to $Q$.
 \item Reset $w$ to $\fw{w}$.
\end{itemize}
\item Set $P(w) = P$ and $Q(w) = Q$.
\end{itemize}

\begin{ex}
Consider the non-affine permutation $w=[5,6,1,3,4,2]$. The steps of MBC are shown in Figure \ref{fig:mbc}. The columns of the $*$'s in the first picture are $1,2,4$ so these numbers form the first row of $P(w)$, while the rows are $1,2,5$ so these numbers form the first row of $Q(w)$. The remaining two rows are handled similarly, yielding
$$P(w) = \tableau[sY]{1,2,4\\3,6\\5}\qquad Q(w)=\tableau[sY]{1,2,5\\3,4\\6}$$
\end{ex}

\begin{figure}
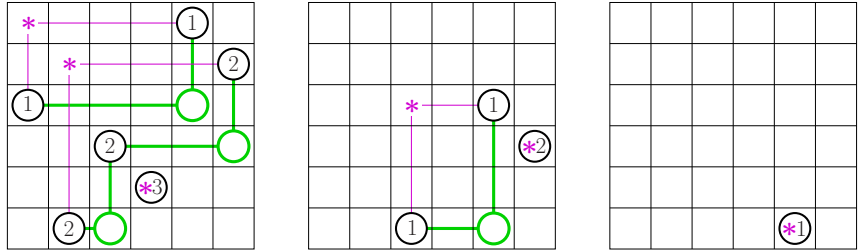

\centering\resizebox{.2\textwidth}{!}{\input{figures/mbc.pspdftex}}\qquad\resizebox{.2\textwidth}{!}{\input{figures/mbc2.pspdftex}}\qquad\resizebox{.2\textwidth}{!}{\input{figures/mbc3.pspdftex}}
\caption{Steps of MBC for the non-affine permutation $[5,6,1,3,4,2]$.}
\label{fig:mbc}
\end{figure}

It turns out that this procedure always produces a pair of standard Young tableaux of the same shape; moreover this is a bijection between $W$ and the collection of pairs of Young tableaux of the same shape and of size $n$. Next we want to describe an algorithm which  inverts the map $w\mapsto (P(w), Q(w))$. Just as with MBC, we start by describing an essential construction in the step.

Suppose $w$ is a partial permutation and $p=(p_1,\dots,p_k)$, $q=(q_1,\dots, q_k)$ are two sequences of integers of the same length. In the examples in Figure \ref{fig:invmbc}, the balls of $w$ are black, and each cell $(p_i, q_i)$ is labeled by the magenta number $i$. Construct a numbering of $w$ iteratively as follows. For a ball $b$ with no balls southeast of it, number $b$ with the largest integer $i$ such that $(p_i, q_i)$ is northwest of $b$. Once all balls southeast of a given ball $b$ have been numbered, number $b$ with the largest integer $i$ such that $(p_i, q_i)$ is northwest of $b$ and all balls southeast of $b$ have numbers strictly larger than $i$. The resulting numbering is shown in blue in Figure \ref{fig:invmbc}. Now for each $i$ form the unique zig-zag $Z_i$ whose back corner-post is $(p_i,q_i)$ and whose outer corner-posts are the balls numbered $i$. Let $\bk{p,q}{w}$ be the partial permutation whose balls are the inner corner-posts of all the zig-zags $Z_i$. Now we can describe the inverse MBC algorithm:

\begin{itemize}
 \item Input a pair $(P,Q)$ of standard Young tableaux of the same shape.
 \item Output: $w\in W$.
 \item Initialize $w$ to the empty partial permutation.
 \item Repeat until $(P,Q) = (\varnothing,\varnothing)$:
\begin{itemize}
 \item Set $p$ to be the last row of $P$ and $q$ to be the last row of $Q$. Remove the last rows from the tableaux.
 \item Reset $w$ to $\bk{p,q}{w}$.
\end{itemize}
\end{itemize}

\begin{ex}
\label{ex:invmbc}
The steps of the algorithm for
$$P = \tableau[sY]{1,3,5\\2,6\\4}\qquad Q=\tableau[sY]{1,4,6\\2,5\\3}$$
are shown in Figure \ref{fig:invmbc}. The result is the non-affine permutation $[4,2,1,6,3,5]$.
\end{ex}
\begin{figure}
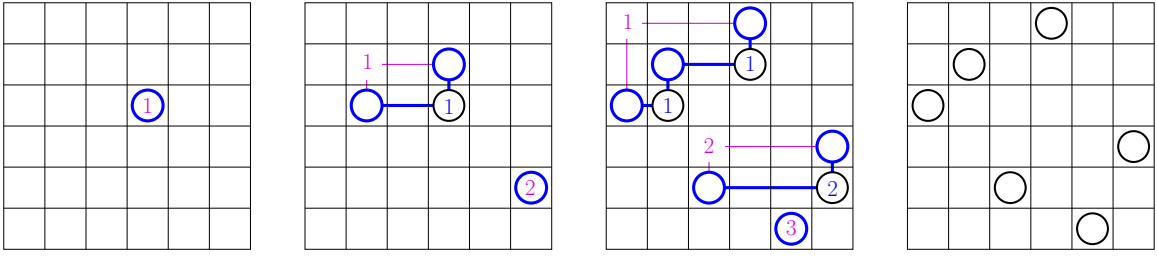

\centering\resizebox{.2\textwidth}{!}{\input{figures/imbc.pspdftex}}\qquad\resizebox{.2\textwidth}{!}{\input{figures/imbc2.pspdftex}}\qquad\resizebox{.2\textwidth}{!}{\input{figures/imbc3.pspdftex}}\qquad\resizebox{.2\textwidth}{!}{\input{figures/imbc4.pspdftex}}
\caption{Steps of inverse MBC in Example \ref{ex:invmbc}.}
\label{fig:invmbc}
\end{figure}

\section{Affine Matrix-Ball Construction}

\subsection{Proper numberings}

In the step of MBC one began by numbering the balls of the permutation according to a certain rule. Attempting to directly apply the same rule to the balls of an affine permutation faces the problem that one does not know how to start. As will be described in Proposition \ref{prop:irrelevant which proper}, to some extent it does not matter. In this section we introduce a collection of numberings, called proper numberings, which work well to produce the two tabloids. In Section \ref{sec:noproofs channels} we will choose a particular proper numbering which will be used in the step of AMBC.

\begin{defn}
Let $w\in\widetilde{W}$ be a permutation. A function $d:\B_w\to\Z$ is a \emph{proper numbering} if it is
\begin{itemize}
\item Monotone: for any $b, b'\in \B_w$, if $b$ lies strictly northwest of $b'$ then $d(b) < d(b')$, and
\item Continuous: for any $b'\in \B_w$, there exists $b$ northwest of it with $d(b) = d(b')-1$.
\end{itemize}
\end{defn}
An example of a proper numbering is given in Figure \ref{fig:proper numbering}. Note that if we start with a proper numbering and increase the number of every ball by the same amount, then the resulting numbering is still proper. 

\begin{figure}
\centering\resizebox{.6\textwidth}{!}{\input{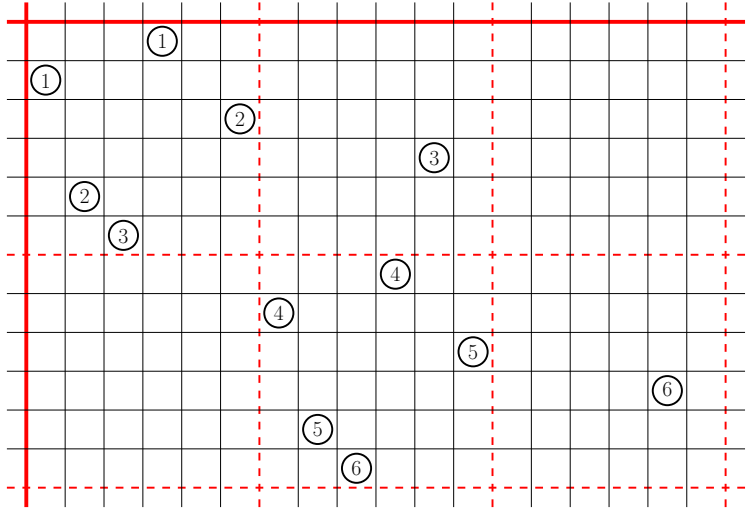}}
	\caption{A proper numbering for the permutation $[4,1,6,11,2,3]$. It is semi-periodic with period 3.}
\label{fig:proper numbering}
\end{figure}

It is not a priori obvious that proper numberings exist. We will describe a construction of such a numbering in Section \ref{sec:noproofs channels}, and undertake a detailed study of their structure theory in Section \ref{sec: proper numberings structure}. 

\subsection{Shi poset}

Given a permutation $w\in\widetilde{W}$, Shi defined a (labeled) poset on $[n]$. As we shall see in Section \ref{sec:Shi algorithm and KL cells}, the Greene-Kleitman invariants (see \cite{GK} and a survey \cite{BF}) of this poset have significance in the Kazhdan-Lusztig theory of the affine symmetric group. 

\begin{defn}
 The (labeled) \emph{Shi poset} $P_w$ associated with $w\in\widetilde{W}$ is the poset on $[n]$ with $i\leqslant_S j$ if 
$$i > j\text{, and } w(i) < w(j),$$
or if
$$w(j) > w(i) + n.$$ The element $i$ is labeled by the residue class $\ol{w(i)}$.
\end{defn}

\begin{ex}
For $w=[2, 8, 1, 14, 7, 16, 15, 0, 3, 9]$, the Shi poset is shown on the left in Figure \ref{fig:shi poset}.
\end{ex}

\begin{figure}
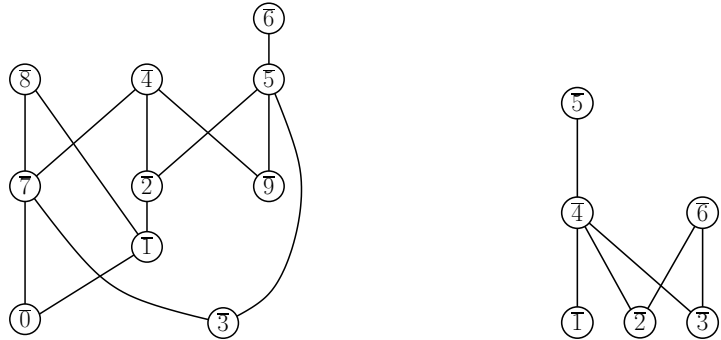

\centering\resizebox{!}{.2\textheight}{\input{figures/shi_poset.pspdftex}}\hspace{.2\textwidth} 
     \resizebox{!}{.15\textheight}{\input{figures/shi_poset_2.pspdftex}} 
\caption{The Hasse diagrams for the Shi posets associated with $[2, 8, 1, 14, 7, 16, 15, 0, 3, 9]$ and $[4,1,6,11,2,3]$.}
\label{fig:shi poset}
\end{figure}

The Shi poset is essentially the natural interpretation of the relation $\leqslant_{SW}$ on translate classes of balls of $w$. The notion extends naturally to partial permutations. We can now state the first major result about the structure of proper numberings.

\begin{prop}
\label{prop: period of proper numbering}
Any proper numbering of a partial permutation is semi-periodic with period $m$ equal to the width (i.e. maximal size of an antichain) of the Shi poset.
\end{prop}

The proof of this statement is found on page \pageref{pf:period of proper numbering}. The numbering in Figure \ref{fig:proper numbering} is semi-periodic with period 3. The Hasse diagram for the corresponding Shi poset is shown on the right of Figure \ref{fig:shi poset}.

\subsection{Channels and channel numberings}
\label{sec:noproofs channels}

In this section we construct the most important examples of proper numberings.

\begin{defn}
Suppose we have a collection $C$ of cells which is invariant under translation by $(n,n)$ and forms a chain in the partial ordering $\leqslant_{SE}$. Then the \emph{density} of $C$ is the number of distinct translation classes in $C$.
\end{defn}
In Figure \ref{fig:proper numbering}, the collection consisting of balls $(2,1), (5,2), (6,3)$ and their translates has density $3$.

\begin{defn}
Suppose $w$ is a partial permutation. Then $C\subseteq \B_w$ is a \emph{channel} if all of the following hold:
\begin{enumerate}
\item $C$ is invariant under translation by $(n,n)$,
\item $C$ forms a chain in the partial ordering $\leqslant_{SE}$, and
\item the density of $C$ is maximal among all subsets of $\B_w$ satisfying the first two conditions. 
\end{enumerate}
\end{defn}
In Figure \ref{fig:proper numbering}, the collection consisting of balls $(2,1), (5,2), (6,3)$ and their translates is a channel since no higher density can be achieved within $\B_w$. On the other hand, balls $(1,4)$ and $(3,6)$ are not part of any channel, since no translation-invariant chain of balls with density $3$ passes through both.



The definition of a channel may be reformulated in terms of the Shi poset. 
\begin{defn}
For a partial permutation $w$, the \emph{projection map} $\varphi_w:\B_w\to P_w$ is the map sending $(i, w(i))$ to the representative of $\ol{i}$ in $[n]$.
\end{defn}
It follows easily that channels are precisely the preimages under the projection map of longest antichains of the Shi poset.

Let $\C_w$ be the set of all channels for a given partial permutation $w$. For a channel $C\in \C_w$ we can consider a proper numbering $\tilde{d}:C\to\Z$, namely number balls in $C$ by consecutive integers from northwest to southeast. Such a numbering is shown in red in Figure \ref{fig:channel numbering} (it is unique up to an overall shift).

\begin{figure}
\centering
\resizebox{.8\textwidth}{!}{\input{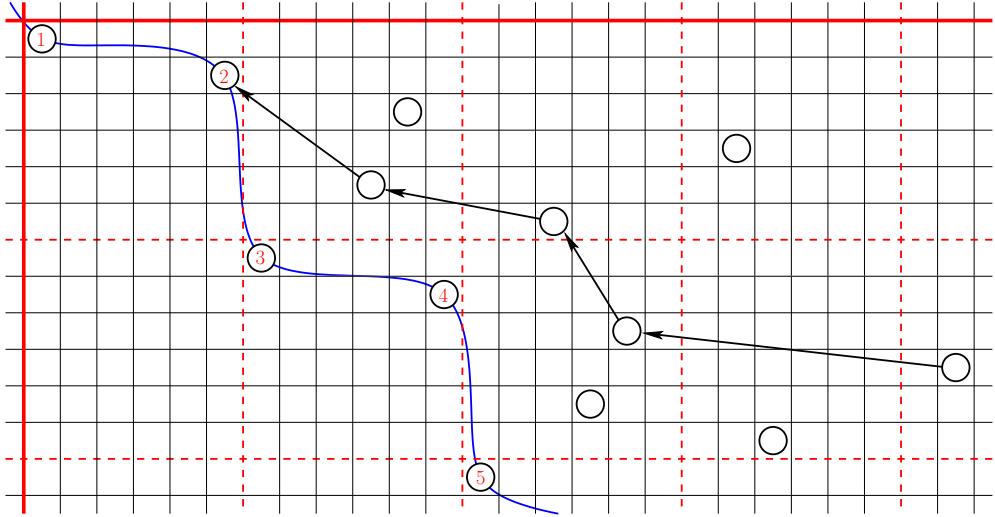}}
\caption{A numbered channel of $[1,6,11,20,10,15]$ and a path to it. The path has $4$ steps and leads to a ball with number $2$. }
\label{fig:channel numbering}
\end{figure}

\begin{defn}
\label{def:channel numbering}
Suppose $C\in\C_w$ and $\tilde{d}$ is a proper numbering of $C$. Define a \emph{channel numbering} of $w$ with respect to $C$ by
$$d_w^C(b) := \sup_{(b_0, b_1, \dots, b_k)} \tilde{d}(b_k) + k,$$
where the supremum is taken over all paths from $b$ to $C$. Sometimes we will omit $w$ from the notation when it is unambiguous and write just $d^C$.
\end{defn}

The numbering $d_w^C$ inherits an arbitrary shift freedom from $\tilde{d}$. Most of the time we do not care what the shift is, and we will explicitly specify the shift in the cases that it is relevant (such as when we consider numberings with respect to different channels and we want to compare them). 

A priori it seems that sometimes $d_w^C(b)$ could be infinite; this is not the case:
\begin{prop}
\label{prop:channel numbering defined}
Suppose $w$ is a partial permutation, $C\in\C_w$ and $\tilde{d}$ is a proper numbering of $C$.  Then for any ball $b$, the set
$$\{\tilde{d}(b_k) + k : (b_0, b_1, \dots, b_k) \text{ is a path from $b$ to $C$}\}$$
is bounded above.
\end{prop} 
The proof of this statement is found on page \pageref{pf:channel numbering defined}.
Once we know that the channel numbering is well defined, the following is obvious. 
\begin{prop}
If $w$ is a partial permutation and $C\in\C_w$, then $d_w^C$ is proper.
\end{prop}

A proper numbering is determined by its value on the channels:
\begin{prop}
\label{prop:proper determined by channels}
Suppose $w$ is a partial permutation and $d, d'$ are two proper numberings. Suppose for all channels $C$ and all balls $b\in C$ we have $d(b) = d'(b)$. Then for all $b\in\B_w$ we have $d(b) = d'(b)$. 
\end{prop}
The proof is found on page \pageref{pf:proper determined by channels}. Though it is not necessary for our main goals in the paper, after the proof we describe precisely which numberings on the channels can be extended to a proper numbering of $w$.

Now we describe how to choose certain distinguished channels which are present for any $w$. Consider the following partial order on $\C_w$, which we will refer to as the \emph{southwest partial ordering}. 

\begin{defn}
\label{def:channel sw ordering}
A channel $C_1$ \emph{is southwest of} a channel $C_2$ if for every element of $C_1$ there exists an element of $C_2$ weakly northeast of it. 
\end{defn}

The definition has an apparent asymmetry. It is easy to see that an equivalent definition would be: $C_1$ is southwest of $C_2$ if each element of $C_2$ has at least one element of $C_1$ southwest of it.

\begin{prop}
\label{prop:sw channel exists}
The southwest partial ordering on $\C_w$ has a least element and a greatest element.
\end{prop}
The proof is found on page \pageref{pf:sw channel exists}; in fact we prove that for any two channels, the set of southwest balls of their union is a channel and the set of northeast balls of their union is a channel. Thus there exists a southwest channel and a northeast channel. 

In AMBC we will be using the southwest channel numbering of balls. Since the southwest channel numbering will play a key role in the paper, we introduce special notation for it.
\begin{defn}
Suppose $w$ is a partial permutation. We write $d_w^{SW}$ for the channel numbering with respect to the southwest channel.
\end{defn}

Next we introduce a notion of distance between channels. To make sure that our definition is consistent we need the following lemma:
\begin{lemma}
\label{lem:proper numbers channels}
Any proper numbering of $w$ restricted to any channel $C$ gives a proper numbering of $C$. 
\end{lemma}
The proof is found on page \pageref{pf:proper numbers channels}.

\begin{defn}
\label{def:distance between channels}
Let $C_1, C_2\in \C_w$ be two channels. Let choose the shifts of $d^{C_1}$ and $d^{C_2}$ so the two numberings coincide on $C_1$. Define the \emph{distance between $C_1$ and $C_2$} by 
$$\h(C_1, C_2) = \abs{d^{C_2}(b) - d^{C_1}(b)},$$
where $b$ is any ball of $C_2$ (the quantity is independent of $b$ by the above lemma). 
\end{defn}

\begin{prop}
\label{prop:h is a pseudometirc}
For any partial permutation $w$, the function $\h$ defined above is a pseudometric on $C_w$.
\end{prop}
The proof is found on page \pageref{pf:h is a pseudometirc}. As a pseudometric, $\h$ naturally partitions $C_w$ into equivalence classes.
\begin{defn}
A \emph{river} is a maximal collection of channels such that the distance between any pair of them is $0$.
\end{defn}

\begin{ex}
The channels of $[4, 1, 11, 6,14,9,3,7,17]$ are shown in Figure \ref{fig:river} (any channel has a blue curve going through it). This permutation has two rivers. 
\end{ex}

\begin{figure}
\centering
\resizebox{.7\textwidth}{!}{\input{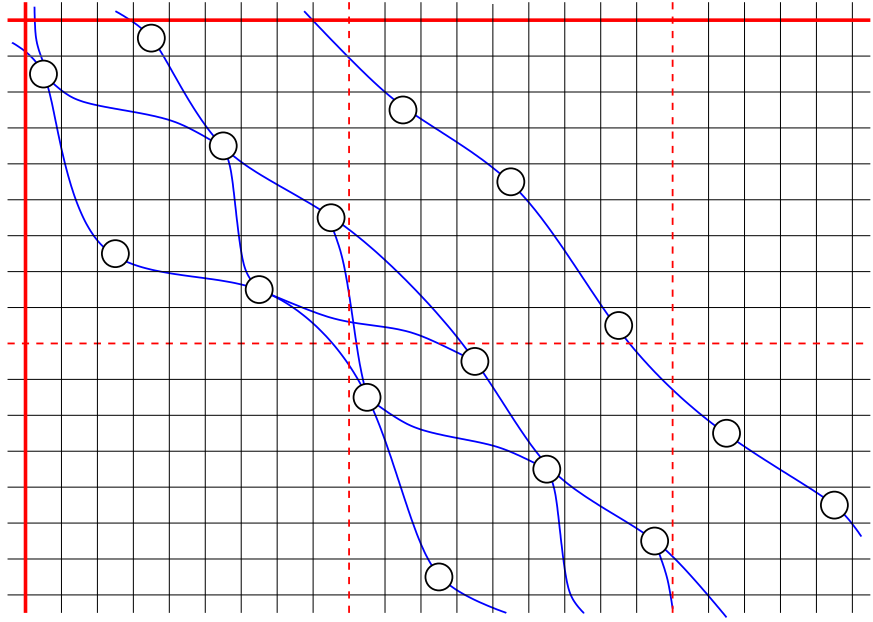}}
\caption{The channels of $[4, 1, 11, 6,14,9,3,7,17]$. This permutation has two rivers.}
\label{fig:river}
\end{figure}

\subsection{Streams}

 In this section we introduce the concept of streams. In MBC, the data necessary to invert a step of the algorithm consisted of two ordered subsets of $[n]$ of the same size (the rows of the insertion and recording tableaux). The streams encode the additional data necessary to invert a step of AMBC.

\begin{defn}
\label{defn:class}
\begin{enumerate}
\item A \emph{stream} is a collection of cells which is invariant under translation by $(n,n)$ and which forms a chain in the partial ordering $\leqslant_{SE}$.
\item Let $A = \{\ol{a_1}, \ldots, \ol{a_k}\}$ and $B = \{\ol{b_1}, \ldots, \ol{b_k}\}$ be two subsets of $[\ol{n}]$ of the same size. Let $\st(A,B)$ be the collection of all streams $S$ satisfying $A = \{\ol{i}:(i,j)\in S\}$ and $B = \{\ol{j}:(i,j)\in S\}$. An element of $\st(A,B)$ will be called an \emph{$(A,B)$-stream}.
\item The \emph{flow} of an $(A,B)$-stream is the number $\abs{A} = \abs{B}$.
\item For a stream $S$, the \emph{class} of $S$ is the collection $\st(A,B)$ with $S\in\st(A,B)$.
\end{enumerate}
\end{defn}

There are infinitely many $(A,B)$-streams. We classify them in the following way.

\begin{lemma}
 There is a unique stream $\str{0}{A,B}$ such that there are exactly $k = |A| = |B|$ elements $(i,j)$ of $\str{0}{A,B}$ such that $1 \leq i,j \leq n$.
\end{lemma}

\begin{proof}
Let $(a_1,\dots, a_k)\in [n]^k$ be the representatives of the classes in $A$ in increasing order; similarly for $B$. The unique $(A,B)$-stream satisfying the condition of the lemma is 
$$\{(a_r + tn, b_r + tn) \mid 1 \leq r \leq k, t \in \mathbb Z\}.$$
\end{proof}

Now, let 
$$\st_0(A,B) = \bigg\{\ldots >_{SE} (i_{-1},j_{-1}) >_{SE} (i_{0},j_{0}) >_{SE} (i_{1},j_{1}) >_{SE} \ldots \bigg\}.$$

\begin{prop}
 For each $r \in \Z$, let $$\st_r(A,B) = \{(i_t, j_{t+r}) \mid t \in\Z\}.$$ Then each $\st_r(A,B)$ is an $(A,B)$-stream, and all $(A,B)$-streams arise this way.
\end{prop}

\begin{proof}
Let $S\in\st(A,B)$. Let $(i_0,j) \in S$, then clearly $j = j_r$ for some $r \in\Z$. Then we claim $(i_1, j_{1+r}) \in S$. Indeed, if not, then the element of $S$ in row $i_1$ is weakly east of $(i_1,j_{2+r})$, which implies the element of $S$ in row $i_2$ is weakly east of $(i_1, j_{3+r})$, etc. This means however that $S$ is missing an element in column $j_{1+r}$, which is impossible.
\end{proof}

\begin{defn}
\label{def:altitude}
The \emph{altitude} of the stream $\st_r(A,B)$ is the number $r$. For a stream $S$, we denote its altitude by $a(S)$.
\end{defn}	

\begin{ex}
 Let $n=6$, $A = \{\ol{1},\ol{3},\ol{6}\}$ and $B = \{\ol{2},\ol{4},\ol{5}\}$. Then the streams
$\st_{-1}(A,B)$, $\st_{0}(A,B)$ and $\st_{1}(A,B)$ are shown in
Figure \ref{fig:streams}.
\end{ex}

\begin{figure}
\centering
\resizebox{.4\textwidth}{!}{\input{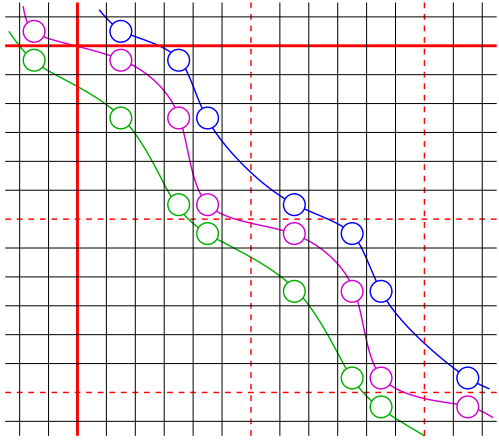}}
\caption{Streams of altitudes $-1, 0$, and $1$ for $A = \{\ol{1},\ol{3},\ol{6}\}$ and $B = \{\ol{2},\ol{4},\ol{5}\}$.}
\label{fig:streams}
\end{figure}


\begin{defn}
The \emph{defining data} of a stream $S$ is the triple $(A,B,r)$ such that $S = \mathfrak{st}_r(A,B)$.
\end{defn}

\subsection{The map}

Let $\lambda = (\lambda_1, \ldots, \lambda_{\ell})$ be a partition of size $\sum_i \lambda_i \leq n$. A \emph{tabloid} $P$ of shape $\lambda$ is an equivalence class of fillings of the Young diagram of shape $\lambda$ with elements of $[\ol{n}]$ under identification of fillings that differ by reordering  elements within rows. We think of the $i$-th row of a tabloid $P$ as a set $P_i\subseteq[\ol{n}]$. Let $\Omega$ be the set of all triples $(P,Q,\rho)$, where $P = (P_1,\ldots, P_\ell)$ and $Q = (Q_1,\ldots, Q_\ell)$ are tabloids of the same shape $\lambda$ of length $\ell$ and size $n$ filled with distinct residue classes, and $\rho = (\rho_1, \ldots, \rho_{\ell}) \in \Z^{\ell}$ (a \emph{weight} of size $\ell$). Note that the $i$-th row in an element of $\Omega$ describes a stream, namely $\mathfrak{st}_{\rho_i}(Q_i, P_i)$.

\begin{defn}
Suppose $w$ is a partial permutation and $d:\B_w\to\Z$ a proper numbering. For each $i\in\Z$ let $Z_i$ be the unique zig-zag whose inner corner-posts are the balls labeled $i$. Define $\fwC{d}{w}$ to be the permutation whose balls are at the outer corner-posts of the zig-zags $Z_i$. Define $\str{d}{w}$ to be the stream whose cells are the back corner-posts of the zig-zags $Z_i$ (it is obvious from the definitions that the back corner-posts do, in fact, form a stream). 

If $C$ is a channel, we write $\fwC{C}{w}$ for  $\fwC{d_w^C}{w}$ and $\str{C}{w}$ for $\str{d_w^C}{w}$. If $C$ is the southwest channel then we write $\fw{w}$ for $\fwC{C}{w}$ and $\st(w)$ for $\str{C}{w}$.
\end{defn}

We are ready to define the map $\Phi:\widetilde{W}\to\Omega$ that is the affine analog of the Robinson-Schensted correspondence. We refer to the algorithm below as the Affine Matrix-Ball Construction (AMBC).

\begin{itemize}
 \item Input $w\in \widetilde{W}$.
 \item Output: $(P,Q,\rho)\in\Omega$.
 \item Initialize $(P,Q,\rho)$ to $(\varnothing,\varnothing,\varnothing)$.
 \item Repeat until $w$ is the empty partial permutation:
\begin{itemize} 
 \item Record the defining data of $\st(w)$ in the next row of $P$, $Q$, and $\rho$.
 \item Reset $w$ to $\fw{w}$. 
\end{itemize}
\end{itemize}

\begin{thm}
\label{thm:AMBC produces valid element}
This algorithm produces a valid element of $\Omega$, i.e. the rows decrease in size and each tabloid is filled with distinct residue classes.
\end{thm}
The proof is found on page \pageref{pf:AMBC produces valid element}.
\begin{ex}
Consider $w=[1,2,17,5,14,18,20]$ for $n=7$. The first step of AMBC is shown in Figure \ref{fig:forward ex 1}. Here the back corner-posts of the zig-zags form the stream $\str{3}{\{\ol{3},\ol{6},\ol{7}\}, \{\ol{1},\ol{2},\ol{5}\}}$, so we set the first row of $P$ to $\{\ol{1},\ol{2},\ol{5}\}$, the first row of $Q$ to $\{\ol{3},\ol{6},\ol{7}\}$, and the first row of $\rho$ to $3$. The second step of AMBC is shown in Figure \ref{fig:forward ex 2}. Here the back corner-posts of the zig-zags form the stream $\str{3}{\{\ol{2},\ol{4},\ol{5}\}, \{\ol{4},\ol{6},\ol{7}\}}$, so we set the second row of $P$ to $\{\ol{4},\ol{6},\ol{7}\}$, the second row of $Q$ to $\{\ol{2},\ol{4},\ol{5}\}$, and the second row of $\rho$ to $3$. In the third step the permutation consists of all the translates of the ball $(1,10)$, and all the zig-zags will be singletons which will form the stream $\str{1}{\{\ol{1}\},\{\ol{3}\}}$.

The resulting triple $(P,Q,\rho)$ is
$$\tableau[sY]{\ol{1}&\ol{2}&\ol{5}\\\ol{4}&\ol{6}&\ol{7}\\\ol{3}}\quad,\qquad\tableau[sY]{\ol{3}&\ol{6}&\ol{7}\\\ol{2}&\ol{4}&\ol{5}\\\ol{1}}\quad,\qquad\tableau[sY]{3\\3\\1}$$

\end{ex}

\begin{figure}
\centering
\resizebox{.7\textwidth}{!}{\input{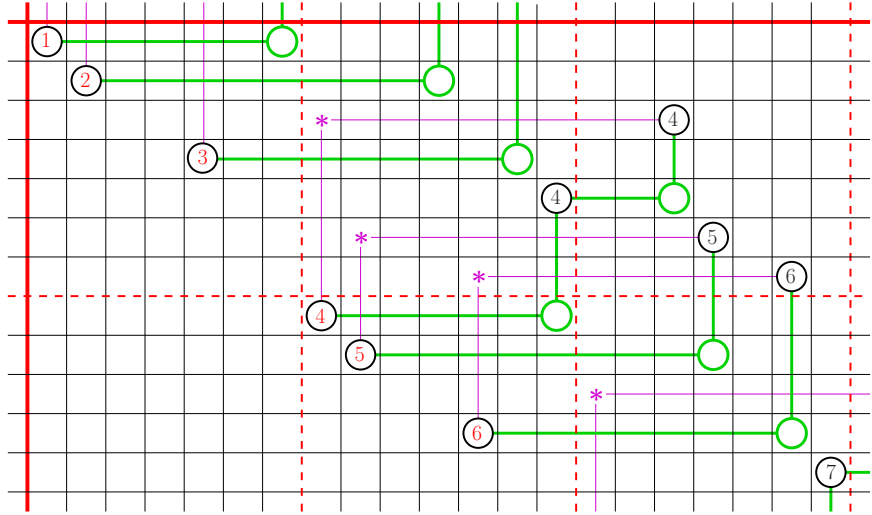}}
\caption{First step of AMBC for $[1,2,17,5,14,18,20]$. The numbering shown is the SW channel numbering; the balls in the channel are numbered in red. The balls of $\fw{w}$ are in green. The cells of $\st(w)$ are denoted by *'s.}
\label{fig:forward ex 1}
\end{figure}

\begin{figure}
\centering
\resizebox{.7\textwidth}{!}{\input{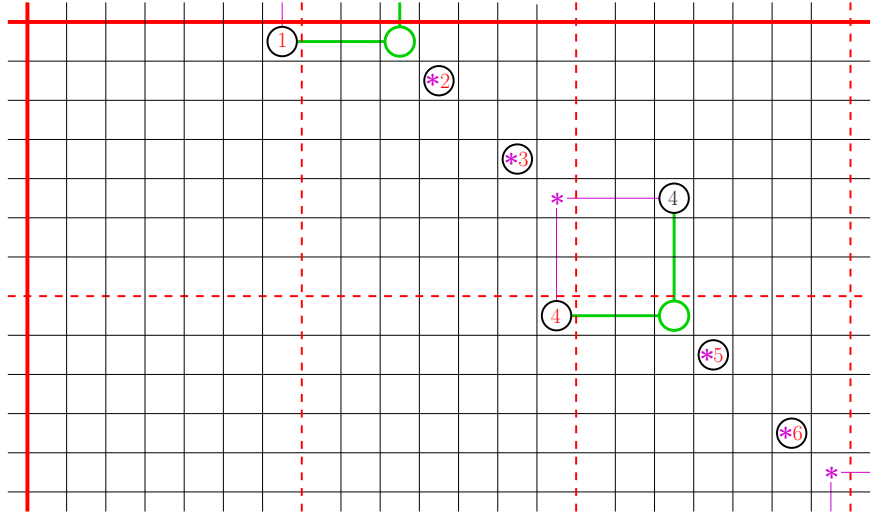}}
\caption{Second step of AMBC for $w=[1,2,17,5,14,18,20]$.}
\label{fig:forward ex 2}
\end{figure}

\begin{rmk}
 This map will be shown to be an injection, but not a surjection. In Section \ref{sec:dominant weights} we shall describe its image, while in Section \ref{sec:backward algorithm} we shall describe a map whose restriction to the image is its inverse.
\end{rmk}

In principle one could choose other proper numberings at each step of the algorithm; it is natural to ask whether we would get different results.

\begin{prop}
\label{prop:irrelevant which proper}
The two tabloids $(P,Q)$ in the outcome do not change if we use an \emph{arbitrary} proper numbering of balls at each step.
\end{prop}
The proof is found on page \pageref{pf:irrelevant which proper}. The weight $\rho$, however, does depend on the choices, so we consistently choose the \emph{southwest channel} numbering in AMBC.

\section{Backward algorithm}
\label{sec:backward algorithm}

We have introduced AMBC which gives a map $\Phi:\widetilde{W}\to\Omega$. In this section we will introduce an algorithm, which we will call the \emph{backward algorithm}, that gives a map $\Psi:\Omega\to\widetilde{W}$. The restriction of $\Psi$ to the image of $\Phi$ will be precisely the inverse of $\Phi$. The behavior of $\Psi$ on the rest of $\Omega$ will be explained in Section \ref{sec:weyl group action} (including a description of the fibers $\Psi^{-1}(w)$); essentially, applying the backward algorithm followed by the forward algorithm preserves the two tabloids and maps the weight into a certain (shifted) dominant chamber.

\subsection{Backward numberings}

First we explain a way to construct a numbering of the balls of a partial permutation with respect to a stream. This numbering will mimic the numbering of balls in the inverse MBC algorithm. 

\begin{defn}
We say that a stream $S$ is \emph{compatible} with a partial permutation $w$ if 
\begin{itemize}
\item $\B_w\cup S$ has at most one cell in each row and at most one cell in each column, and 
\item the flow of $S$ is greater than or equal to the width of $P_w$.
\end{itemize}
\end{defn}

The numbering will be defined by an algorithm:
\begin{itemize}
\item Input: a partial permutation $w$, a compatible stream $S$, and a proper numbering $f:S\to\Z$. 
\item Output: a numbering $d:\B_w\to\Z$.
\item For $b\in\B_w$, let $d_0(b)=\max\limits_{b'\leqslant_{NW} b} f(b')$. Initialize $d$ to $d_0$.
\item Repeat until $d$ is monotone:
\begin{itemize}
\item Choose $b\in\B_w$ such that there are balls southeast of it with the same value of $d$, but no such balls northwest of it. 
\item Define $d'$ to be the semi-periodic numbering which coincides with $d$ on all balls which are not translates of $b$ and has $d'(b) = d(b) -1$.
\item Set $d = d'$.
\end{itemize}
\item Output $d$.
\end{itemize}

\begin{defn}
We refer to the numbering $d_0$ constructed in the above algorithm as the \emph{stream numbering} of $w$. 
\end{defn}

\begin{prop}
\label{prop: backward numbering defined}
 The above algorithm terminates on any valid input and the numbering $d$ does not depend on the choices made. 
\end{prop}
The proof is found on page \pageref{pf:backward numbering defined}. 
\begin{defn}
We refer to the common numbering $d$ above as \emph{the backward numbering} corresponding to $w$ and $S$, and denote it $d_w^{\jbk,S}$.
\end{defn}

\begin{ex}
Consider the partial permutation $[6,1,\varnothing,8,5,7,\varnothing,\varnothing]$ and the stream 
$\str{0}{\{\ol{3},\ol{7},\ol{8}\}, \{\ol{2},\ol{3},\ol{4}\}}$. Number $S$ so that the cell labeled $1$ is the northwest cell both of whose coordinates are positive. The process of obtaining the corresponding backward numbering is illustrated in Figure \ref{fig: backward numbering example}.
\end{ex}

\begin{figure}
\centering
\resizebox{.45\textwidth}{!}{\input{figures/backward_numbering_1.pspdftex}}\quad\resizebox{.45\textwidth}{!}{\input{figures/backward_numbering_2.pspdftex}}
\resizebox{.45\textwidth}{!}{\input{figures/backward_numbering_4.pspdftex}}
\caption{The backward numbering of $[6,1,\varnothing,8,5,7,\varnothing,\varnothing]$ with respect to the stream $\str{0}{\{\ol{3},\ol{7},\ol{8}\}, \{\ol{2},\ol{3},\ol{4}\}}$. The cells of the stream have magenta numbers; these numbers correspond to the values of $f$ on the cells. The top-left figure shows the initial numbering $d_0$, the top-right figure shows the next iteration of $d$, and the bottom figure shows the final backward numbering $d$. }
\label{fig: backward numbering example}
\end{figure}

\begin{defn}
Suppose $S$ is a stream and $f$ is a proper numbering of it. Then by $S^{(i)}$ we mean the cell of $S$ numbered $i$. 
\end{defn}

Often when dealing with the backward algorithm, the proper numbering $f$ is implicit since the overall shift is of no interest. We will only need to talk about it explicitly when we compare the results of the backward step with different streams and hence care about their relative numberings.

\subsection{The map}

Now we are ready to describe the backward algorithm (map $\Psi$):
\begin{itemize}
\item Input: $(P,Q,\rho)\in\Omega$.
\item Output: $w\in \widetilde{W}$.
\item Initialize $w$ to the empty partial permutation.
\item Repeat until $(P,Q,\rho) = (\varnothing,\varnothing,\varnothing)$:
\begin{itemize}
\item Remove the last row of $P$ and call it $A$; remove the last row of $Q$ and call it $B$; remove the last row of $\rho$ and call it $r$. 
\item Let $S = \mathfrak{st}_r(A,B)$; give $S$ some proper numbering.
\item Let $d= d_w^{\jbk,S}$.
\item For each $i\in\Z$, consider the zig-zag $Z_i$ with back corner-post at $S^{(i)}$ and outer corner-posts at the balls of $\B_w$ labeled $i$ (see Figure \ref{fig:backward ex}).
\item Consider the partial permutation $\bk{S}{w}$ whose balls are at the inner corner-posts of all the $Z_i$'s. 
\item Let $w = \bk{S}{w}$.
\end{itemize}
\item Output $w$.
\end{itemize}

\begin{ex}
\label{ex:backward algorithm}
We will apply the backward algorithm to the following element of $\Omega$:
$$(P,Q,\rho) = \left(\quad
\tableau[sY]{\ol{1}&\ol{4}&\ol{5}&\ol{7}\\\ol{3}&\ol{6}\\\ol{2}},\quad
\tableau[sY]{\ol{2}&\ol{3}&\ol{5}&\ol{7}\\\ol{1}&\ol{4}\\\ol{6}},\quad 
\tableau[sY]{2\\0\\1}\quad\right).$$
Taking off the last row gives us a stream $S_1$ of flow $1$. Since the partial permutation at the beginning of the step is empty, each zig-zag will have only one cell. Thus the partial permutation at the end of the step, $w_1 = \bk{S_1}{\varnothing}$ will simply consist of the cells of $S_1$. This is shown in the top-left part of Figure \ref{fig:backward ex}. We are left with
$$\left(\quad
\tableau[sY]{\ol{1}&\ol{4}&\ol{5}&\ol{7}\\\ol{3}&\ol{6}},\quad 
\tableau[sY]{\ol{2}&\ol{3}&\ol{5}&\ol{7}\\\ol{1}&\ol{4}},\quad 
\tableau[sY]{2\\0}\quad\right).$$

Taking off the last row gives the stream $S_2$ shown by magenta numbers in the top-right part of Figure \ref{fig:backward ex}. The black balls show the partial permutation $w_1$; the numbering is the backward numbering with respect to $S_2$. The blue balls show the new partial permutation $w_2 = \bk{S_2}{w_1}$. Finally, the bottom of Figure \ref{fig:backward ex} shows the last step of the backward algorithm; the blue balls give the permutation 
$$\Psi(P,Q,\rho) = [1,6,9,7,10,5,11]\in\widetilde{S_7}.$$
\end{ex}

\begin{figure}
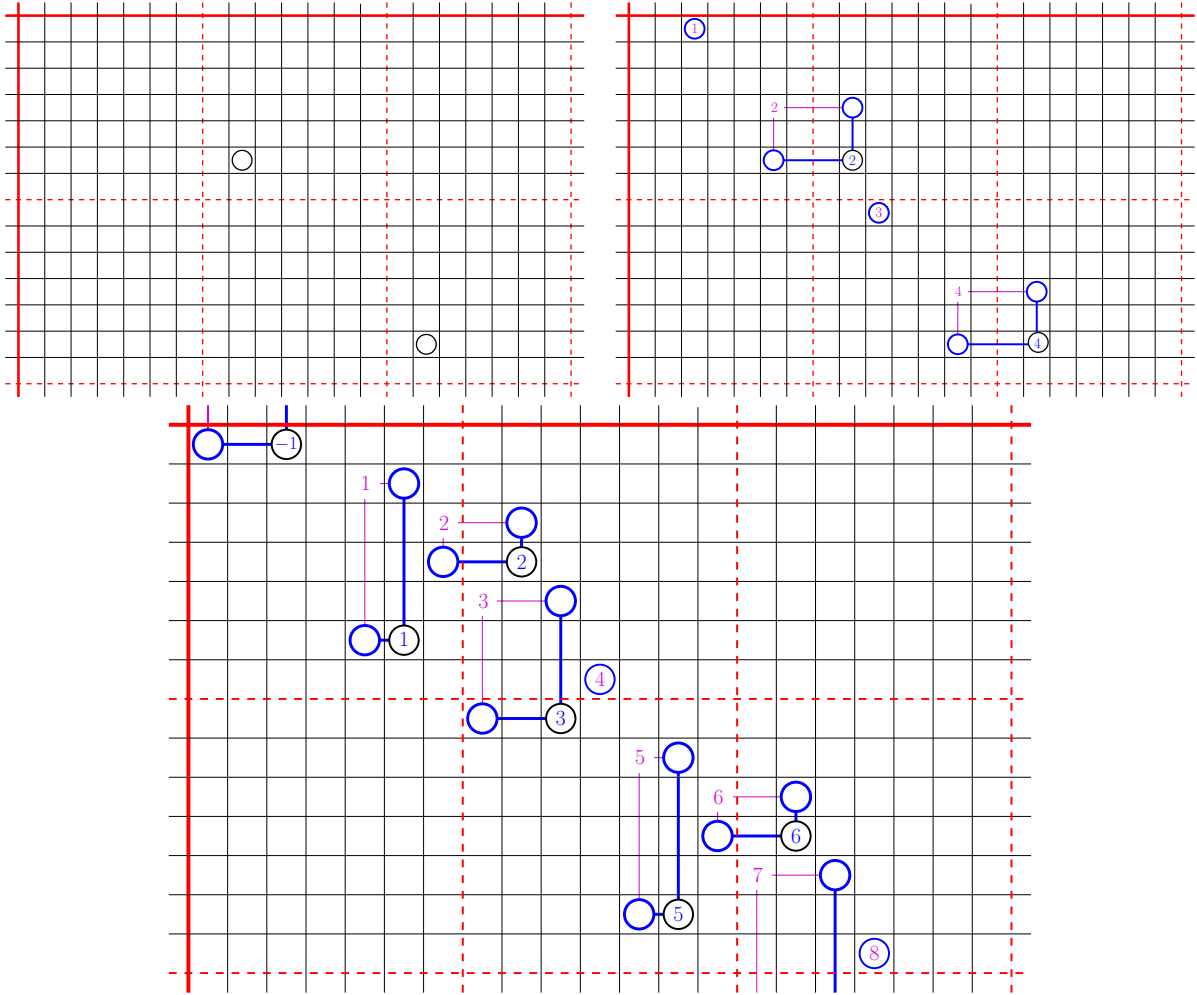

\centering
\resizebox{.47\textwidth}{!}{\input{figures/backward_ex_1.pspdftex}}\quad \resizebox{.47\textwidth}{!}{\input{figures/backward_ex_2.pspdftex}}
\resizebox{.7\textwidth}{!}{\input{figures/backward_ex_3.pspdftex}}
\caption{The steps in the application of the backward algorithm described in Example \ref{ex:backward algorithm}.}
\label{fig:backward ex}
\end{figure}

\section{Bijectivity}

In this section we briefly discuss injectivity of $\Phi$ and then describe the image of $\Phi$ as a subset of $\Omega$. We describe it as a collection of triples $(P,Q,\rho)\in\Omega$, where $\rho$ satisfies certain inequalities depending on $P$ and $Q$. The primary goal is to describe the inequalities explicitly.

The first result implies injectivity:
\begin{thm}
\label{thm:forward backward id}
Suppose $w$ is a partial permutation. Then $\Psi(\Phi(w)) = w$.
\end{thm}
This is actually true at each step regardless of which channel numbering we use:
\begin{prop}
\label{prop:fw bk is id}
Let $w$ be a partial permutation and let $C$ be a channel. Then 
$$\bk{\str{C}{w}}{\fwC{C}{w}} = w.$$
\end{prop}
The proof is found on page \pageref{pf:fw bk is id}. Notice that the proposition is valid only for channel numberings, not for general proper numberings.

\subsection{Concurrent streams}
\label{sec:concurrent}

The inequalities defining dominance will restrict the altitudes of streams encoded by some rows to be sufficiently large. The description of the minimal allowed altitudes uses the notion of concurrent streams, which we will describe in this section. 

For a stream $T$, we will sometimes consider the \emph{partial permutation $t:= \bk{T}{[\varnothing,\dots,\varnothing]}$ corresponding to $T$}. Thus $\B_t$ is the set of cells of $T$. We will say that $S$ is \emph{compatible} with $T$ if $S$ is compatible with $t$.

\begin{lemma}
\label{lem:pre-concurrent}
Let $S, T$ be two streams of the same flow with $S$ compatible to $T$. Then $\bk{S}{t}$ can be partitioned into two disjoint channels. 
\end{lemma}
The proof is found on page \pageref{pf:pre-concurrent}.

\begin{defn}
\label{def:concurrent}
Let $S, T$ be two streams having the same flow with $S$ compatible to $T$. Then $T$ is \emph{concurrent to $S$} if the two disjoint channels of $\bk{S}{t}$ are part of the same river (i.e. the distance between them is zero). 
\end{defn}

\begin{rmk}
Note that the definition is not symmetric; $T:=\str{0}{\{\ol{3},\ol{4},\ol{5}\}, \{\ol{1},\ol{3},\ol{5}\}}$ is concurrent to $S:=\str{0}{\{\ol{1},\ol{2},\ol{6}\}, \{\ol{2},\ol{4},\ol{6}\}}$, but $S$ is not concurrent to $T$. 
\end{rmk}

\begin{prop}
\label{prop:unique concurrent}
Consider four sets $A, A', B, B'\subsetneq [n]$, all of the same size, with $A\cap A' = \varnothing$ and $B\cap B' = \varnothing$. Then there exists a unique $r\in\Z$ such that $\mathfrak{st}_r(A',B')$ is concurrent to $\mathfrak{st}_0(A,B)$. Moreover, for any $l$, $\mathfrak{st}_{r+l}(A',B')$ is the unique stream from the class $\st(A',B')$ that is concurrent to $\mathfrak{st}_l(A,B)$.
\end{prop}
The proof is found on page \pageref{pf:unique concurrent}.

\begin{ex}
Suppose $n = 6$, $S = \str{0}{\{\ol{1},\ol{3},\ol{5}\},\{\ol{1},\ol{2},\ol{6}\}}$. Then the stream from the class $\st(\{\ol{2},\ol{4},\ol{6}\},\{\ol{3},\ol{4},\ol{5}\})$ concurrent to $S$ is $\str{1}{\{\ol{2},\ol{4},\ol{6}\},\{\ol{3},\ol{4},\ol{5}\}}$. The two streams are shown in Figure \ref{fig:concurrent streams}.
\end{ex}

\begin{figure}
\centering
\resizebox{.4\textwidth}{!}{\input{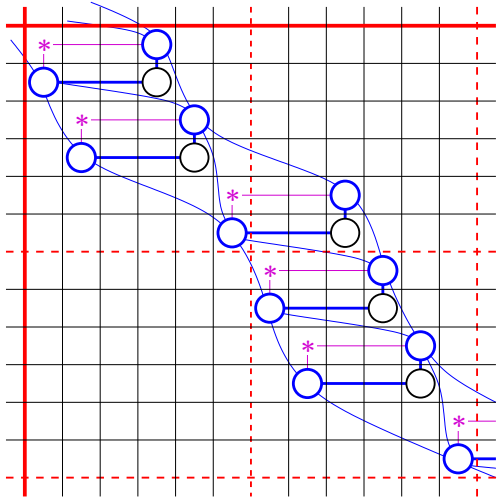}}
\caption{The stream $\str{1}{\{\ol{2},\ol{4},\ol{6}\},\{\ol{3},\ol{4},\ol{5}\}}$, denoted by black balls, is concurrent to $\str{0}{\{\ol{1},\ol{3},\ol{5}\},\{\ol{1},\ol{2},\ol{6}\}}$, denoted by magenta $*$'s, since the backward step results in one river.}
\label{fig:concurrent streams}
\end{figure}

\subsection{Dominant weights}
\label{sec:dominant weights}

Suppose $(P,Q,\rho)\in\Omega$ and $P$ and $Q$ have shape $\lambda$. Let $r_1 = 0$. For each $2\leqslant i\leqslant\ell(\lambda)$: if $\lambda_i < \lambda_{i-1}$ then let $r_i = 0$, otherwise let $r_i$ be the unique integer such that $\mathfrak{st}_{r_i}(P_i, Q_i)$ is concurrent to $\mathfrak{st}_{0}(P_{i-1}, Q_{i-1})$. We will refer to the numbers $r_i$ as the \emph{offset constants}.

\begin{defn}
Consider $(P,Q,\rho)\in\Omega$. Suppose $P$ and $Q$ have shape $\lambda$, and the the offset constants are $r_i$ for $1\leqslant i \leqslant \ell(\lambda)$. Then $\rho$ is \emph{dominant} if for each $1\leqslant i < \ell(\lambda)$ one of the following holds:
\begin{itemize}
\item $\lambda_i > \lambda_{i+1}$, or
\item $\lambda_i = \lambda_{i+1}$ and $\rho_{i+1} - r_{i+1}\geqslant \rho_i$.
\end{itemize}
\end{defn}

It is sometimes convenient to use a slightly different version of the offset constants; define $s_i:=\sum_{j=i_0}^i r_j$, where $i_0$ is the smallest integer with $\lambda_i = \lambda_{i_0}$. We will refer to these as the \emph{symmetrized offset constants}. It is clear that in the language of the above definition, $\rho$ is dominant if for each $1\leqslant i < \ell(\lambda)$ one of the following holds:
\begin{itemize}
\item $\lambda_i > \lambda_{i+1}$, or
\item $\lambda_i = \lambda_{i+1}$ and $(\rho - \mathbf{s})_{i+1}\geqslant (\rho - \mathbf{s})_i$,
\end{itemize}
where $\mathbf{s} = (s_i)_{1\leqslant i\leqslant \ell(\lambda)}$.

\begin{rmk}
In the standard presentation of the Lie algebra $\mathfrak{sl}_n$, the dominant integral weights are increasing sequences of integers. Thus our dominant weights are actually located in some shift of the dominant chamber of a product of special linear Lie algebras.
\end{rmk}

We introduce special notation for the dominant part of $\Omega$.
\begin{defn}
Let $\Omega_{dom}:=\{(P,Q,\rho)\in\Omega: \rho\text{ is dominant}\}$
\end{defn}

\begin{thm}
\label{thm:image1}
$\Phi(\widetilde{W})\subseteq\Omega_{dom}$.  
\end{thm}
The proof is found on page \pageref{pf:image1}.

\begin{thm}
\label{thm:image2}
Suppose $(P,Q,\rho)\in\Omega_{dom}$. Then $\Phi(\Psi(P,Q,\rho)) = (P,Q,\rho)$.
\end{thm}
The proof is found on page \pageref{pf:image2}.

\section{Weyl group action}
\label{sec:weyl group action}

Theorem \ref{thm:image2} describes the composition $\Phi\circ\Psi$ when the weight is dominant. In this section we describe what happens if we started with a non-dominant weight.

\begin{defn}
Let $(P,Q,\rho)\in\Omega$ with $\mathop{\mathrm{sh}}(P) = \mathop{\mathrm{sh}}(Q) = \lambda$ and for $1\leqslant i\leqslant\ell(\lambda)$ let $s_i$ be the symmetrized offset constants. Let $\mathbf{s} = (s_i)_{1\leqslant i\leqslant \ell(\lambda)}$. Let $i_1$ be the smallest index for which $\lambda_{i_1} < \lambda_1$, $i_2$ be the smallest index for which $\lambda_{i_2} < \lambda_{i_1}$, etc.; suppose $i_l$ is the last one of these. Define the \emph{dominant representative} of $\rho$ to be the weight $\rho'$ obtained as follows:
\begin{itemize}
\item Let $\mathbf{x} = \rho - \mathbf{s}$.
\item For $1\leqslant j < l$, let $(y_{i_j}, y_{i_j+1}, \dots, y_{i_{j+1}-1})$ be the increasing rearrangement of $(x_{i_j}, x_{i_j+1}, \dots, x_{i_{j+1}-1})$.
\item Let $\rho' = \mathbf{y} + \mathbf{s}$.
\end{itemize}
\end{defn}
Thus the dominant representative of $\rho$ is precisely what one would get by mapping the weight $\rho$ into the (shifted) dominant chamber by reflections in the (shifted) hyperplanes defining it.

\begin{ex}
Consider the following element of $\Omega$:
$$(P,Q,\rho) = \left(\quad\tableau[sY]{\ol{1}&\ol{2}&\ol{3}\\\ol{5}&\ol{7}&\ol{10}\\\ol{4}&\ol{8}&\ol{11}\\\ol{12}&\ol{13}\\\ol{6}&\ol{9}},\quad \tableau[sY]{\ol{1}&\ol{2}&\ol{5}\\\ol{3}&\ol{8}&\ol{12}\\\ol{4}&\ol{10}&\ol{11}\\\ol{6}&\ol{13}\\\ol{7}&\ol{9}},\quad \tableau[sY]{2\\1\\0\\0\\0}\quad\right).$$
We will find the dominant representative of $\rho$. The vector $\mathbf{s}$ in this case is $\mathbf{s} = (s_i) =(0,0,0,0,1)$. So the vector $\rho-\mathbf{s}$ is $(2,1,0,0,-1)$. The first three entries of $\rho-\mathbf{s}$ and the last two entries of $\rho-\mathbf{s}$ need to be rearranged so each group is increasing, yielding $(0,1,2,-1,0)$. To get $\rho'$ we add $\mathbf{s}$: $\rho' = (0,1,2,-1,1)$.
\end{ex}

\begin{thm}
\label{thm:weyl}
Let $(P,Q,\rho)\in\Omega$. Then $\Phi(\Psi(P,Q,\rho)) = (P,Q,\rho')$, where $\rho'$ is the dominant representative of $\rho$.
\end{thm} 
The proof is found on page \pageref{pf:weyl}. As an easy corollary we can describe the fibers $\Psi^{-1}(w)$.
\begin{cor}
\label{cor:weyl}
Suppose $P$ and $Q$ are two tabloids with the same shape $\lambda$; let $\mathbf{s} = (s_i)_{1\leqslant i\leqslant \ell(\lambda)}$ be the vector of symmetrized offset coefficients. Let $G$ be the parabolic subgroup of the permutation group $S_{\ell({\lambda})}$ generated by transpositions $s_i$ for which $\lambda_i = \lambda_{i+1}$. Then $\Psi(P,Q,\rho) = \Psi(P,Q,\rho')$ if and only if $\rho-\mathbf{s} = a(\rho'-\mathbf{s})$ for some $a\in G$.
\end{cor}

\section{Asymptotic realization via the Robinson-Schensted insertion}
\label{sec:noproofs asymptotic rsk}

In this section we describe how AMBC is related to the classical bumping algorithm (see \cite{ec2} for an introduction). 
Any $w\in\widetilde{A}_{n-1}$ can be
presented as an infinite semi-periodic sequence of integers: $$\ldots,
w(-1), w(0), w(1), w(2), \ldots,$$ where
$w(i+n) = w(i) + n$. One can then apply the usual RSK algorithm to
insert this sequence, choosing the initial place arbitrarily, for
example at $w(1)$.

\begin{ex}
 Let $n=6$ and let $w = [-4,5,-2,7,3,6]$. Then the sequence of $w(i)$, for $i \geq 1$ is
 $$-4,5,-2,7,3,6,2,11,4,13,9,12,8,\ldots.$$
 Inserting this sequence, we obtain a sequence of tableaux
\def\Tscale{1.3}

$$
\hspace{-1cm}\tableau[Y]{-4},\quad\tableau[Y]{-4&5}, \quad\tableau[Y]{-4&-2\\5},\quad\tableau[Y]{-4&-2&7\\5},\quad
\tableau[Y]{-4&-2&3\\5&7},\quad\tableau[Y]{-4&-2&3&6\\5&7},
$$
\bigskip
$$
\hspace{2cm}
\tableau[Y]{-4&-2&2&6\\3&7\\5},\quad\tableau[Y]{-4&-2&2&6&11\\3&7\\5},\quad\tableau[Y]{-4&-2&2&4&11\\3&6\\5&7}, \ldots
$$
\end{ex}

Denote by $P_i(w)$ the tableau obtained by inserting $w(1), \ldots, w(i)$. One can then create the associated tabloid $\overline P_i(w)$ by passing from each number to its residue modulo $n$ and then forgetting the order of elements in each row.

\begin{ex}
 The above sequence of tableaux results in the following sequence of tabloids:
$$
\hspace{-3cm}
\tableau[sY]{\ol{2}}, \tableau[sY]{\ol{2}&\ol{5}}, \tableau[sY]{\ol{2}&\ol{4}\\\ol{5}},
\tableau[sY]{\ol{1}&\ol{2}&\ol{4}\\\ol{5}},
\tableau[sY]{\ol{2}&\ol{3}&\ol{4}\\\ol{1}&\ol{5}}, \tableau[sY]{\ol{2}&\ol{3}&\ol{4}&\ol{6}\\\ol{1}&\ol{5}},
$$
\bigskip
$$
\hspace{2cm}\tableau[sY]{\ol{2}&\ol{2}&\ol{4}&\ol{6}\\\ol{1}&\ol{3}\\\ol{5}}, \tableau[sY]{\ol{2}&\ol{2}&\ol{4}&\ol{5}&\ol{6}\\\ol{1}&\ol{3}\\\ol{5}},
\tableau[sY]{\ol{2}&\ol{2}&\ol{4}&\ol{4}&\ol{5}\\\ol{3}&\ol{6}\\\ol{1}&\ol{5}}, \ldots
$$
\end{ex}

\begin{thm}
\label{thm: asymptotic}
 For large enough $i$ we have 
$$\overline P_{i+n}(w) = \overline P_i(w) + P(w)$$ 
where $P(w)$ is the tabloid given by the Shi correspondence, and each row of the sum of two tabloids is, by definition, the union (as multisets) of the corresponding rows of the two tabloids.
\end{thm}
The proof is found on page \pageref{pf:asymptotic}.

 \begin{ex}
  In the example above we have
	\def\Tscale{1.3}

  $$
  P_{13}(w) = \tableau[Y]{-4&-2&2&4&8&12\\3&6&9&13\\5&7&11},
  $$
  which gives
  $$
  \ol P_{13}(w) = \tableau[sY]{\ol{2}&\ol{2}&\ol{2}&\ol{4}&\ol{4}&\ol{6}\\\ol{1}&\ol{3}&\ol{3}&\ol{6}\\\ol{1}&\ol{5}&\ol{5}} =
\tableau[sY]{\ol{2}&\ol{2}&\ol{4}&\ol{6}\\\ol{1}&\ol{3}\\\ol{5}} + \tableau[sY]{\ol{2}&\ol{4}\\\ol{3}&\ol{6}\\\ol{1}&\ol{5}} = \overline P_7(w) +
P(w).
  $$
  \end{ex}

  This allows one to obtain the following asymptotic version of Shi's insertion.

 \begin{cor}
  The number of $\ol{j}$'s in the $k$-th row of $P(w)$ (which may only be $0$ or $1$) is equal to
  $$\lim_{i \to \infty} \frac{\text{ number of $\ol{j}$'s in $k$-th row of
$\ol P_i(w)$}}{i/n}.$$
 \end{cor}

\begin{rmk}
 Note that this can be used to compute the Shi insertion $P(w)$ in
practice. Indeed, once $i$ is sufficiently large, $\ol P_i(w)$ becomes a
multiple of $P(w)$ plus small ``noise''. One can easily tell this
noise apart and remove it to find $P(w)$. 
\end{rmk}

\begin{rmk}
 One can also get the tabloid $Q(w)$ via the limit of the usual Robinson-Schensted insertion by inserting $w^{-1}$.
\end{rmk}

\section{Distances between channels}

Suppose $w$ is a partial permutation which has $k\geqslant 2$ pairwise disjoint channels. We mentioned with regard to Proposition \ref{prop:sw channel exists} that the collection of southwest (or northeast) balls of a union of two channels is itself a channel. Take any collection of $k$ pairwise disjoint channels. Repeatedly replace a pair of channels by the pair with one having the southwest elements of their union and the other having the northeast elements of their union. This will lead to a maximal disjoint collection $\{C_1,C_2,\dots,C_k\}$ of channels of $w$ such that for all $i$, $C_i$ is southwest of $C_{i+1}$. As we will see in Section \ref{sec:apps to channels}, there exists a maximal disjoint collection $\{D_1,D_2,\dots,D_{k-1}\}$ of channels of $\fw{w}$ which interlace $\{C_1,C_2,\dots,C_k\}$ as shown in Figure \ref{fig:interlacing}.

\begin{figure}
\centering
\resizebox{.4\textwidth}{!}{\input{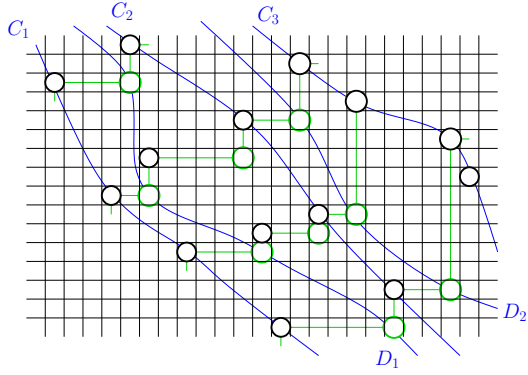}}
\caption{The collection $\{D_1,D_2\}$ of channels of $\fw{w}$ interlaces the collection $\{C_1,C_2,C_3\}$ of channels of $w$.}
\label{fig:interlacing}
\end{figure}

\begin{thm}
\label{thm:dist and alt} 
Suppose $w$ is a partial permutation which has $k\geqslant 2$ disjoint channels. Let $\{C_1,C_2,\dots,C_k\}$ and $\{D_1,D_2,\dots,D_{k-1}\}$ be interlacing collections of channels of $w$ and $\fw{w}$ (as described in Corollary \ref{cor:interlacing channels}). Let $S = \st(w)$ and $T = \st(\fw{w})$. Let $\rho_1 = a(S)$, $\rho_2 = a(T)$, and $r_2$ be the relevant offset constant. Then 
$$\h(C_1, C_2) = \rho_2-\rho_1-r_2,$$
and
$$\h(C_i, C_{i+1}) = \h(D_{i-1}, D_i), \text{ for } 1 < i < k. $$
\end{thm}

The theorem will be split into two cases; see Theorems \ref{thm:dist nonfirst channels} and \ref{thm:dist first channels} for proofs. 


\section{Shi's algorithm and Kazhdan-Lusztig cells}
\label{sec:Shi algorithm and KL cells}

In this section we describe Shi's algorithm (\cite{Shi}) to find a tabloid $P'(w)$ from a (partial) permutation $w$, and state that, in fact, $P'(w)$ coincides with the tabloid $P(w)$ obtained using AMBC. The algorithm consists of doing manipulations to $w$ which preserve the AMBC tabloid $P(w)$ to arrive at a particularly nice partial permutation from which Shi reads off a row of $P'(w)$. Thus it will be sufficient to show that for these special permutations, the row of $P'(w)$ he reads off coincides with the the row of $P(w)$ we read off.

There will be three types of operations involved. We refer to the first type as \emph{repositioning empty rows}; in window notation, this amount to rearranging the $\varnothing$ signs. For example, $[\varnothing,\varnothing,3,5,9]$ is obtained from $[3,5,\varnothing,9,\varnothing]$ by repositioning empty rows. It is clear from the construction of AMBC that this has no effect on the $P$-tabloid of the partial permutation.

The second type of operation is called \emph{shifting the window}. For a partial permutation with window notation $[w_1, \ldots, w_n]$ shifting the window by $1$ gives the partial permutation $[w_2, \ldots, w_n, w_1+n]$ (where $\varnothing + n = \varnothing$). In general, shifting the window is the $\Z$-action generated by the above operation. Since the effect on the matrix for the partial permutation is to shift all the balls vertically by the same amount, this also clearly does not affect the $P$-tabloid.

The third type of operation is a little more involved.
\begin{defn}
Suppose $w$ is a permutation. A permutation $w'$ is obtained from $w$ via a \emph{Knuth move} if for some $i\in\Z$,
\begin{itemize}
\item $w'(i) = w(i+1)$, $w'(i+1) = w(i)$, and for $j\notin\ol{i}\cup\ol{i+1}$, $w'(j) =w(j)$, and
\item $w(i-1)$ or $w(i+2)$ is (numerically) between $w(i)$ and $w(i+1)$.
\end{itemize}
\end{defn}
Lemma \ref{lem:knuth preserves p} shows that these moves do not change the $P$-tabloid.

While we want to start with $w\in\widetilde{A}_{n-1}$, it is convenient to work in slightly greater generality: a partial permutation $w$ such that in the window notation all the $\varnothing$'s are at the beginning. We will call such partial permutations \emph{$\varnothing$-forward} permutations. Consider the following \emph{combing} procedure, which is applicable to a $\varnothing$-forward permutation $w = [w_1,\ldots,w_n]$. 
\begin{itemize}
 \item Let $w_i$ be the first element such that $w_{i-1}>w_i$ (if no such element exists, the procedure does nothing);
 \item Let $w_j$ be the first element with $j<i$ and $w_i < w_j$;
 \item Replace $w$ with
 $$[w_1, \ldots, w_{j-1}, w_i, w_{j+1}, \ldots, w_{i-1},w_{i+1}, \ldots, w_n, w_j+n].$$ 
\end{itemize}
It is not difficult to see that $w'$ is obtained from $w$ by a combination of Knuth moves, window shifts, and repositioning of empty rows. 

Now we can describe Shi's algorithm for affine insertion. By the Greene-Kleitman invariants of $w$ we mean the numbers $\ell_1, \ell_2, \ldots, \ell_m$ where $\ell_1$ is the size of the largest antichain of the Shi poset $P_w$, $\ell_1 + \ell_2$ is the maximal size of union of two antichains, etc.
\begin{itemize}
\item Input: a $\varnothing$-forward partial permutation $w$ with Greene-Kleitman invariants $\ell_1, \ell_2, \ldots,\ell_m$.
\item Output: a tabloid $P'(w)$.
\item Initialize $P'$ to the empty tabloid.
\item For $i$ from $1$ to $m$ repeat:
\begin{itemize}
 \item (C) Comb $w$ repeatedly until the non-$\varnothing$ elements of the window appear in increasing order (this will happen as mentioned in \cite[\textsection 4.3]{Shi}). 
 \item If $w=[\varnothing,\dots,\varnothing, w_1, \ldots, w_k]$ with $w_{\ell_i} < w_1 + n$, then:
  \begin{itemize}
   \item remove $w_1,\ldots, w_{\ell_i}$ from the window and put $\varnothing$'s in their place,
	 \item let the $i$-th row of $P'$ be $\{\ol{w_1}, \ldots, \ol{w_{\ell_1}}\}$,
	 \item proceed to the next value of $i$.
\end{itemize}
 \item Otherwise, set $w$ to $[\varnothing,\dots,\varnothing,w_2, \ldots, w_k, w_1+n]$ and go to step (C) without increasing $i$.
\end{itemize}
\item Output $P'(w) = P'$.
\end{itemize}

\begin{thm} \label{thm:shi} \cite[\textsection 7]{Shi}
The algorithm described gives a bijection between the left Kazhdan-Lusztig cells in type $\widetilde{S}_{n}$ and tabloids of size $n$ filled with $[\ol{n}]$.
\end{thm}

\begin{ex}
 Let $n=7$ and take $w = [7,8,18,5,2,3,13]$. The Greene-Kleitman invariants of $w$ are $3,2,1,1$. 
We start with $i=1$. Repeatedly combing $w$ yields:
$$\begin{array}{lclcl}
[7,8,18,5,2,3,13] &\mapsto& [5,8,18,2,3,13,14]   &\mapsto& [2,8,18,3,13,14,12]\\ 
                  &\mapsto& [2,3,18,13,14,12,15] &\mapsto& [2,3,13,14,12,15,25]\\  
                  &\mapsto& [2,3,12,14,15,25,20] &\mapsto& [2,3,12,14,15,20,32].
\end{array}$$
Since $2+7<12$, we are in the second case; the $2$ from the beginning moves to become a $9$ at the end, and we continue combing:
$$\begin{array}{lclcl}
[2,3,12,14,15,20,32] &\mapsto& [3,12,14,15,20,32,9] &\mapsto& [3,9,14,15,20,32,19]\\ 
                     &\mapsto& [3,9,14,15,19,32,27] &\mapsto& [3,9,14,15,19,27,39].
\end{array}$$
Again, $3+7<14$, so we continue:
$$\begin{array}{lclcl}
[3,9,14,15,19,27,39] &\mapsto& [9,14,15,19,27,39,10] &\mapsto& [9,10,15,19,27,39,21]\\
                     &\mapsto& [9,10,15,19,21,39,34] &\mapsto& [9,10,15,19,21,34,46].
\end{array}$$
Now $9+7>15$, so we are in the first case. We set the first row of $P'$ to $\{\ol{9},\ol{10},\ol{15}\} = \{\ol{1},\ol{2},\ol{3}\}$,  reset $w$ to $[\varnothing,\varnothing,\varnothing,19,21,34,46]$, and proceed to $i=2$. In this case we are lucky and end up in the first case right away since $19 + 7 > 21$. We set the second row of $P'$ to $\{\ol{5}, \ol{7}\}$, reset $w$ to $[\varnothing,\varnothing,\varnothing,\varnothing,\varnothing,34,46]$, and proceed to $i = 3$. In the steps $i=3$ and $i=4$ we trivially end up in the first case of the fork (since $\ell_3 = \ell_4 = 1$), so we set the third row to $\{\ol{6}\}$, and the fourth row to $\{\ol{4}\}$. Thus
$$P'(w) = \tableau[sY]{\ol{1},\ol{2},\ol{3}\\\ol{5},\ol{7}\\\ol{6}\\\ol{4}}.$$
\end{ex}

\begin{thm} \label{thm:mbshi}
For $w\in\widetilde{S_n}$, $P(w) = P'(w)$.
\end{thm}
The proof is found on page \pageref{pf:mbshi}.

\section{The groups \tps{$S_n\subsetneq \ol{W}\subsetneq \widetilde{W}$}{Sn inside affine W inside extended affine W}}
\label{sec:nonext}
The finite group of permutations $S_n$ injects into $\ol{W}$ via the map $w_1w_2\ldots w_n\mapsto [w_1,w_2,\ldots, w_n]$ (slightly abusing notation, we will think of $S_n$ as actually sitting inside of $\ol{W}$). Also, by definition, $\ol{W}\subsetneq \widetilde{W}$. The purpose of this section is to describe, in terms of conditions on the AMBC image, when an extended affine permutation belongs to one of the smaller groups.

\begin{defn}
A tabloid $T$ is \emph{standardizable} if the tableau of the same shape formed by
\begin{enumerate}
\item replacing each residue class $\ol{i}$ by its representative $i\in [n]$, and then
\item sorting each row in increasing order,
\end{enumerate}
is a standard Young tableau.
\end{defn}

\begin{thm}
Suppose $w\in\widetilde{W}$ and $\Phi(w) = (P,Q,\rho)$. Then $w\in S_n$ if and only if $P$ and $Q$ are standardizable and $\rho = 0$.
\end{thm}
\begin{proof}
Suppose $w\in S_n$. Then the matrix for $w$ is block diagonal with $n\times n$ blocks. It is easy to see that in this case there is only one proper numbering which, up to a shift, coincides with the MBC numbering on the block $[n]\times [n]$. So in this case AMBC will just mimic MBC on the block $[n]\times [n]$, so the tabloids $P$ and $Q$ will be standardizable. Moreover, the back corner-posts of the zig-zags with balls in $[n]\times [n]$ will also be in $[n]\times [n]$, so the altitudes of the streams will all be $0$.

If we instead assume that $P$ and $Q$ are standardizable and $\rho = 0$, then we can find a finite permutation $w'$ whose image under the finite Robinson-Schensted correspondence is the pair of standardizations of $P$ and $Q$. By the above argument, viewing $w'$ as an affine permutation we get $\Phi(w') = (P,Q,0)$. Theorem \ref{thm:forward backward id} then shows that $w = w'$.
\end{proof}

Now we discuss when an extended affine permutation actually belongs to the non-extended affine Weyl group.

\begin{thm}
\label{thm:gravity}
Suppose $w\in\widetilde{W}$ and $\Phi(w) = (P,Q,\rho)$. Then $w\in\ol{W}$ if and only if $\sum\limits_i\rho_i = 0$.
\end{thm}

The rest of the section is devoted to the proof. For a cell $c = (a,b)$ define the \emph{diagonal} of $c$ to be $b-a$. Of course, any translate of $c$ will have the same diagonal.
\begin{defn}
Consider a partial permutation $w$. Let $\B$ be the collection of balls of $w$ in rows $1,\dots, n$. The \emph{center of gravity of $w$} is
$$G_w = \frac{1}{n}\sum_{(i,w(i))\in \B} w(i)-i.$$
Equivalently, $G_w$ is the average diagonal of translate classes of balls of $w$.
\end{defn}

The same definition works for any collection of cells which is preserved by translation by $(n,n)$, for example a stream. Note that by definition, for $w\in\widetilde{W}$ we have $w\in\ol{W}$ if and only if $G_w = 0$.

\begin{lemma}
Suppose $w$ is a partial permutation. Then $G_w = G_{\fw{w}} + G_{\st(w)}.$
\end{lemma}
\begin{proof}
Let $Z_i$ be a zig-zag whose inner corner-posts are the balls of $w$ numbered $i$ by the southwest channel numbering. 
The columns of the balls in $\B_\fw{w}\cap Z_i$ 
together with the column of the back corner-post of $Z_i$ make up the same set as the columns of balls in $\B_w\cap Z_i$. 
Similarly for the sets of rows. Hence the sum of diagonals of $\B_w\cap Z_i$ is equal to the sum of diagonals of $\B_\fw{w}\cap Z_i$ plus the diagonal of the back corner-post of $Z_i$. Summing over $i$ from $1$ to $m$ finishes the proof.
\end{proof}


Suppose $\Phi(w) = (P,Q,\rho)$, where the tabloids have $\ell$ rows. Repeated application of the above lemma shows that $G_w = \sum\limits_{i=1}^\ell G_{\str{\rho_i}{P_i, Q_i}}$. Thus to prove Theorem \ref{thm:gravity} it suffices to prove the following lemma.

\begin{lemma}
Suppose $w\in\widetilde{W}$, $\Phi(w) = (P,Q,\rho)$, and the tabloids have $\ell$ rows. Then 
$$\sum\limits_{i=1}^\ell G_{\str{\rho_i}{P_i, Q_i}} = \sum_{i=1}^\ell \rho_i.$$
\end{lemma}
\begin{proof}
Suppose $\rho = (0,0,\dots, 0)$. Then the cells of $\bigcup_i \str{\rho_i}{P_i, Q_i}$ coincide with the balls of some non-affine permutation (not $w$). Hence the left-hand side is also $0$ in this case.

Now if we have any weight $\rho$ and we increase $\rho_i$ by 1, then the right hand side increases by $1$. Also, the left hand side increases by $1$ since the sum of diagonals of cells of $\str{\rho_i}{P_i, Q_i}$ increases by $n$ while the other streams are unaffected. This finishes the proof.
\end{proof}

\part{Proofs.}

\section{Channels, rivers, and proper numberings}

First we will prove a few basic results about channels. Next we will prove that channel numberings are well-defined and proper. Then we will describe the structure theory of general proper numberings, and in particular show that the period of a proper numbering is equal to the width of the Shi poset. We conclude by showing how to obtain proper numberings by using an ``asymptotic'' version of the numbering rule in the non-affine Matrix-Ball Construction.

\subsection{Southwest and northeast channels}

In this section we show that the southwest partial ordering on channels has a least element and a greatest element. Let $\A_w$ be the set of longest antichains in $P_w$. There is a natural partial ordering on $\A_w$: $A_1\leqslant A_2$ if for every $a\in A_1$ there exists $b\in A_2$ such that $a\leqslant b$. 

Suppose $C_1, C_2\in \C_w$ are channels, and $A_1, A_2\in\A_w$ are the corresponding maximal antichains. Observe that in this case $C_1$ is southwest of $C_2$ (recall Definition \ref{def:channel sw ordering}) if and only if $A_1\leqslant A_2$.	

\begin{lemma}
\label{lem:minimal elements of maximal antichains}
Let $A_1, A_2\in \A_w$. Then the set of minimal elements $A_{\mathrm{min}}$ of $A_1\cup A_2$ forms a longest antichain of $P_w$, and likewise for the set of maximal elements.
\end{lemma}
\begin{proof}
Suppose $m=\abs{A_1} = \abs{A_2}$. The set of minimal elements of a subset of a poset is by construction an antichain, so one only needs to check that $A_{\mathrm{min}}$ has $m$ elements. By Dilworth's Theorem, $P_w$ is a disjoint union of $m$ chains; each of those intersect $A_1$ precisely once, and similarly for $A_2$. It is easy to see that for each chain, $A_{\mathrm{min}}$ will contain the smaller of the two elements of $A_1\cup A_2$ in it.
\end{proof}

The lemma implies that taking the southwest elements of the union of two channels gives another channel. Doing this repeatedly we see that the southwest elements of the union of all the channels form a channel.  Thus $\C_w$ has a least element with respect to the southwest partial ordering. Similarly, it has a greatest element. This proves Proposition \ref{prop:sw channel exists}.\label{pf:sw channel exists}

\subsection{Channel numberings}

In this section we prove that channel numberings are well-defined. Let $w$ be a partial permutation and let $C\in\C_w$. Let $m$ be the width of $P_w$. Number the balls of $C$ consecutively; denote this numbering by $\tilde{d}:C\to\Z$. Recall that the channel numbering was defined for any $b\in\B_w$ by 
$$d_w^C(b) := \sup_{(b_0=b, b_1, \dots, b_k)} \tilde{d}(b_k) + k,$$
where the supremum is over all paths from $b$ to $C$ (Definition \ref{def:channel numbering}).

\begin{defn}
 Given a path $(b_0, b_1, \dots, b_k)$ from a ball $b_0$ to $C$ (i.e. $b_k\in C$), we will refer to the number $\tilde{d}(b_k) + k$ as the \emph{worth} of the path with respect to $C$ (see Figure \ref{fig:channel numbering} for an example). 
\end{defn}

First we will do a special case; namely we will show that the channel numbering coincides with $\widetilde{d}$ on the channel. The original statement will easily follow from this special case.

\begin{prop}
\label{prop:channel numbering defined on channel}
Suppose $C$ is a channel of a partial permutation $w$ and $\tilde{d}:C\to\Z$ is a proper numbering of $C$.  Then for any ball $b\in C$ we have $\tilde{d}(b) = d_w^C(b)$, where $d_w^C$ is obtained from this particular fixed shift of $\tilde{d}$. 
\end{prop}
\begin{proof}
\label{pf:channel numbering defined on channel}
Taking the trivial path of length $0$ shows that $d_w^C(b)\geqslant \tilde{d}(b)$. Suppose, toward a contradiction, that $d_w^C(b) > \tilde{d}(b)$. Then there is a path $p'$ from $b$ to $C$ of higher worth than the path to the same ball of $C$ which follows the channel. Extend $p'$ by following the channel to the nearest $(n,n)$-translate of $b$. Denote the result by $p = (b_0,\dots, b_l)$; its worth is $\tilde{d}(b_l)+l$. Then $b_0 = b$ and $b_l = b - t\cdot (n,n)$ for some $t> 0$. Moreover, the worth of the path along the channel is $\tilde{d}(b_l) + tm$, so $l > tm$.

By Dilworth's Theorem, the Shi poset is a union of $m$ chains. Consider the projections of the balls $b_1, b_2, \ldots, b_l$ to $P_w$. By the pigeonhole principle, at least $t+1$ of them lie in the same chain $c$ of $P_w$. The preimage of $c$ under the projection consists of all the translates of a collection $c'$ of balls which form a chain with respect to $\leqslant_{SW}$. Each translate of $c'$ contains at most one ball of $p$, since any pair of balls in $p$ is incomparable with respect to $\leqslant_{SW}$. Let $R$ be the collection of cells strictly northwest of $b$ and weakly southeast of $b_l$. We will now show that only $t$ translates of $c'$ can possibly intersect $R$ (see Figure \ref{fig:channel numbering defined} for an illustration). This will contradict the fact that $t+1$ of the balls $b_1, b_2, \ldots, b_l$ project into $c$. 

Let $b'\in R$ be a ball of the northwest-most translate of $c'$ which intersects $R$. Then $b'+t(n,n)$ is weakly southeast of $b$, and hence the translate of $c'$ containing $b'+t(n,n)$ does not intersect $R$. Thus the translates of $c'$ which could intersect $R$ are those containing $b', b'+(n,n),\ldots, b'+(t-1)(n,n)$; there are only $t$ of them.
\end{proof}

\begin{figure}
\centering\resizebox{.4\textwidth}{!}{\input{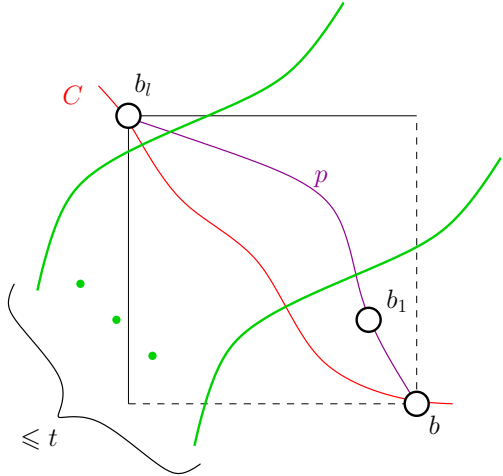}}
\caption{The preimage of $c$ (green) in the proof of Proposition \ref{prop:channel numbering defined on channel}.}
\label{fig:channel numbering defined}
\end{figure}

Now it is easy to finish the general case.

\begin{proof}[Proof of Proposition \ref{prop:channel numbering defined}]
\label{pf:channel numbering defined}
Notice that if a ball $b$ of $w$ is northwest of a ball $b'$, then for any path from $b$ to $C$ there exists a path from $b'$ to $C$ of higher worth. Thus if $d(b')$ is finite then $d(b) < d(b')$ is finite. Now every ball is northwest of some ball of $C$. Hence the proposition follows from Proposition \ref{prop:channel numbering defined on channel}.
\end{proof}

\begin{rmk}
While a priori computing a channel numbering is an infinite problem, one can, in fact, reduce it to a finite one. A path containing two translates of the same ball can be shortened without decreasing its worth. Hence we only need to consider finitely many paths. 
\end{rmk}

For future convenience we give the following two results.
\begin{cor}
\label{cor:reinterpret m}
Suppose $b$ and $b'$ are two balls of $w$ and $b' = b + t(n,n)$ for some $t\geqslant 0$. Suppose $(b = b_0,\dots,b_k =b')$ is a reverse path. Then $k\leqslant tm$.
\end{cor}
\begin{proof}
The argument is exactly the same as in the proof of Proposition \ref{prop:channel numbering defined on channel}.
\end{proof}

\begin{lemma}
\label{lem:high flow implies channel}
 Suppose for some $t>0$ we have a path $(b_0,\dots, b_{tm})$ such that $b_{tm} = b_0 - t\cdot (n,n)$. Then for each $i$ there exists a channel $C$ such that $b_i\in C$. Moreover, $C$ may be chosen so that every ball of $C$ is a translate of some element of the path.
\end{lemma}
\begin{proof}
First suppose that the projections of $b_1,\dots, b_{tm}$ are distinct elements of $P_w$. No more than $t$ of these may be in a single chain (as we have seen in the proof of Proposition \ref{prop:channel numbering defined on channel}), and each antichain has at most $m$ elements. Hence these $tm$ elements may be partitioned into $t$ disjoint antichains, each of size $m$. Then $b_0, \ldots, b_{tm}$ lie in the channels corresponding to these maximal antichains.

It remains to get rid of the restriction that the projections of $b_1,\dots, b_{tm}$ are distinct. The proof is by induction on $t$. The case $t=1$ works by the definition of a channel. Suppose $t>1$. By the previous paragraph, we can restrict to the case when there exist $1\leqslant i < j\leqslant tm$ such that $\varphi_w(b_i) = \varphi_w(b_j)$. Fix an innermost such pair of indices $i, j$, i.e. such that $b_{i+1},\dots, b_j$ project to distinct elements. Since $\varphi_w(b_i) = \varphi_w(b_j)$, we have $b_j = b_i-t'(n,n)$ for some $t'$. One can see that $j = i + t'm$; otherwise the path $p=(b_0,\dots, b_i, b_{j+1}+t'(n,n),\dots, b_{tm}+t'(n,n))$ would have enough elements to contradict Corollary \ref{cor:reinterpret m}. Thus the argument from the first paragraph shows that $b_i,\dots, b_j$ all lie on channels. That path $p$ has fewer than $tm$ steps, so by inductive assumption all elements of $p$ lie on channels. This finishes the proof.
\end{proof}

\subsection{General proper numberings}
\label{sec: proper numberings structure}

In this section we describe the structure of general proper numberings. First we point out that not every proper numbering is a channel numbering.

\begin{ex}
For the permutation $[6,1,8,3,10,5]$ there are only two channels. However, up to an overall shift, there are three proper numberings, shown Figure \ref{fig:proper nonunique}). The one on the left (resp. right) is the southwest (resp. northeast) channel numbering. The one in the middle is not a channel numbering.
\end{ex}

\begin{figure}
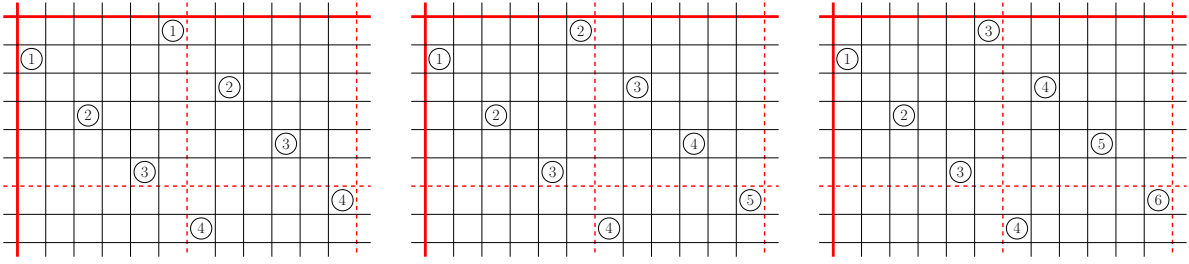

\centering
\resizebox{.3\textwidth}{!}{\input{figures/proper_nonunique1.pspdftex}}
\quad \resizebox{.3\textwidth}{!}{\input{figures/proper_nonunique2.pspdftex}}
\quad \resizebox{.3\textwidth}{!}{\input{figures/proper_nonunique3.pspdftex}}
\caption{Three proper numberings of $[6,1,8,3,10,5]$.}
\label{fig:proper nonunique}
\end{figure}

\begin{rmk}
\label{rmk:monotone is sometimes channel}
We can give a criterion of when a monotone numbering is a channel numbering. Consider a partial permutation $w$, a channel $C$, and its corresponding channel numbering $d_w^C$. Suppose we have another monotone numbering $d$ which coincides with $d_w^C$ on $C$. We will now show that for a ball $b$ we have $d(b) = d_w^C(b)$ precisely when there is a path $(b=b_0, b_1,\dots, b_l)$ to $C$ such that $d(b_{i+1}) = d(b_i)-1$. Indeed, since there necessarily exists a path $(b=b'_0, b'_1,\dots, b'_{l'})$ to $C$ such that $d_w^C(b_{i+1}) = d_w^C(b_i) - 1$, both inequalities $d_w^C(b)\leqslant d(b)$ and $d_w^C(b)\geqslant d(b)$ follow from the fact that the two numberings are monotone.
\end{rmk}

The next goal is to show that any proper numbering is semi-periodic with period equal to the width of the Shi poset. In this series of lemmas we assume that $w$ is a partial permutation, $m$ is the width of $P_w$, and $d$ is an arbitrary proper numbering. We begin with a lemma which proves one side of the inequality. 

\begin{lemma}
\label{lem:half of periodicity}
For any $b \in \B_w$ we have $$d(b - (n,n)) \geq d(b) - m.$$ 
\end{lemma}
\begin{proof}
Toward a contradiction, assume that for some $b$ we have $d(b - (n,n)) < d(b) -
m.$ 
By continuity, one can find a path
 $$\left(b = b^{(0)}_0, b^{(0)}_1, \ldots, b^{(0)}_{M^{(0)}} = b^{(1)}_0\right)$$ 
such that $d\left(b^{(1)}_0\right) = d(b-(n,n))$ and for $0\leqslant i < M^{(0)}$ we have $d\left(b^{(0)}_{i+1}\right) = d\left(b^{(0)}_i\right)-1$. According to our assumption, $M^{(0)} > m$. 
\begin{figure}
\centering
\resizebox{.6\textwidth}{!}{\input{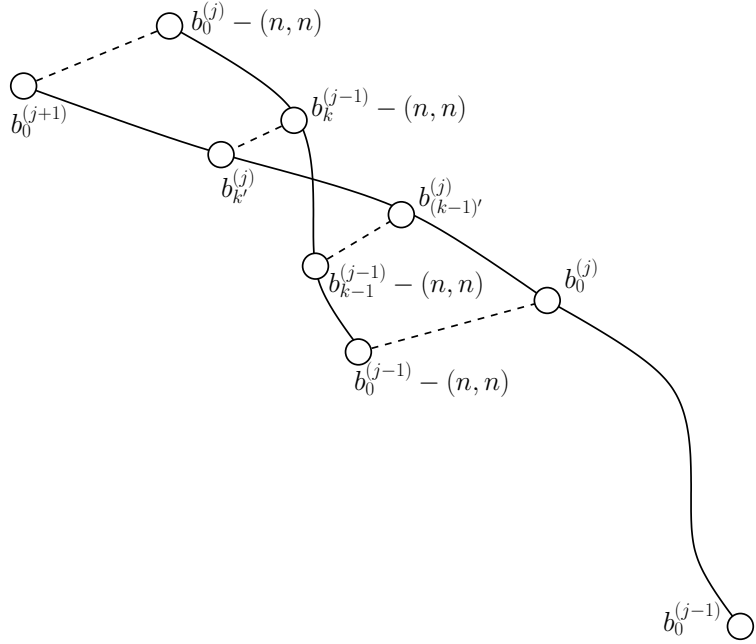}} 
\caption{Constructing a long path between $b^{(j)}_0$ and $b^{(j)}_0-(n,n)$.}
\label{fig:extra_proper_axiom}
\end{figure}

 Consider the path 
$$\left(b^{(0)}_0 -(n,n), b^{(0)}_1 - (n,n), \ldots,  b^{(1)}_0 - (n,n)\right).$$
By monotonicity, for each $i$ we have $d\left(b^{(0)}_{i+1} - (n,n)\right) \leq d\left(b^{(0)}_i - (n,n)\right) - 1$. Therefore,
\[d\left(b^{(1)}_0\right) - M^{(0)} = d\left(b^{(0)}_0 -(n,n)\right) - M^{(0)} \geq d\left(b^{(1)}_0 - (n,n)\right),\]
and
\[d\left(b^{(1)}_0\right)-d\left(b^{(1)}_0 -(n,n)\right)\geqslant M^{(0)} > m.\]

Thus, starting with a violation of the claim between balls $b^{(0)}_0$ and $b^{(0)}_0 -(n,n)$ we constructed another violation of the claim between balls $b^{(1)}_0$ and $b^{(1)}_0 -(n,n)$. Now we can repeat the same construction building a new pair
$b^{(2)}_0$ and $b^{(2)}_0 - (n,n)$ with $d\left(b^{(2)}_0\right)-d\left(b^{(2)}_0 -(n,n)\right) > m$, etc. 

 Let us compare locations of $b^{(j)}_0$ and  $b^{(j-1)}_0 - (n,n)$. First, they must be distinct. Indeed, otherwise the path we built between $b^{(j-1)}_0$ and $b^{(j-1)}_0 - (n,n)$ has length strictly greater than $m$, contradicting Corollary \ref{cor:reinterpret m}. Furthermore, $b^{(j)}_0$ and $b^{(j-1)}_0 - (n,n)$ must be incomparable in the $NW$ order, since the monotone numbering $d$ takes the same value at these balls. Therefore, $b^{(j)}_0$ is either strictly northeast or strictly southwest of  $b^{(j-1)}_0 - (n,n)$. We claim that there are infinitely many values of $j$ where it is strictly northeast and infinitely many where it is strictly southwest. Indeed, if this were not the case and $b^{(j)}_0 = (a_j, b_j)$, then the sequence $(a_j - b_j)$ would eventually be either increasing or decreasing. However, those differences take only a finite set of values since they are constant on translate classes of balls.

Therefore, there exists $j$ such that $b^{(j)}_0$ is northeast of $b^{(j-1)}_0 - (n,n)$ while $b^{(j+1)}_0$ is southwest of $b^{(j)}_0 - (n,n)$ (see Figure \ref{fig:extra_proper_axiom}). Consider the paths
$$p=\left(b^{(j)}_0, b^{(j)}_1, \ldots,  b^{(j)}_{M^{(j)}} = b^{(j+1)}_0 \right) \text{  and  }$$ 
$$q=\left(b^{(j-1)}_0-(n,n), b^{(j-1)}_1-(n,n), \ldots, b^{(j-1)}_{M^{(j-1)}} = b^{(j)}_0 - (n,n)\right).$$ 
Note that the balls of $p$ are numbered (by $d$) by consecutive integers. For each $0\leqslant i\leqslant M^{(j-1)}$, let $i'$ be such that $d\left(b^{(j-1)}_i - (n,n)\right) = d\left(b^{(j)}_{i'}\right)$. Let $k$ be the first index such that $b^{(j-1)}_k - (n,n)$ lies strictly northeast of $b^{(j)}_{k'}$. Now notice that $b^{(j)}_{(k-1)'}$ is strictly south of $b^{(j)}_{k'}$ which is strictly south of $b^{(j-1)}_k - (n,n)$. Also, by the choice of $k$, $b^{(j)}_{(k-1)'}$ is east of $b^{(j-1)}_{k-1} - (n,n)$ which is strictly east of $b^{(j-1)}_k - (n,n)$. So $b^{(j)}_{(k-1)'}$ is strictly southeast of $b^{(j-1)}_k - (n,n)$ and we have a path
$$\left(b^{(j)}_0, \ldots, b^{(j)}_{(k-1)'}, b^{(j-1)}_k -(n,n), \ldots, b^{(j-1)}_{M^{(j-1)}}-(n,n) = b^{(j)}_{0}-(n,n)\right).$$ 
The length of this path is at least equal to the length of $q$, which is strictly greater than $m$, contradicting Corollary \ref{cor:reinterpret m}.
\end{proof}

It is now easy to deduce the desired semi-periodicity on the channels.
\begin{cor}
\label{cor:proper numbers channels contiguosly}
For any channel $C \in \C_w$ and any $b\in C$ we have 
$$d(b - (n,n)) = d(b) - m.$$ 
Moreover, $d$ numbers $C$ by consecutive integers.
\end{cor}
\begin{proof}
Indeed, $d(b - (n,n)) \leq d(b) - m$ since the path $p$ from $b$ to $b - (n,n)$ along $C$ has $m$ steps and the value of $d$ must decrease with each step. Now Lemma \ref{lem:half of periodicity} gives us the opposite inequality. This proves the equality and shows that, in fact, $d$ must decrease by exactly $1$ with every step of $p$. Since $b$ was an arbitrary ball of the channel, this implies that $d$ numbers $C$ by consecutive integers.
\end{proof}
The second part of this statement is exactly the assertion of Lemma \ref{lem:proper numbers channels}\label{pf:proper numbers channels}.

In the next result we show that the semi-periodicity is broken at most finitely many times. 

\begin{lemma}
\label{lem:finitely many gaps}
There are at most finitely many $b\in\B_w$ such that 
\[d(b - (n,n)) > d(b) - m.\]
\end{lemma}
\begin{proof}
 For $b \in \B_w$, we say that there is a \emph{gap} before $b$ if we have $d(b - (n,n)) > d(b) - m$. We will show that for each $b$, only finitely many balls in its translation class have a gap before them. Since there are finitely many translation classes of balls in $w$, this is all we need for the lemma.

Toward a contradiction, assume that there are infinitely many gaps before balls in the translation class of $b$. We will assume that infinitely many of them lie southeast of $b$; the northwest case is similar. Consider an arbitrary channel $C \in \C_w$, and let $\bar b \in C$ be such that $d\left(\bar b\right) = d(b)$ (it exists by Corollary \ref{cor:proper numbers channels contiguosly}). There exists $\ell$ such that $b$ is southeast of $\bar b - \ell(n, n)$. Choose $k$ large enough that more than $\ell m$ translates of $b$ which lie strictly southeast of $b$ and northwest of $b + k(n,n)$ have gaps before them.  We arrive at a contradiction with monotonicity:
$$d(b + k(n, n)) < d(b) + km - \ell m = d\left(\bar b\right) + km - \ell m = d\left(\bar b  + (k-\ell)(n, n)\right),$$ 
while $b + k(n, n)$ is southeast of $\bar b  + (k-\ell)(n, n).$ 
\end{proof}

Now we can finish the proof of semi-periodicity of $d$.

\begin{proof}[Proof of Proposition \ref{prop: period of proper numbering}]
\label{pf:period of proper numbering}
Choose $b\in\B_w$. We want to prove that $d(b - (n,n)) = d(b) - m$. We know by Lemma \ref{lem:half of periodicity}
that $d(b - (n,n)) \geq d(b) - m.$ 

By continuity we can build an infinite path
$$\left(b_0 = b - (n,n), b_1, b_2, \dotsc\right)$$ 
with $d(b_{i+1}) = d(b_i)-1$. By Lemma \ref{lem:finitely many gaps}, there exists $K$ such that for all $i>K$ we have $d(b_i-(n,n)) = d(b_i)-m$. By the Pigeonhole Principle, there exist $j>i>K$ such that $b_i$ and $b_j$ are translates. So
$b_i = b_j + k(n,n)$, where $k = \frac{j-i}{m}$, and $d(b_i) - d(b_j) = j-i$. Consider the following path:
$$\left(b_0 + k(n, n), b_1 + k(n, n),  \dotsc,  b_j + k(n, n) = b_i.\right)$$
By monotonicity, we must have 
\[d(b_0 + k(n, n)) \geq j + d(b_i) = j + d(b_j) + (j-i) = d(b_0) + (j-i)= d(b_0) + km.\]
Now suppose we had $d(b_0) > d(b_0 + (n,n)) - m.$ Then repeatedly using Lemma \ref{lem:half of periodicity} would yield
\[d(b_0) > d(b_0 + (n,n)) - m\geqslant d(b_0 + 2(n,n)) - 2m\geqslant\ldots\geqslant d(b_0 + k(n,n)) - km.\]
This is a contradiction, so we must have $d(b_0) = d(b_0 + (n,n)) - m$.
\end{proof}

Continuing our study of the structure theory of proper numberings, we show that a proper numbering is determined by its value on the channels.
\begin{proof}[Proof of Proposition \ref{prop:proper determined by channels}]
\label{pf:proper determined by channels}
Choose a ball $b$, and form a path $p=(b_0=b,b_1,\dots)$ such that and $d(b_{i+1}) = d(b_i) - 1$. By pigeonhole principle, the projections of two balls of $p$ must coincide. Thus by Lemma \ref{lem:high flow implies channel}, $p$ necessarily leads to some channel. Thus $d(b)$ is the worth of a path from $b$ to some channel. However, since the value of $d$ must decrease at each step of a path, $d(b)$ is greater than or equal to the worth of any path from $b$ to any channel. Hence $d(b)$ is equal to the maximal worth of a path from $b$ to one of the channels. The same statement can be made about $d'$. Hence $d = d'$.
\end{proof}

\subsection{Distance between channels and rivers}

In this section we are concerned with the notion of distance between channels from Definition \ref{def:distance between channels}. The main result of the section is the proof of Proposition \ref{prop:h is a pseudometirc} which states that distance is a pseudometric on the set of channels. We then discuss when a numbering of all channels of $w$ extends to a proper numbering. We do not need this last discussion for future results in this paper, however it finishes our classification of proper numberings. The notion of distance will be more profoundly relevant when studying the weight $\rho(w)$.

We start with several preliminary lemmas.
\begin{lemma}
\label{lem:if nonzero distance then southwest comparable}
Suppose for some $C_1, C_2\in\C_w$, neither $C_1$ is southwest of $C_2$ nor $C_2$ is southwest of $C_1$. Then $\h(C_1, C_2) = 0$.
\end{lemma}
\begin{proof}
Let $A_1$ and $A_2$ be the maximal antichains corresponding to $C_1$ and $C_2$, respectively. By Lemma \ref{lem:minimal elements of maximal antichains}, the set of minimal elements of $A_1\cup A_2$ is a maximal antichain. Hence the set $C$ of southwest-most elements of $C_1\cup C_2$ forms a channel. Since neither of the channels is southwest of the other, $C$ intersects both $C_1$ and $C_2$. Since each channel is numbered by consecutive integers in any proper numbering, it follows easily that $\h(C_1, C_2) = 0$.
\end{proof}
\begin{lemma}
\label{lem:zero distance implies same channel numberings}
Suppose for some $C_1, C_2\in\C_w$, $\h(C_1, C_2) = 0$. If we choose the shifts of $d^{C_1}$ and $d^{C_2}$ so they coincide on $C_1$, then $d^{C_1} = d^{C_2}$.
\end{lemma}
\begin{proof}
Let $b\in\B_w$. Choose a path $(b_0 = b, b_1,\ldots, b_k)$ from $b$ to $C_1$ such that $d^{C_1}$ decreases by $1$ at each step. Now by choice of shifts of the channel numberings, we have $d^{C_2}(b_k) = d^{C_1}(b_k)$ so there exists a path $p=(b_k, b_{k+1},\ldots, b_l)$ from $b_k$ to $C_2$ with $d^{C_2}$ decreasing by $1$ at each step. But since $\h(C_1, C_2) = 0$, we have $d^{C_2}(b_l) = d^{C_1}(b_l)$. Now since $d^{C_1}$ must have decreased at each step of $p$ and the values at both endpoints coincide with those of $d^{C_2}$, we conclude that $d^{C_1}$ decreases by $1$ at each step of $p$. Now Remark \ref{rmk:monotone is sometimes channel} applied to the path $(b=b_0, \ldots, b_l)$ shows that $d^{C_1}(b) = d^{C_2}(b)$. Since $b$ was arbitrary, we have $d^{C_1} = d^{C_2}$.
\end{proof}
\begin{rmk}
\label{rmk:distance direction}
Suppose $C_1, C_2 \in\C_w$. Choose the shifts of $d^{C_1}$ and $d^{C_2}$ so the two numberings coincide on $C_1$. Then for any $b\in C_2$ we have $d^{C_2}(b)\geqslant d^{C_1}(b)$. In fact, for any proper numbering $d$ which coincides with $d^{C_1}$ on $C_1$ and for any ball $b$, we must have $d(b)\geqslant d^{C_1}(b)$. This follows from considering the path from $b$ to $C_1$ with $d^{C_1}$ decreasing by $1$ at each step and noting that $d$ must decrease by at least $1$ on each step. This implies that $\h(C_1,C_2) = d^{C_2}(b)-d^{C_1}(b)$.
\end{rmk}
\begin{lemma}
\label{lem:channel numbering determined by closest channel}
Suppose $C_1,C_2\in\C_w$ and $C_1$ is southwest of $C_2$. Choose shifts of $d^{C_1}$ and $d^{C_2}$ so that they coincide on $C_1$. Then for any $b\in B_w$ which is southwest of $C_1$, we have $d^{C_1}(b) = d^{C_2}(b)$.
\end{lemma}
\begin{proof}
First we show that $d^{C_1}(b) \leqslant d^{C_2}(b)$. Choose a path $(b=b_0, \ldots, b_k)$ from $b$ to $C_1$ such that $d^{C_1}$ decreases by $1$ at each step. So $d^{C_2}(b_k) = d^{C_1}(b_k) = d^{C_1}(b) - k$. Now choose a path $(b_k,b_{k+1}, \ldots, b_l)$ from $b_k$ to $C_2$ such that $d^{C_2}$ decreases by $1$ at each step. So $d^{C_2}(b_l) = d^{C_2}(b_k) - l = d^{C_1}(b) - k - l$. The path $(b=b_0, \ldots, b_l)$ has $k+l$ steps and its end is numbered $d^{C_1}(b) - k - l$ by $d^{C_2}$. By definition, $d^{C_2}(b)\geqslant d^{C_1}(b)$.

Now we show that $d^{C_1}(b) \geqslant d^{C_2}(b)$. Suppose $p=(b=b_0, \ldots, b_k)$ from $b$ to $C_2$ such that $d^{C_2}$ decreases by $1$ at each step. If $p$ intersects $C_1$ then following $p$ from $b$ until the intersection gives a path of worth $d^{C_2}(b)$ (relative to $C_1$) thus proving the inequality. We will now show that one could, in fact, have chosen $p$ which intersects $C_1$.

Suppose $p$ does not intersect $C_1$. Choose $l$ such that $b_l$ is southwest of $C_1$ and $b_{l+1}$ is northeast of $C_1$. No ball of $C_1$ is located both southeast of $b_{l+1}$ and northwest of $b_l$ since such a ball could be added to $p$ to increase its worth with respect to $d^{C_2}$. Let $c$ be the northwest-most ball of $C_1$ which is northeast of $b_l$ and $c'$ be the southeast-most ball of $C_1$ which is southwest of $b_l$. Then $c$ and $c'$ are consecutive elements of $C_1$. Let $(c' = c_0, c_1, \ldots, c_m = c-(n,n))$ be the path along $C_1$. Then the path $(b=b_0,\ldots, b_l, c_0,\ldots,c_m, b_{l+1}-(n,n), \ldots, b_k-(n,n))$ is a path from $b$ to $C_2$ of worth $d^{C_2}(b)$ and it intersects $C_1$.
\end{proof}

Now we can show that distance is a pseudometric on channels.
\begin{proof}[Proof of Proposition \ref{prop:h is a pseudometirc}]
\label{pf:h is a pseudometirc}
It is clear that $\h$ is nonnegative and symmetric, so we only need to check the triangle inequality. Suppose $C_1,C_2$, and $C_3$ are channels. We want to show that $\h(C_1, C_3) \leqslant \h(C_1, C_2) + \h(C_2, C_3)$. 

First consider the case when one of the pairwise distances is $0$. If $\h(C_1, C_3) = 0$, then the statement is clear since $\h$ is non-negative. Since $C_1$ and $C_3$ are interchangeable, without loss of generality assume $\h(C_1, C_2) = 0$. By Lemma, \ref{lem:zero distance implies same channel numberings}, $d^{C_1} = d^{C_2}$. So $\h(C_1, C_3) = \h(C_2, C_3)$ since the quantities only depend on the respective channel numberings.

Now move on to the case when no pairwise distances are $0$; thus the channels $C_1, C_2, C_3$ are pairwise comparable in the southwest order (by Lemma \ref{lem:if nonzero distance then southwest comparable}). Transposing in the main diagonal does not change any of the relevant quantities, so without loss of generality we assume that $C_2$ is not the southwest of the three channels. Since $C_1$ and $C_3$ are interchangeable, again without loss of generality we assume that $C_1$ is southwest of $C_3$. Thus there are two cases depending on whether $C_2$ is southwest or northeast of $C_3$. Suppose $C_2$ is southwest of $C_3$.

Choose shifts of the three channel numberings so that $d^{C_1}$ and $d^{C_2}$ agree on $C_1$ while $d^{C_2}$ and $d^{C_3}$ agree on $C_2$. Let $b\in C_2$ and $b'\in C_3$. Then $\h(C_2,C_3) = d^{C_3}(b') - d^{C_2}(b')$ (see Remark \ref{rmk:distance direction}). Note that by Lemma \ref{lem:channel numbering determined by closest channel},  $d^{C_1}$ and $d^{C_3}$ agree on $C_1$. Thus $\h(C_1,C_3) = d^{C_3}(b') - d^{C_1}(b').$ Also note that $d^{C_2}$ coincides with $d^{C_1} + h(C_1,C_2)$ (i.e. the numbering which numbers a ball $b$ by $d^{C_1}(b)+ h(C_1,C_2)$) on $C_2$. Applying the transpose of Lemma \ref{lem:channel numbering determined by closest channel} shows that $d^{C_1}(b') + h(C_1,C_2) = d^{C_2}(b').$ Now if , then we have $h(C_1,C_2) = d^{C_2}(b) - d^{C_1}(b)$, $h(C_2,C_3) = d^{C_3}(b') - d^{C_2}(b')$. So  
\[\h(C_1,C_2) + \h(C_2,C_3) = \h(C_1,C_2) + d^{C_3}(b') - d^{C_2}(b') = d^{C_3}(b') - d^{C_1}(b') = \h(C_1,C_3),\]
and the desired inequality is actually an equality.

If $C_2$ is the northeast channel of the three, then the same argument as above shows that  $\h(C_1, C_2) = \h(C_1, C_3) + \h(C_3, C_2)$, and a strict inequality holds.
\end{proof}

A pseudometric on a set naturally partitions it into equivalence classes, which in this case we call rivers. Now we describe which numberings of channels can be extended to proper numberings of $w$.

\begin{defn}
\label{def:consistent}
Suppose $\tilde{d}:\bigcup\limits_{C\in\C_w} C \to\Z$ is a numbering which restricts to a proper numbering on each channel. We say that $\tilde{d}$ is \emph{consistent on channels} if for any $C_1, C_2\in\C_w$ and some (equiv. every) $b\in C_2$ we have 
\begin{equation}
\label{eqn:consistency}
d^{C_1}(b) \leqslant \tilde{d}(b)\leqslant d^{C_1}(b) + \h(C_1, C_2),
\end{equation}
where the shift of $d^{C_1}$ is chosen so that $d^{C_1}$ and $\tilde{d}$ coincide on $C_1$. 
\end{defn}

\begin{prop}
Suppose $w$ is a partial permutation and $\tilde{d}:\bigcup\limits_{C\in\C_w} C\to\Z$. Then there exists a proper numbering which restricts to $\tilde{d}$ if and only if $\tilde{d}$ is consistent on channels.
\end{prop}
\begin{proof}
First, we will show that the condition of being consistent on channels is necessary. Suppose $d$ is a proper numbering and $C_1, C_2\in\C_w$ and $b\in C_2$. Choose the shifts of $d^{C_1}$ and $d^{C_2}$ to coincide with $d$ on $C_1$. Consider a path of maximal worth (with respect to $d^{C_1}$) from $b$ to $C_1$. The value of $d$ must decrease at every step, and the two numberings coincide on $C_1$, so $d^{C_1}(b) \leqslant d(b)$. 

There exists $b'\in C_1$ and a path $p$ from $b'$ to $C_2$, ending at $b$, such that $d^{C_2}$ decreases by $1$ at each step. The value of $d$ must decrease at every step of $p$ and $d(b') = d^{C_2}(b')$, so $d(b) \leqslant d^{C_2}(b)$. By Remark \ref{rmk:distance direction}, $d^{C_2}(b) = d^{C_1}(b) + \h(C_1, C_2)$. This demonstrates that the restriction of $d$ to the union of channels is consistent on channels.

Now suppose $\tilde{d}$ is consistent on channels. Then for any $b$ let 
\[d(b) := \sup_{(b_0 = b, b_1, \dots, b_k)} \tilde{d}(b_k) + k,\]
where the supremum is taken over all paths from $b$ to channels of $w$. The resulting numbering is always well defined and proper regardless of whether $\tilde{d}$ satisfies \eqref{eqn:consistency}. The inequalities on $\tilde{d}$ ensure that for any $b$ in any channel we have $d(b) = \tilde{d}(b)$, i.e. that $d$ restricts to $\tilde{d}$ on the channels. 
\end{proof}

\begin{rmk}
The two inequalities in the Definition \ref{def:consistent} are superfluous in that one of them is obtained from the other by reversing the roles of $C_1$ and $C_2$. The reason we presented them in this way was to make an algorithmic construction of all proper numberings possible, as follows. Fix a proper numbering on the southwest channel $C_1$. Find the southwest channel $C_2$ that does not intersect $C_1$. Assign to it any of the proper numberings consistent with $C_1$. Find the southwest channel $C_3$ that does not intersect $C_1$ or $C_2$. Assign to it any numbering that is consistent with both $C_1$ and $C_2$. Repeat until no new channel can be found; this will define a numbering consistent on channels.
\end{rmk}

\subsection{Stabilization}
\label{sec:stabilization of numberings}

In this section we show that if we start from an arbitrary point numbering balls as in the non-affine Matrix-Ball Construction, then  eventually we will get a proper numbering.

First, we describe what we mean by the above numbering scheme. For each $z\in P_w$ pick a ball $b_z$ in $\varphi^{-1}(z)$; only the balls $\{b_z + k(n,n): z\in P_w, k\geqslant 0\}$ will be numbered. For example balls below a horizontal line between the $0$-th and $1$-st rows would do. Number these balls as in non-affine MBC; an example is shown in Figure \ref{fig:cutoff numbering example}. Denote the resulting numbering by $d_{cut}$.

\begin{figure}
\centering
\resizebox{.6\textwidth}{!}{\input{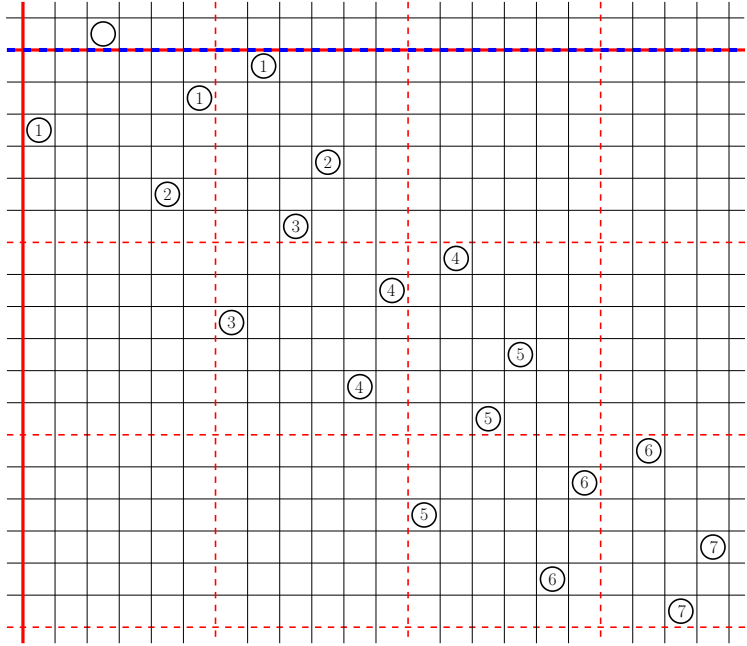}}
\caption{An example of the cutoff numbering $d_{cut}$. The cutoff line is the dashed blue line between the $0$-th and $1$-st rows. We observe semi-periodicity for balls numbered at least $c = 4$.}
\label{fig:cutoff numbering example}
\end{figure}

It is clear that the resulting numbering satisfies the Monotonicity and Continuity properties for balls for which $d_{cut} > 1$ (balls numbered $1$ do not have balls numbered $0$ northwest of them). In fact, even more is true.
\begin{thm}
\label{thm:stabilization to proper}
Given a choice of $d_{cut}$, there exists $c\in\Z^{> 0}$ such that for $b\in \B_w$ with $d_{cut}(b)\geqslant c$, we have $d_{cut}(b+(n,n)) = d_{cut}(b)+m$.
\end{thm}

\begin{proof}

Notice that in the numbering $d_{cut}$, any pair of balls with the same number are related in the NE partial order. Modulo translation by multiples of $(n,n)$, there are finitely many collections of balls which form chains in the NE partial order. Hence there are numbers $c,c',t\in \Z^{>0}$ such that $d_{cut}^{-1}(c')$ is formed by translating $d_{cut}^{-1}(c)$ by $t(n,n)$. Let $m' = c'-c$. From the definition of the non-affine Matrix-Ball numbering, one sees that if $d_{cut}(b) \geqslant c$, then $d_{cut}(b + t(n,n)) = d_{cut}(b) + m'$. Consider a new numbering $\tilde{d}$ which coincides with $d_{cut}$ on balls numbered (by $d_{cut}$) $c$ or above, and is extended to the remaining balls according to the rule $\tilde{d}(b-t(n,n)) = \tilde{d}(b)-m'$. It is clear that $\tilde{d}$ is monotone and continuous, i.e. it is a proper numbering. By Proposition \ref{prop: period of proper numbering}, it is semi-periodic with period $m$.
\end{proof}

Thus while $d_{cut}$ may behave erratically at the beginning, eventually it coincides with a proper numbering. 

\section{Tabloids}
In this section we prove two interesting results about the tabloids obtained in AMBC. When describing the algorithm we used the southwest channel numbering at every step. The first result is that, in fact, if we use an arbitrary proper numbering at every step, then the tabloids $P(w)$ and $Q(w)$ will be the same. The second result is that the tabloid $P(w)$ is actually the same as the tabloid obtained from Shi's algorithm.

\subsection{Independence of the proper numbering choices}
\label{sec: independence of proper numbering}

As part of the argument, we will be inserting via regular Robinson-Schensted algorithm larger and larger pieces of an affine permutation. We start by introducing necessary notation.

\subsubsection{Partial non-affine permutations}
A \emph{partial non-affine permutations} (PNAP) is a sequence $(a_1,\dots, a_l)$ where each $a_i$ belongs to $\Z\cup\{\varnothing\}$ and all the integers present in the sequence are distinct. The integer $l$ will be referred to as the \emph{length} of the PNAP. We can apply the Robinson-Schensted bumping algorithm to such a sequence: every time we run into a $\varnothing$ we don't do anything to the tableaux, but increment the next entry to be added to the recording tableau. We can also view a PNAP as a matrix in our usual way and apply the (non-affine) MBC. The tableaux resulting from the two algorithms must coincide; this reduces to the statement that MBC provides an implementation of the Robinson-Schensted correspondence. For a partial permutation $w$, let $a_w^k$ be the PNAP $(w(1), w(2),\dots, w(kn))$.

We say that a sequence $(f_k)_k$ of PNAPs is \emph{uniformly bounded} if there exist a constant $N$ such that for every $k$, the diagonal of every ball of $f_k$ is between $-N$ and $N$. All the sequences of PNAPs we deal with (for example, $(a_w^k)_k$) will be uniformly bounded, so we will often not bother mentioning it. We next give a basic result about uniform boundedness which will be used repeatedly in the rest of the section.

\begin{lemma}
\label{lem:bounded height}
Suppose $(f_k)_k$ is a uniformly bounded sequence of PNAPs. Then the number of rows of $P(f_k)$ is bounded above by a constant independent of $k$.
\end{lemma}
\begin{proof}
Let $N$ be as the uniform bound on the diagonals of $f_k$, as in the definition. Since every ball of $\fw{f_k}$ lies strictly directly east of some ball of $f_k$ and strictly directly south of some other ball, the balls of $\fw{f_k}$ must lie (weakly) between diagonals $- N+1$ and $N-1$. So after $N+1$ steps of MBC, no balls will be left. Hence $P(f_k)$ has at most $N+1$ rows.
\end{proof}

\subsubsection{Stabilization}
The idea of the proof of independence of the proper numbering choices is the following. We take $a_w^k$ for larger and larger values of $k$ and apply MBC. It turns out that the residue classes of the entries by which the tableaux grow with each additional step eventually stabilize (Lemma \ref{lem:stablizing}). However, we will also show that if we do AMBC with arbitrary proper numbering choices, then the resulting tabloids are the same as in the above stabilization. Thus the tabloids do not depend on the numbering choices. In this section we introduce the relevant stabilization phenomena.

\begin{defn}
A sequence $(f_k)_k$ of PNAPs is called \emph{stable} if for $k$ large enough, the balls in $\B_{f_{k+1}}\setminus\B_{f_k}$ are obtained from the balls in $\B_{f_{k}}\setminus\B_{f_{k-1}}$ by translation by $(n,n)$.
\end{defn}
Note that the sequence $(a_w^k)_k$ is manifestly stable. A stable sequence of PNAPs naturally defined an affine partial permutation; namely the one whose balls are the translates of $\B_{f_{k}}\setminus\B_{f_{k-1}}$ once $k$ is large. We will refer to this partial permutation as the \emph{underlying partial permutation} of $(f_k)_k$.

\begin{defn}
Suppose $(f_k)_{k}$ is a sequence of PNAPs and $i\in\Z^{>0}$. Let $x^{i,k}_1,x^{i,k}_2,\ldots x^{i,k}_l$ be the entries of the $i$-th row of $P(f_k)_i$. Define $S^{i,k}$ to be the multiset $\left\{\ol{x^{i,k}_1},\ldots,\ol{x^{i,k}_l}\right\}$. We say that $(f_k)_k$ \emph{$P$-stabilizes in row $i$ to a set $S\subseteq [\ol{n}]$} if for sufficiently large $k$ we have $S^{i,k-1}\subseteq S^{i,k}$ and  $S^{i,k}\setminus S^{i,k-1} = S$.
\end{defn}
While we could have let $S$ be a multiset in the above definition, it turns out that it will be a set in all cases of interest.

\begin{lemma}
\label{lem:stablizing}
A stable sequence of PNAPs $(f_k)_{k}$ $P$-stabilizes in every row.
\end{lemma}
\begin{proof}
For all $k$, let $d_k$ be the MBC numbering of $f_k$. Let $w$ be the underlying partial permutation of $(f_k)_k$. By the proof of Theorem \ref{thm:stabilization to proper} (the theorem itself technically does not apply since it was stated for cutoff numberings of affine permutations), there exists $M$ such that for $k$ sufficiently large, for every $b\in\B_{f_k}$ with $d_k(b) > M$, if $b + (n,n)$ also belongs to $\B_{f_k}$ then $d_k(b+(n,n)) = d_k(b)+m$, where $m$ is the channel density of $w$. Choose $k$ large enough that the above is true and $d_k$ takes values larger than $M$ on all balls of $\B_{f_k}\setminus \B_{f_{k-1}}$. Let $d$ be the proper numbering of $w$ such that for every $b\in\B_{f_k}$ with $d_k(b)>M$, we have $d(b) = d_k(b)$.

Recall that the first row of of $P(f_k)$ consists of the column indices of the southwest balls of each zig-zag. We can see that $S^{i,k-1}\subseteq S^{i,k}$. Indeed, if a the southwest ball $b$ of a zig-zag of $f_{k-1}$ is not the southwest ball of the corresponding zig-zag of $f_{k}$, then the translate of $b$ in $\B_{f_k}\setminus \B_{f_{k-1}}$ contributes the same element to $S^{i,k}$ as $b$ contributed to $S^{i,k-1}$. Furthermore, we can see that $S^{i,k}\setminus S^{i,k-1}$ is just the set of residue classes of columns of southwest balls of the zig-zags of $w$ corresponding to $d$.

Hence the sequence $P$-stabilizes in row $1$. Since the MBC numbering stabilizes to a proper numbering in the sense explained in the first paragraph, the sequence $(\fw{f_k})_k$ is stable. The fact that this sequence $P$-stabilizes in row $1$ is exactly the statement that $(f_k)_{k}$ $P$-stabilizes row $2$. By Lemma \ref{lem:bounded height}, the number of rows is uniformly bounded, so we can repeat this argument to show $P$-stabilization in all rows.
\end{proof}

\subsubsection{Criterion for equality of $P$-stabilizations}

We would like to show that the $P$-stabilization of $(a_w^k)_k$ coincides, row by row, with the tabloid we get from AMBC with arbitrary proper numbering choices. The way we will do this is interpret the AMBC tabloid as a $P$-stabilization of a sequence of PNAPs related to $(a_w^k)_{k}$, and then use a criterion for when two sequences of PNAPs $P$-stabilize to the same subset of $[\ol{n}]$.

\begin{defn}
\label{def:asymptotically alike}
Suppose $f = (f_k)_k$ and $g = (g_k)_k$ are two sequences of PNAPs. Moreover assume that for each $k$, the size of $f_k$ (i.e. the number of elements in the sequence) is the same as the size of $g_k$; call this number $l_k$. Then $f$ and $g$ are \emph{asymptotically alike} if 
there exists $M\in\Z^{\geqslant 0}$ (independent of $k$) such that for $M<i<l_k-M+1$ we have $f_k(i) = g_k(i)$.
\end{defn}
\begin{defn}
For skew tableaux (resp. tabloids) $T_1, T_2$, define $\h(T_1,T_2)$ to be the number of integers which lie in different rows in $T_1$ and $T_2$. If an integer is only in one of the tableaux (resp. tabloids) then it is considered to lie in different rows; if an integer is not in either of the tableaux (resp. tabloids) then it is considered to lie in the same row. 
\end{defn}
\begin{rmk}
One readily sees that $\h$ is a pseudometric on the set of skew tableaux and a metric on the set of tabloids. 
\end{rmk}

\begin{lemma}
\label{lem:asymptotically the same}
Let $(f_k)_k$, $(g_k)_k$ be uniformly bounded sequences of PNAPs which are asymptotically alike. Then $\h(P(f_k),P(g_k))$ is bounded above by a constant independent of $k$.
\end{lemma}

\begin{proof}
Let $M$ be as in Definition \ref{def:asymptotically alike}. For every $k$, let $s^f_{k}$ be the first $M$ elements of $f_k$, $e^f_{k}$ be the last $M$ elements of $f_k$, and $m^f_{k}$ be the middle $l_k - 2M$ elements of $f_k$. Analogously for $g_k$; so $m^f_k = m^g_k$.

First let us compare $P(s^f_{k}m^f_{k})$ and $P(s^g_{k}m^g_{k})$, where by putting two PNAPs next to each other we mean the concatenation of them as sequences. By the jeu-de-taquin theory (see, for example 
\cite[Corollary A1.2.6]{ec2} ) they are the unique straight-shaped tableaux of the jeu-de-taquin classes of the following skew tableaux
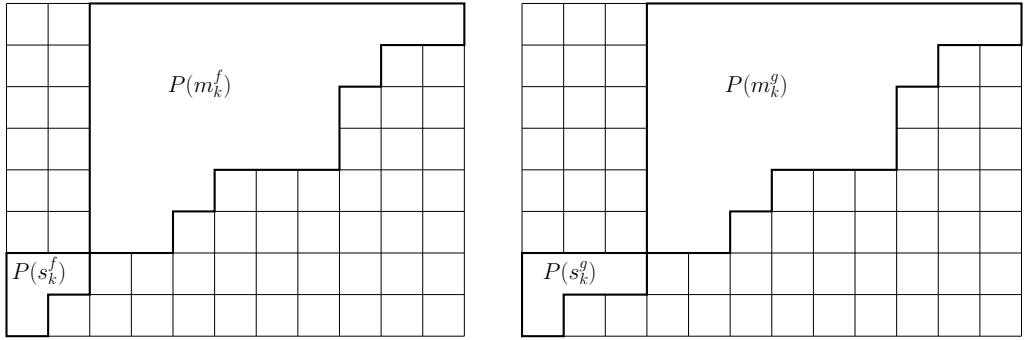
\[\resizebox{!}{.2\textheight}{\input{figures/insert_s1ak.pspdftex}}\qquad\resizebox{!}{.2\textheight}{\input{figures/insert_s2ak.pspdftex}}\] 
The distance $\h$ between the two skew tableaux is at most $2M$.

The number of jeu-de-taquin slides which need to be performed is equal to the area of the top-left rectangle in the figure. Let $N$ be the maximal number of rows of $P(m^f_k)$ (it exists by Lemma \ref{lem:bounded height}). The width of that rectangle is $\leqslant M$ and the height is $\leqslant N$. So at most $M\cdot N$ slides are necessary. The number of elements 
whose row is changed by a slide is at most the height of the skew tableau, which in our case is bounded by $M+N$. Thus 
$$\h(P(s^f_{k}m^f_{k}),P(s^g_{k}m^g_{k}))\leqslant 2M+2M\cdot N(M+N).$$

Now we analyze what happens as we insert the ending sections via the bumping algorithm. As each element is inserted, the number of rows increases by at most $1$. Since the number of rows of $P(s^f_{k}m^f_{k})$ (and $P(s^g_{k}m^g_{k})$) is at most $M+N$ and there are at most $M$ insertions, the number of rows during the process remains at most $N+2M$. Thus each insertion changes the row of at most $N+2M+1$ integers. Hence
$$\h(P(f_k),P(g_k))\leqslant 2M+2M\cdot N(M+N) + 2M(N+2M+1).$$
\end{proof}

If both sequences in the above lemma $P$-stabilized to different subsets in some row, then $\h(P(f_k),P(g_k))$ would grow linearly in $k$, so we have the following corollary.

\begin{cor}
\label{cor:criterion for same stabilization}
If $(f_k)_k$, $(g_k)_k$ are two uniformly bounded sequences of PNAPs which are asymptotically alike and, moreover, $P$-stabilize in row $i$ to $S_f$ and $S_g$, respectively, then $S_f = S_g$.
\end{cor}

\subsubsection{Another sequence of PNAPs}

This section connects the previous results to AMBC with arbitrary proper numbering choices. We call a sequence $\mathfrak{D} = (d_1,\dots, d_l)$ a \emph{numbering sequence} if $d_1$ is a proper numbering of $w$, $d_2$ is a proper numbering of $w_2:=\fwC{d_1}{w}$, $d_3$ a proper numbering of $w_3:=\fwC{d_2}{w_2}$, etc. For a numbering sequence $\mathfrak{D}$ we can form a tabloid $P(\mathfrak{D})$ as in AMBC except that during the $i$-th step we use the proper numbering $d_i$ instead of the southwest channel numbering.

\begin{defn}
Suppose $w$ is a partial permutation and $w'$ is a PNAP of length $l$ which satisfies $w(i) = w'(i)$ for every $1\leqslant i \leqslant l$. For a proper numbering $d$ of $w$, define $w'\vert_d$ to be the PNAP of the same length which consists of only those balls of $w'$ for which all the balls of $w$ with the same value of $d$ are also part of $w'$.
\end{defn}

\begin{ex}
An example is shown in Figure \ref{fig: ak and akp}. Here 
\[w = [7,2,15,4,19,20,13,16,18,11]\] 
and $w' = a_w^2$ and the proper numbering $d$ is given in the figure. The balls to be removed to get $w'\vert_d$ are circled in blue; for example, the balls $(1.7)$ and $(4,4)$ need to be removed since there is a ball of $w$, $(-1,8)$, which is also numbered $2$ by $d$ and is not part of $w'$.
\begin{figure}
\centering
\resizebox{.8\textwidth}{!}{\input{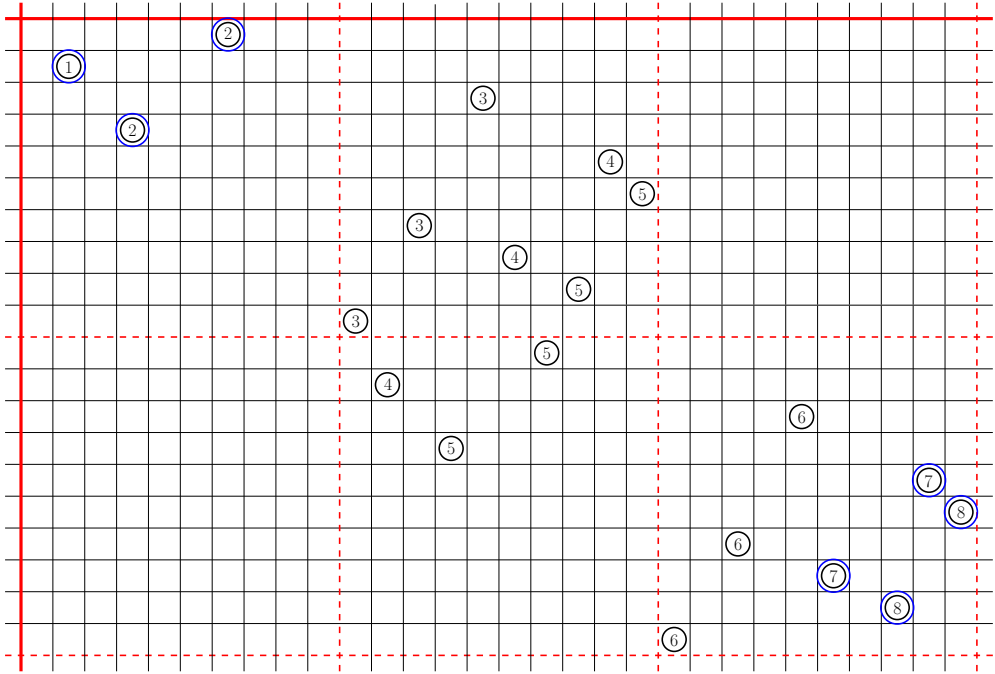}}
\caption{Cutting a PNAP according to a proper numbering}
\label{fig: ak and akp}
\end{figure}
\end{ex}

\begin{lemma}
\label{lem:asymptotics first step}
Suppose $d$ is a proper numbering of a partial permutation $w$. For each $k$, let $f_k = a_w^k$ and let $g_k = a_w^k\vert_d$. Then for each $i$, both sequences $(f_k)_k$ and $(g_k)_k$ $P$-stabilize in row $i$ to the same subset of $[\ol{n}]$. In row $1$,
they $P$-stabilize to the set $S$ of residue classes of columns of the back corner-posts of the zig-zags associated with $d$.
\end{lemma}
\begin{proof}
We start by showing that $(f_k)$ and $(g_k)$ are asymptotically alike. Let $d_{max}$ (resp. $d_{min}$) be the largest (resp. smallest) index of a diagonal of a ball of $w$. Let $t_w:=d_{max}-d_{min}$. For each $i\in\Z$, the difference between the row of the southern ball numbered $i$ and the row of the northern ball numbered $i$ is at most $t_w$ (in fact, this is true for any pair of cells comparable with respect to $\leqslant_{NE}$). Hence to get from $f_k$ to $g_k$ we will not be discarding any balls south of row $t_w$ and north of row $kn-t_w$. Thus the two sequences are, indeed, asymptotically alike.  

Next we show that the sequences $P$-stabilize in every row. Since $(f_k)_k$ is manifestly stable, its $P$-stabilization follows by Lemma \ref{lem:stablizing}. We claim that $(g_k)_k$ is also stable. A ball $b\in\B_w$ south of row $2t_w$ is a ball of $g_k$ precisely when every $b'\in\B_w$ with $d(b') = d(b)$ lies north of row $kn$. But this is translation invariant, so provided $b$ is south of row $2t_w$, we have $b\in \B_{g_k}$ if and only if $b+(n,n)\in \B_{g_{k+1}}$. Choosing $k$ with $(k-1)n>2t_w$ finishes the claim. By Corollary \ref{cor:criterion for same stabilization}, $(f_k)_k$ and $(g_k)_k$ $P$-stabilize in every row and the $P$-stabilizations coincide.

Now the MBC numbering of $g_k$ agrees up to shift with $d$, so, according to the proof of Corollary \ref{cor:criterion for same stabilization}, $(g_k)$ $P$-stabilizes in row $1$ to $S$.
\end{proof}

Now we can finish the proof that for every numbering sequence $\mathfrak{D}$, $(a_w^k)_k$ $P$-stabilizes to $P(\mathfrak{D})$ (and hence $P(\mathfrak{D})$ depends only on $w$).

\begin{proof}[Proof of Theorem \ref{thm: asymptotic} and Proposition \ref{prop:irrelevant which proper}]
\label{pf:asymptotic}\label{pf:irrelevant which proper}
Let $\mathfrak{D} = (d_1,\ldots,d_l)$ be a numbering sequence for $w$. Then by Lemma \ref{lem:asymptotics first step}, the sequences $(a_w^k)_k$ and $(a_w^k\vert_{d_1})_k$ $P$-stabilize in row $1$ to $P(\mathfrak{D})_1$, and their $P$-stabilizations coincide. Now consider the sequence $(\fw{a_w^k\vert_{d_1}})_k$ which $P$-stabilizes in row 
$1$ to the second row of the $P$-stabilization of $(a_w^k)_k$. Now $\left(\fw{a_w^k\vert_{d_1}}\right)_k$ is asymptotically alike to $\left(a_\fwC{d_1}{w}^k\right)_k$. The remaining conditions of Corollary \ref{cor:criterion for same stabilization} are also satisfied, so $\left(\fw{a_w^k\vert_{d_1}}\right)_k$ and $\left(a_\fwC{d_1}{w}^k\right)_k$ have the same $P$-stabilizations. Again by Lemma \ref{lem:asymptotics first step}, $\left(a_\fwC{d_1}{w}^k\right)_k$ $P$-stabilizes in row $1$ to $P(\mathfrak{D})_2$, so this is the second row of the $P$-stabilization of $(a_w^k)_k$. This argument may be repeated to finish the inductive proof.
\end{proof}

\subsection{Equivalence to Shi's algorithm}

The first half of Shi's algorithm consisted of a series of Knuth moves, window shifts, and repositioning of empty rows (recall the definition from Section \ref{sec:Shi algorithm and KL cells}). It is obvious that window shifts and repositioning of empty rows preserve the tabloid $P(w)$. Of course they do affect $Q(w)$, but that is of no relevance in this section. We will show that Knuth moves also preserve $P(w)$. After that we will describe the southwest channel numbering of a permutation in the resulting special form. We will see that taking the initial maximal antichain and placing it in the row of the tabloid corresponds to the operation of forming a row of $P(w)$ in AMBC. Moreover, we will see that the step of replacing the initial maximal antichain with $\varnothing$'s corresponds exactly to a step of AMBC. 

\begin{lemma}
\label{lem:knuth preserves p}
Suppose $w$ differs from $w'$ by a Knuth move. Then $P(w) = P(w')$.
\end{lemma}
\begin{proof}
Consider the PNAPs $a_w^k$ and $a_{w'}^k$. Then $a_w^k$ and $a_{w'}^k$ differ by a sequence of ordinary Knuth moves as well as possibly in the first and last entries. It is well known that non-affine Knuth moves preserve the insertion tableau \cite[Corollary A1.1.12]{ec2}, so Lemma \ref{lem:asymptotically the same} guarantees that the distance between $P(a_w^k)$ and $P(a_{w'}^k)$ is bounded above independently of $k$. The asymptotic interpretation of $P(w)$ given in Theorem \ref{thm: asymptotic} finishes the proof.
\end{proof}

Now we describe the the southwest channel numbering for a permutation in the special form that results after the combing stage.

\begin{lemma}
Suppose $w = [\varnothing, \ldots,\varnothing,w_1,\ldots, w_l]$ is a partial permutation, where $w_1,\ldots, w_l$ is an increasing sequence of integers and $w_m < w_1+n$ ($m$ is the width of $P_w$). For $1\leqslant i\leqslant l$, define $d( (n+1-i, w(n+1-i)) + k(n,n)) = (l+1-i) + km$ (for an example, see Figure \ref{fig:combed_perm}). Then $d$ is the southwest channel numbering of $w$.
\end{lemma}

\begin{figure}
\centering\resizebox{.8\textwidth}{!}{\input{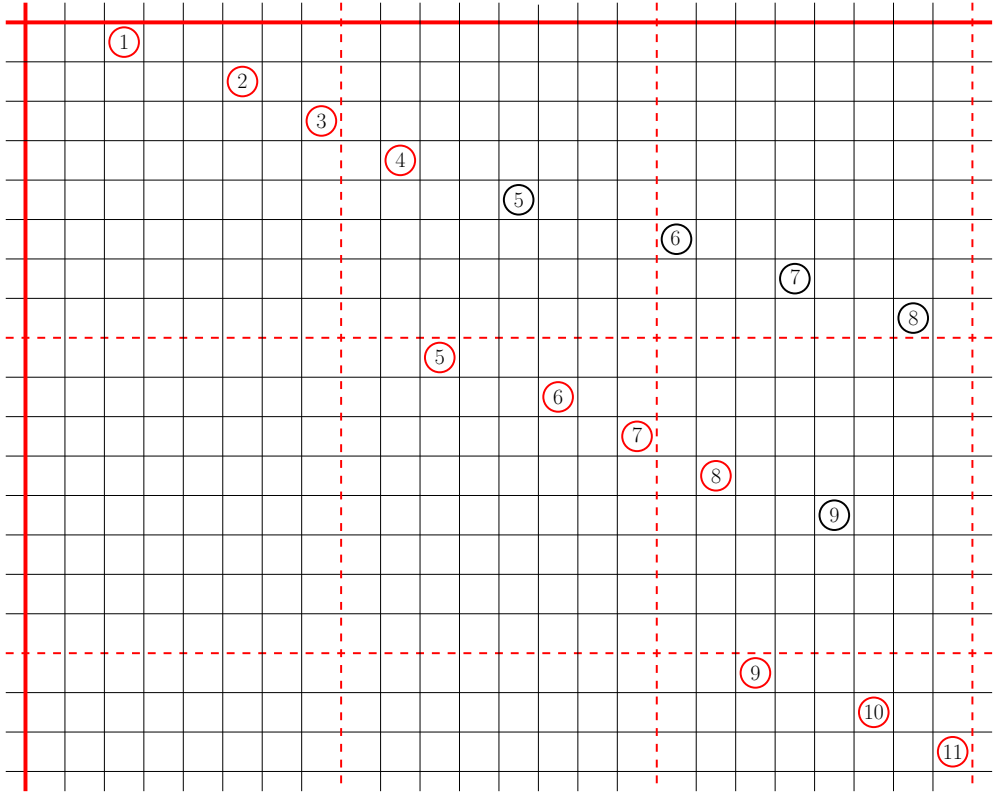}}
\caption{A combed permutation with southwest channel marked in red and with southwest channel numbering.}
\label{fig:combed_perm}
\end{figure}

\begin{proof}
First, let us show that $d$ is a proper numbering. It is easy to see that continuity holds (the non-trivial case follows from the inequality $w_m < w_1+n$). It is sufficient to show that for each $i$, we have $d^{-1}(i)$ is a chain in the $\leqslant_{SW}$ partial ordering. Indeed, if $b$ is northwest of $b'$ and $d(b)\geqslant d(b')$ then by Continuity we can find $b''$ northwest of $b$ with $d(b'') = d(b')$; of course $b''$ and $b'$ are incomparable with respect to $\leqslant_{SW}$. By periodicity, it is sufficient to consider $1\leqslant i\leqslant m$. 

Consider the set of balls $d^{-1}(i)$. A ball in this set is specified by two parameters $1\leqslant j\leqslant n$ and $k\in\Z$ satisfying one equation $j + km = i$; its coordinates are $(j + kn, w(j) + kn)$. Consider two balls in this set with adjacent values of the parameter $k$. We will show that the ball $b$ with parameters $j=i-km$ and $k$ is northeast of the ball $b'$ with parameters $i-(k+1)m$ and $k+1$. 
We have $b = (i+k(n-m), w(i-km)+kn)$ and $b' = (i+(k+1)(n-m), w(i-(k+1)m)+(k+1)n)$. It is clear that $b'$ is south of $b$ since 
$n\geqslant m$. 

Now since $i-(k+1)m, i-km\in [n]$ and the permutation is combed, the balls in rows $i-(k+1)m, \dots, i-km$ form a chain with respect to $\leqslant_{SE}$. Since there are $m+1$ balls and all of them lie within $n$ consecutive rows, they must not lie within $n$ consecutive columns (otherwise the width of $P_w$ would be at least $m+1$). Hence $w(i-(k+1)m) < w(i-km)-n$. Thus $w(i-(k+1)m)+(k+1)n < w(i-km)+kn$. Hence $b'$ is west of $b$, finishing the proof that $d$ is a proper numbering.

The condition $w_m < w_1 + n$ implies that the balls $(n-l+1, w_1),\ldots, (n-l+m, w_m)$ as well as their translates form a channel. This channel consists of the southwest balls of $w$, hence it is necessarily the southwest channel. It is clear that from every ball one can take a path to this channel so that the value of $d$ decreases by one on each step; namely go up one row at a time until you reach a translate of one of the above balls. So $d$ must be the southwest channel numbering.
\end{proof}

\begin{proof}[Proof of Theorem \ref{thm:mbshi}]\label{pf:mbshi}
It is easy to see from the above lemmas that the first row of the tabloid resulting from Shi's algorithm coincides with the first row of the tabloid from AMBC. Moreover, one can see that to get $\fw{w}$ from $w$ one needs to erase the southwest channel and shift the remaining balls south by $n-m$. Since an overall shift south does not affect the $P$-tabloid, taking $\fw{w}$ corresponds to replacing a segment of the window in Shi's algorithm with $\varnothing$'s. The position of $\varnothing$'s in the permutation has no effect on the $P$-tabloid, so we may as well shift them to the beginning.
\end{proof}

\section{A tale of two streams}
\label{sec:two streams}

In this section we prove some basic results concerning backward numberings and use them to work out in detail what happens in the backward algorithm (i.e. when building $w$ from $(P,Q,\rho)$) when the shape of the tabloids $P$ and $Q$ is a two-row rectangle.

\subsection{Basic results concerning backward numberings}

First we prove that the backward numbering is well-defined.
\begin{proof}[Proof of Proposition \ref{prop: backward numbering defined}]
\label{pf:backward numbering defined}
Consider a partial permutation $w$ and a stream $S$ compatible with it. We will first find a monotone numbering $d':\B_w\to\Z$ such that for every $b\in\B_w$ and every $i$, $d'(b)\leqslant d_i(b)$. 

Let $m_1$ be the width of the Shi poset of $w$ and let $m_2\geqslant m_1$ be the flow of the stream $S$. Consider an arbitrary proper numbering $d_p$ of $w$. Define another numbering $d_p':\B_w\to\Z$ by 
$$d_p'( (i,w(i))+k(n,n)) = d_p(i, w(i))+km_2,$$
where $1\leqslant i\leqslant n$ and $k\in\Z$. By construction, $d_p'$ is a semi-periodic numbering with period $m_2$. Since $m_2\geqslant m_1$, it is also monotone. We know that $d_0$ is a semi-periodic numbering with period $m_2$. So $z:=\max\limits_{b\in\B_w} d_p'(b)-d_0(b)$ is well defined. Define $d'$ by $d'(b) = d_p'(b) - z$. This is clearly a monotone, semi-periodic numbering of period $m_2$ which is no larger than $d_0$ on every ball.

Now suppose for some $i$ and for all $b\in\B_w$, $d'(b) \leqslant d_i(b)$. We will show that in this case, regardless of choice, $d'(b) \leqslant d_{i+1}(b)$. Choose a ball $b$ such that there are balls southeast of it with the same value of $d_i$, but no such balls northwest of it. Let $k = d_i(b)$. We will consider $d_{i+1}$ which results from decrementing the value on $b$. Choose $b'$ southeast of $b$ with $d_i(b') = k$. By assumption, $d'(b')\leqslant k$. But $d'$ is monotone, so $d'(b) < k $. Thus $d'(b)\leqslant k-1 = d_{i+1}(b)$. Therefore it is only possible to make finitely many steps in the algorithm, i.e. it terminates.

The only thing remaining is to show that the result of the backward numbering algorithm is independent of the choices. This is done via a standard diamond lemma argument (originally formalized in \cite{dimond}). Consider a numbering $d_i$ which resulted from some choices. Suppose we have two possible numberings $d_{i+1}$ and $\tilde{d}_{i+1}$ which result from making different choices in the next step of the numbering algorithm. Suppose $b$ and $\widetilde{b}$, respectively, are the balls whose number has to be decremented to get the new numberings. Since $d_i$ is weakly monotone, the ball $\widetilde{b}$ can be chosen in the next step for $d_{i+1}$ to produce $d_{i+2}$. 
Similarly, $b$ can be chosen in the next step for $\tilde{d}_{i+1}$, and this will produce the same numbering $d_{i+2}$. Thus we have local confluence which implies global confluence by the diamond lemma.
\end{proof}

\begin{rmk}
\label{rmk:lower bound for backward numbering}
Suppose $w$ is a partial permutation and $S$ is a compatible stream. Suppose we have a monotone numbering $d'$ which is nowhere greater than the stream numbering $d_0$ (equivalently, for any $b\in\B_w$, $S^{(d(b))}$ lies northwest of $b$). Then it follows from the above proof that $d'\leqslant d_w^{\jbk,S}$. Thus we get an alternative definition of the backward numbering: it is the largest monotone numbering $d$ such that for every $b\in\B_w$, $S^{(d(b))}$ lies northwest of $b$.
\end{rmk}

\begin{rmk}
\label{rmk:backward equals stream}
There must exist a ball $b$ whose numbering does not change during the algorithm; i.e. $d_0(b) = d_w^{\jbk,S}(b)$. Otherwise we could add 1 to the backward numbering of all balls and get a monotone numbering which is nowhere greater than $d_0$. This contradicts Remark \ref{rmk:lower bound for backward numbering}.
\end{rmk}

\begin{prop}
\label{prop:backward numbers channel consecutively}
Suppose $w$ is a partial permutation and $S$ is a compatible stream whose flow is equal to the width of the Shi poset of $w$. Then $d_w^{\jbk,S}(C) = \Z$, i.e. the backward numbering numbers channels by consecutive integers.
\end{prop}
\begin{proof}
By construction, $d_w^{\jbk,S}$ is semi-periodic with period equal to the flow of $S$. Consider $b, b'\in C$ such that $b' = b + (n,n)$. Now $d_w^{\jbk,S}$ must increase by at least $1$ at every step of the reverse path from $b$ to $b'$ along $C$. The number of steps is the width of $P_w$. Since it is equal to the flow of $S$, $d_w^{\jbk,S}$ must have increased by exactly $1$ at each step. This finishes the proof.
\end{proof}

A step of the backward algorithm produces a partial permutation, and this permutation comes with an induced numbering; in the next two lemmas we will show that this numbering is proper.

\begin{lemma}
\label{lem:zigzags dont intersect}
Consider a partial permutation $w$ and a compatible stream $S$. A step of the backward algorithm involves a collection $\{Z_i\}_{i\in\Z}$ of zig-zags. These zig-zags pairwise do not intersect; moreover every cell of $Z_i$ is strictly south and strictly east of some cell of $Z_{i-1}$.
\end{lemma}
\begin{proof} Let $d$ (resp. $d_0$) be the backward (resp. stream) numbering. By construction, the southwest cell $c_{i,SW}$ of $Z_i$ is east of the southwest cell $c_{i-1,SW}$ of $Z_{i-1}$. Similarly, $c_{i,NE}$ is south of $c_{i-1,SW}$. So, no cell of $Z_i$ lies southwest of $c_{i-1,SW}$ or northeast of $c_{i-1,NE}$. Thus every cell of $Z_i$ is either southeast or northwest of some cell $Z_{i-1}$. We will show that the second of these is impossible.

Suppose there exists a cell $a_0$ of $Z_i$ northwest of a cell $a_1$ of the $Z_{i-1}$ (see Figure \ref{fig:zigzags dont intersect} for a schematic illustration). There exists $c\in Z_i\cap\B_{\bk{S}{w}}$ either directly north or directly west of $a_0$, and $b\in Z_{i-1}\cap\B_w$ either directly south or directly east of $a_1$. So $c$ is northwest of $b$. By construction, the only balls of $w$ which share a row or column with $c$ must be numbered $i$, so $c$ is, in fact, strictly north and strictly west of $b$. Now $d(b) = i-1$ and $d_0(b)\geqslant i$, so $d(b)\neq d_0(b)$. Looking at the backward numbering algorithm, an easy induction shows that $d(b)\neq d_0(b)$ implies that there exists $b'\in\B_w$ southeast of $b$ with $d(b') = d(b)+1$. But then $c$ and $b'$ are both in $Z_i$, which is impossible given their relative positions. 
x`
Since no cell of $Z_i$ can be weakly northwest of a cell of $Z_{i-1}$, the two zig-zags do not intersect. Choose a cell $x\in Z_i$. From the first paragraph, we now conclude that $x$ is southeast of some cell $x'\in Z_{i-1}$. It remains to find a cell $z'\in Z_{i-1}$ which is strictly south and strictly east of $x$. Now $x$ must be either strictly south of strictly east of $x'$. Suppose, without loss of generality, that it is the first one. If it is also strictly east, then we are done. Otherwise, let $y'$ be the southern cell of $Z_{i-1}$ directly south of $x'$. Then $x$ is strictly south of $y'$ since otherwise it would be in the intersection of the zig-zags. Let $z'$ be the western cell of $Z_{i-1}$ directly west of $y'$. Since $c_{i-1,SW}$ is strictly west of $x$, $z'\neq y'$ and thus $x$ is strictly east of $z'$, finishing the proof. 
\end{proof}

\begin{figure}
\centering\resizebox{.2\textwidth}{!}{\input{figures/zigzags_dont_intersect.pspdftex}}
\caption{Cell positions in the proof of Lemma \ref{lem:zigzags dont intersect}.}
\label{fig:zigzags dont intersect}
\end{figure}

\begin{defn}
\label{def:numbering induced by backward step}
Let $w$, $S$, and $\{Z_i\}_{i\in\Z}$ be as in Lemma \ref{lem:zigzags dont intersect}. Define a numbering $d$ of $\bk{S}{w}$ by $d(b) = i$ if $b\in Z_i$. Then $d$ will be called \emph{the numbering induced by the backward step}.
\end{defn}

\begin{lemma}
\label{lem:backward is proper}
Let $w$, $S$, and $\{Z_i\}_{i\in\Z}$ be as in Lemma \ref{lem:zigzags dont intersect}. Then the numbering $d$ induced by the backward step is proper.
\end{lemma}
\begin{proof}
First we prove monotonicity. Suppose $b,c\in\B_{\bk{S}{w}}$, $b$ is northwest of $c$, and $d(b) > d(c)$. By repeated application of Lemma \ref{lem:zigzags dont intersect}, we can find a cell $b'\in Z_{d(c)}$ which is strictly northwest of $b$. Then $b'$ and $c$ lie in the same zig-zag and one of them is strictly northwest of the other, giving the desired contradiction.

Now we prove continuity; suppose $b\in\B_{\bk{S}{w}}$ with $d(b) = i$. By Lemma \ref{lem:zigzags dont intersect}, there exists some cell $a$ of $Z_{i-1}$ strictly northwest of $b$. Let $b'\in\B_{\bk{S}{w}}$ be a ball either directly north or directly west of $a$; then $b'$ is northwest of $b$ and $d(b') = d(b)-1$, as desired.
\end{proof}

\begin{rmk}
We will later see (Proposition \ref{prop:backward numbering is channel}) that the numbering in the above lemma is a channel numbering.
\end{rmk}

\begin{cor}
\label{cor:with of backward Shi poset}
Consider a partial permutation $w$ and a compatible stream $S$. The width of the Shi poset of $\bk{S}{w}$ is equal to the flow of $S$.
\end{cor}
\begin{proof}
Let $d$ be the proper numbering from Lemma \ref{lem:backward is proper}. By construction, the period of $d$ is equal to the flow of $S$.  Proposition \ref{prop: period of proper numbering} finishes the proof.
\end{proof}

\subsection{The case of two-row rectangles.}
\label{subsec:two streams}

In this section we explore in detail the notion of stream concurrency; equivalently we study what happens when the Greene-Kleitman shape of the Shi poset is a two-row rectangle. 

We begin by introducing the operations of \emph{shifting} streams and channels.

\begin{defn}
Suppose $S = \str{r}{A,B}$ is a stream. Then the \emph{shift} of $S$ by $k$ is the stream
$$S\sh{k}: = \str{r+k}{A,B}.$$
\end{defn}

\begin{defn}
Suppose $t$ is a permutation consisting of a single channel; say $t = \{(p_i, q_i):i\in\Z\}$ where for each $i$, $(p_i, q_i)$ is northwest of $(p_{i+1},q_{i+1})$. Then the \emph{shift} of $t$ by $k$ is the permutation $t\sh{k} = \{(p_i, q_{i+k}):i\in\Z\}$
\end{defn}

\begin{lemma}
\label{lem:backward inequalities}
Consider a partial permutation $t$ consisting of a single channel and a compatible stream $S$ whose flow is equal to the density of $t$. Let $t'=t\sh{1}$. Let $d = d_t^{\jbk,S}$ and $d' = d_{t'}^{\jbk,S}$. Then for every $b'\in\B_{t'}$, we have $d(b)\leqslant d'(b')\leqslant d(b'') = d(b) + 1$, where $b\in \B_t$ is the ball directly west of $b'$ and $b''\in \B_t$ is the ball directly south of $b'$.
\end{lemma}

\begin{figure}
\centering
\resizebox{.2\textwidth}{!}{\input{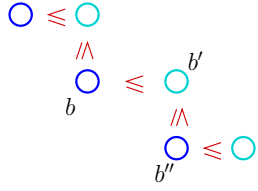}}
\caption{Inequalities for backward numberings of $t$ and $t'=t\sh{1}$ in Lemma \ref{lem:backward inequalities}. The balls of $t$ are dark blue while those for $t'$ are light blue.}
\label{fig:inequalities}
\end{figure}

\begin{proof} The situation is schematically illustrated in Figure \ref{fig:inequalities}. Let $\widetilde{d'}$ be the numbering of $t'$ which coincides with $d$ on rows (i.e. for every $b'\in\B_{t'}$, $\widetilde{d'}(b') = d(b)$). Then by Remark \ref{rmk:lower bound for backward numbering}, for every $b'\in\B_{t'}$, $\widetilde{d'}(b')\leqslant d'(b')$. This proves the first of the stated inequalities. The second one follows by the same argument. The equality $d(b'') = d(b) + 1$ follows from Proposition \ref{prop:backward numbers channel consecutively}.
\end{proof}

By the above lemma, $d$ and $d'$ coincide on either rows or columns. The next result shows that once they start coinciding on rows, they will keep doing so as we shift $t$ further.

\begin{lemma}
\label{lem:higher streams}
Consider a partial permutation $t$ consisting of a single channel and a compatible stream $S$ whose flow is equal to the density of $t$. Let $t'=t\sh{1}$, and $t'' = t'\sh{1} = t\sh{2}$. Let $d=d_t^{\jbk,S}$, and similarly for $d'$ and $d''$. Suppose that for some $b\in \B_t$, we have $d(b) = d'(b')$, where $b'\in\B_{t'}$ is directly east of $b$. Then $d'(b') = d''(b'')$, where $b''\in\B_{t''}$ is directly east of $b'$.
\end{lemma}

\begin{proof}
Note that if for some $b$ we have $d(b) = d'(b')$, then it is true for every $b$ by Proposition \ref{prop:backward numbers channel consecutively}. Choose $b$ so that $d'(b') = d'_0(b')$ (as in Remark \ref{rmk:backward equals stream}). Let $i = d'(b')$; the above choice ensures that $S^{(i+1)}$ does not lie northwest of $b'$. By Proposition \ref{prop:backward numbers channel consecutively}, the ball of $t$ directly south of $b'$ is numbered $i+1$, so $S^{(i+1)}$ must be southwest of $b'$. Hence $S^{(i+1)}$ is southwest of $b''$. Thus $d''(b'')\leqslant i$. By the Lemma \ref{lem:backward inequalities} and Proposition \ref{prop:backward numbers channel consecutively}, however, $i\leqslant d''(b'')\leqslant i+1$. Hence $d''(b'')= i$, as desired.
\end{proof}

\begin{rmk}
The above lemma can be reflected in the main diagonal. Let $t$ and $S$ and $d$ be as above; let $t'=t\sh{-1}$ and $t''=t\sh{-2}$ with backward numberings $d'$ and $d''$, respectively. Suppose for some $b\in\B_t$, $d(b) = d'(b')$, where $b'\in\B_{t'}$ is directly south of $b$. Then $d'(b') = d''(b'')$, where $b''\in\B_{t''}$ is directly south of $b'$.
\end{rmk}

\begin{rmk}
Consider the collection of streams $\st(A,B)=\left\{\str{r}{A,B}\right\}_{r\in\Z}$. For $r\in\Z$ let $t_r$ be the partial permutation associated to $\str{r}{A,B}$. Fix a stream $S$ compatible with some (equivalently, every) $t_r$. From Lemma \ref{lem:higher streams} and its reflection we conclude that there exists a unique $r_0 = r_0(S, A, B)$ such that for any $k>0$ the backward numberings of $t_{r_0}$ and $t_{r_0+k}$ coincide on rows while the backward numberings of $t_{r_0}$ and $t_{r_0-k}$ coincide on columns. 
\end{rmk}

The above remark allows us to extend the notion of altitude to streams of the same flow but in different classes.

\begin{defn}
\label{def:no lower}
Suppose $S$ and $T\in\st(A,B)$ are streams of the same flow and $S$ is compatible with $T$. Then $T$ is \emph{no lower than} $S$ if $a(T)\geqslant r_0$ (recall Definition \ref{def:altitude}), where $r_0$ was defined in the above remark. Similarly, we can define when $T$ is \emph{no higher than} $S$. If $T=\str{r_0}{A,B}$, we say that $T$ \emph{is at the same height as} $S$.
\end{defn}

In the next series of results we prove that the notion of $T$ having the same height as $S$ is precisely the notion of $T$ being concurrent to $S$ from Definition \ref{def:concurrent}. In particular there is the same terminological problem as with the definition of concurrency; namely that the definition is not symmetric and it can happen that $S$ is at the same height as $T$ but $T$ is not at the same height as $S$. We begin by proving a lemma necessary to make sense of the definition of concurrency, namely that $\bk{S}{t}$ partitions into two disjoint channels.

\begin{proof}[Proof of Lemma \ref{lem:pre-concurrent}]
\label{pf:pre-concurrent}
By Corollary \ref{cor:with of backward Shi poset}, the width of the Shi poset of $\bk{S}{t}$ is the flow of $S$. By Proposition \ref{prop:backward numbers channel consecutively}, each of the zig-zags involved in the step of the backward algorithm will have a single outer corner-post and two inner corner-posts. By Lemma \ref{lem:zigzags dont intersect}, the southwest (resp. northeast) corner-posts form a chain with respect to $\leqslant_{SE}$, and this chain has the correct density to form a channel.
\end{proof}

\begin{lemma}
\label{lem:sw iff higher}
Let $T$ be a stream with corresponding permutation $t$. Consider a compatible stream $S$ with flow equal to that of $T$. Let $\widetilde{d}$ be the numbering induced by the backward step. Then $\widetilde{d}$ is the southwest channel numbering if any only if $T$ is no lower than $S$. 
\end{lemma}
\begin{proof}
By Lemma \ref{lem:pre-concurrent}, $\bk{S}{t}$ consists of a southwest channel, call it $C_1$, and a northeast channel, $C_2$. Let $t' = t\sh{1}$. Let $d$ (resp. $d'$) be the backward numbering of $t$ (resp. $t'$). Similarly, let $d_0$ (resp. $d_0'$) be the stream numberings. 

Suppose $T$ is no lower than $S$, i.e. $d$ and $d'$ coincide on rows. By Remark \ref{rmk:backward equals stream}, there exists $b'\in\B_{t'}$ satisfying $d'(b') = d_0'(b')$; let $i = d'(b')$. Then $S^{(i+1)}$ cannot lie northwest of $b'$ and must lie northwest of the ball of $t$ directly south of $b'$. So $S^{(i+1)}$ lies southwest of $b'$. The ball of $C_1$ numbered $i$ (by $\widetilde{d}$) lies directly west of $b'$. The ball of $C_2$ labeled $i+1$ lies directly east of $S^{(i+1)}$ and directly south of $b'$. Thus there exists a path (with one step) from $C_2$ to $C_1$ such that $\widetilde{d}$ decreases by $1$. By Remark \ref{rmk:monotone is sometimes channel}, is the southwest channel numbering.

The other implication involves the same steps in reverse order.
\end{proof}

Now we are ready to show that the notions of having the same height and concurrency are, in fact, the same. Moreover, we will conclude that concurrency is preserved by shifting the streams simultaneously. These statements comprise the main result in Section \ref{sec:concurrent}. 

\begin{proof}[Proof of Proposition \ref{prop:unique concurrent}]
\label{pf:unique concurrent}
Recall the notation; for sets $A, B, A' ,B'$ of the same size, we need to show that there exists a unique $r$ such that $T:=\str{r}{A',B'}$ is concurrent to $S:=\str{0}{A,B}$. This follows immediately, with $r = r_0$, from the previous lemma and its reflection in the main diagonal. In particular it shows that $T$ is concurrent to $S$ if and only if $T$ is at the same height as $S$.

Let us now show that $T':=\str{r+1}{A',B'}$ is concurrent to $S':=\str{1}{A,B}$; repeating the argument and reflecting about the main diagonal will finish the proof. Let $t$ (resp. $t'$) be the partial permutation corresponding to $T$ (resp. $T'$). Consider $S'$ with the proper numbering which matches that of $S$ on rows. Let $d$ (resp. $d'$) be the backward numbering of $t$ (resp. $t'$) with respect to $S$ (resp. $S'$). Let $\widetilde{d'}$ be the numbering of $t'$ which coincides with $d$ on rows. 

We will now show that $d' = \widetilde{d'}$. By Remark \ref{rmk:lower bound for backward numbering}, for every $b'\in\B_{t'}$, $\widetilde{d'}\leqslant d'$.  Since $T$ is no higher than $S$, $\widetilde{d'} = d_{t'}^{\jbk,S}$ (not $S'$!).
By Remark \ref{rmk:backward equals stream}, there exists $b'\in \B_{t'}$ with $\widetilde{d'}(b')$ equal to the stream numbering (with respect to $S$). As in the proof of Lemma \ref{lem:sw iff higher}, we conclude that $S^{(\widetilde{d'}(b')+1)}$ lies south of $b'$. By construction, $S'^{(\widetilde{d'}(b')+1)}$ also lies south of $b'$. So $d'(b')\leqslant \widetilde{d'}(b')$. Hence $d'(b')=\widetilde{d'}(b')$, and since both numberings number $t'$ by consecutive integers, they must always coincide.

As noted above, $b'$ lies north of $S'^{(d'(b')+1)}$, so $T'$ is no lower than $S'$. Similarly, let $b\in \B_t$ be the ball such that $S^{(d(b)+1)}$ is east of $b$. Then $b''\in\B_{t'}$ which lies directly north of $b$ has $d'(b'') = d(b)-1$ and $S'^{(d(b))}$ is east of it. Hence $T'$ is at the same height as $S'$. As mentioned in the first paragraph, this means $T'$ is concurrent to $S'$.
\end{proof}

\section{Weights}
\label{sec:weights}

\subsection{A key lemma}

In this section we will present a key (trivial) lemma, and then give examples of constructions resulting from it and several applications. The constructions will be used over and over again in the rest of the paper. 

\begin{lemma}
\label{lem: key lemma}
Consider a cell $c$ and two cells $b_1, b_2$ such that $b_1$ is strictly west (and possibly south or north) of $c$, $b_2$ is strictly north (and possibly east or west) of $c$, and $b_1$ is southwest of $b_2$. Consider a forward zig-zag $Z$ (recall Definition \ref{def:zig-zag}) from $b_1$ to $b_2$ which has no cell southeast of $c$. Then there exists an outer corner-post $c'$ of $Z$ such that $c'$ is northwest of $c$. The situation is illustrated on the left side of Figure \ref{fig:positioning}.
\end{lemma}

\begin{figure}
\centering
\resizebox{.4\textwidth}{!}{\input{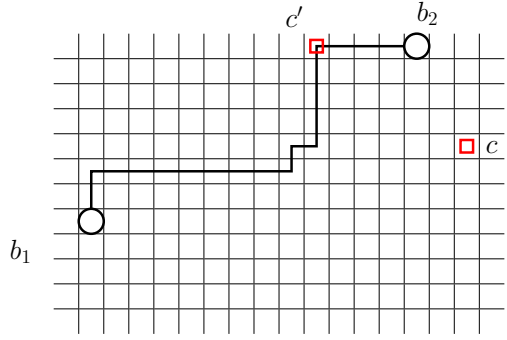}}
\caption{Cell positioning in Lemma \ref{lem: key lemma}.}
\label{fig:positioning}
\end{figure}

\begin{proof}
Let $Z=(z_1,\ldots, z_k)$. Since $z_1 = b_1$, we know that it is strictly west of $c$. Choose the largest $l$ such that $z_l$ is strictly west of $c$. By the assumptions, $z_l$ must be strictly north of $c$. By definition of outer corners, there must be an outer corner of $Z$ northwest of $z_l$.
\end{proof}

Of course, there is a version of this lemma for reverse zig-zags; it is obtained by reflection in an anti-diagonal.

\begin{prop}
Consider a partial permutation $w$, a proper numbering $d$, and the corresponding collection $\{Z_j\}_{j\in\Z}$ of zig-zags. Suppose we have two paths on the set of balls of $\fwC{d}{w}$: $p=(b_0, b_1, \dots, b_k)$ and $q=(c_0, c_1, \dots, c_k = b_k)$. Moreover assume that for all $i$, $b_i$ lies southwest of $c_i$, and they both belong to the zig-zag $Z_{j_i}$. Then for any ball $a_0\in Z_{j_0}$ of $w$ between $b_0$ and $c_0$, there exists a path $(a_0,a_1,\dots, a_{k-1})$ such that $a_i\in Z_{j_i}$, and $a_i$ is between $b_i$ and $c_i$. 
An example is shown on the left hand side of Figure \ref{fig:funnel walk}.
\end{prop}
\begin{proof}
This statement follows by a repeated application of Lemma \ref{lem: key lemma}.
\end{proof}

\begin{defn}
The path $(a_0,a_1,\dots, a_{k-1})$ in the above proposition will be called a \emph{funnel walk}.  The proposition may be reflected in an anti-diagonal; in this case the resulting path is called a \emph{reverse funnel walk} (as example is shown in the right hand side of Figure \ref{fig:funnel walk}).
\end{defn}

\begin{figure}
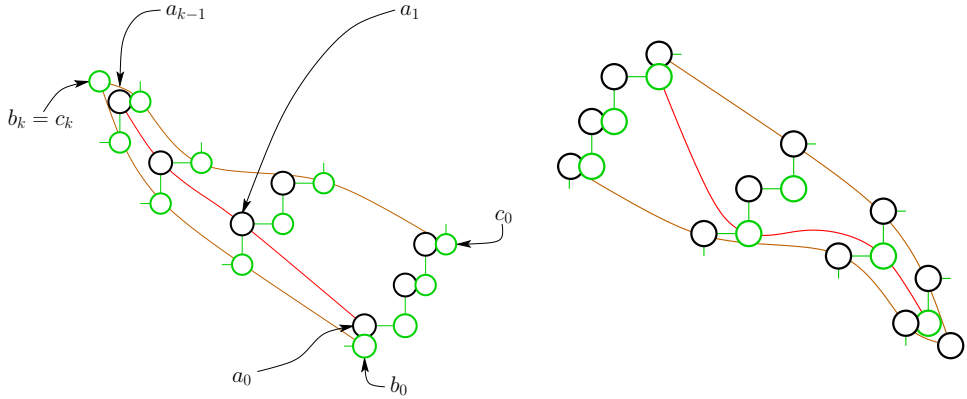

\centering
\resizebox{.4\textwidth}{!}{\input{figures/funnel_walk_ng.pspdftex}}\qquad\raisebox{.032\textwidth}{\resizebox{.329\textwidth}{!}{\input{figures/funnel_walk_rev_ng.pspdftex}}}
\caption{A funnel walk and a reverse funnel walk. The balls of $w$ are shown in black; the balls of $\fw{w}$ are shown in green. The paths $p$ and $q$ are shown with brown curves. The walk $(a_0,a_1,\dots, a_{k-1})$ is indicated by the red line.}
\label{fig:funnel walk}
\end{figure}

\begin{prop}
Consider a partial permutation $w$, a proper numbering $d$, and the corresponding collection $\{Z_j\}_{j\in\Z}$ of zig-zags. Suppose we have two infinite paths on the set of balls of $\fwC{d}{w}$: $p=(b_0, b_1, \dots)$ and $q=(c_0, c_1, \dots)$. Moreover assume that for all $i$, $b_i$ lies southwest of $c_i$, and they both belong to the zig-zag $Z_{j_i}$. Then for any ball $a_0\in Z_{j_0}$ of $w$ between $b_0$ and $c_0$, there exists a path $(a_0,a_1,\dots)$ such that $a_i\in Z_{j_i}$, and $a_i$ is between $b_i$ and $c_i$. 
\end{prop}
\begin{proof}
This statement follows by a repeated application of Lemma \ref{lem: key lemma}.
\end{proof}

\begin{defn}
The path $(a_0,a_1,\dots)$ in the above proposition will be called a \emph{bounded walk}.  The proposition may be reflected in an anti-diagonal; in this case the resulting path is called a \emph{reverse bounded walk}.
\end{defn}

\begin{prop}
\label{prop:semibounded}
Consider a partial permutation $w$, a proper numbering $d$, and the corresponding collection $\{Z_j\}_{j\in\Z}$ of zig-zags. Suppose we have a (finite or infinite) path on the set of balls of $\fwC{d}{w}$: $p=(b_0, b_1, \dots)$ and that for each $i$, $b_i\in Z_{j_i}$. Suppose $a_0\in Z_{j_0}\cap\B_w$. If $a_0$ is northeast (resp. southwest) of $b_0$ then there exists a path $(a_0,a_1,\ldots)$ with $a_i\in Z_{j_i}\cap\B_w$ such that each, for every $i$,  $a_i$ is northeast (resp. southwest)  of $b_i$. An example is shown in Figure \ref{fig:semi-bounded walk}.
\end{prop}
\begin{proof}
Without loss of generality, $a_0$ is northeast of $b_0$. Since $d$ is proper, there exist balls of $w$ northwest of $a_0$ in the zig-zag of $b_1$. We show that at least one of them is northeast of $b_1$. The ball $b_1$ is west of $b_0$ and hence west of $a_0$. If $b_1$ is north of $a_0$ then the ball of $w$ directly north of $b_1$ is also northwest of $a_0$. If $b_1$ is south of $a_0$ then any ball of its zig-zag which is northwest of $a_0$ is also northeast of $b_1$. Repeating the argument gives the sought path $(a_0, a_1, \dots)$ northeast of $p$. 
\end{proof}

\begin{rmk}
\label{rmk:semi-bounded walk}
Unlike with the other two constructions, this one does not have the antidiagonal symmetry, due to the fact that the ends of the zig-zags are always balls of $w$, not $\fwC{d}{w}$. The difference is that the reverse version is sometimes forced to stop; it is not always possible to find the next element. However (for future applications) we can analyze when this happens.

Consider a (finite or infinite) reverse path $p = (c_0,c_1,\dots)$ on the balls of $w$ with $c_i\in\Z_{j_i}$. Choose some $a_0\in\B_{\fwC{d}{w}}\cap Z_{j_0}$. Without loss of generality, $a_0$ is northeast of $c_0$. The argument from the previous paragraph can fail in two ways. First, there do not have to be balls of $\fwC{d}{w}$ in the zig-zag of $c_1$ southeast of $a_0$. Second, even if there is such a ball, we cannot repeat the argument since there may not be a ball of $\fwC{d}{w}$ directly east of $c_1$. Thus (in either case) if we could not find the next element $a_1$, then $Z_{j_1}$ lies completely south of $a_0$. An example is shown on the right in Figure \ref{fig:semi-bounded walk}.
\end{rmk}

\begin{defn}
\label{def:semi bounded walk}
We will refer to a path as in Proposition \ref{prop:semibounded} as a \emph{semi-bounded walk}. The version reflected in the antidiagonal, which may be shorter than the original path, will be referred to as a \emph{reverse semi-bounded walk}.
\end{defn}

\begin{figure}
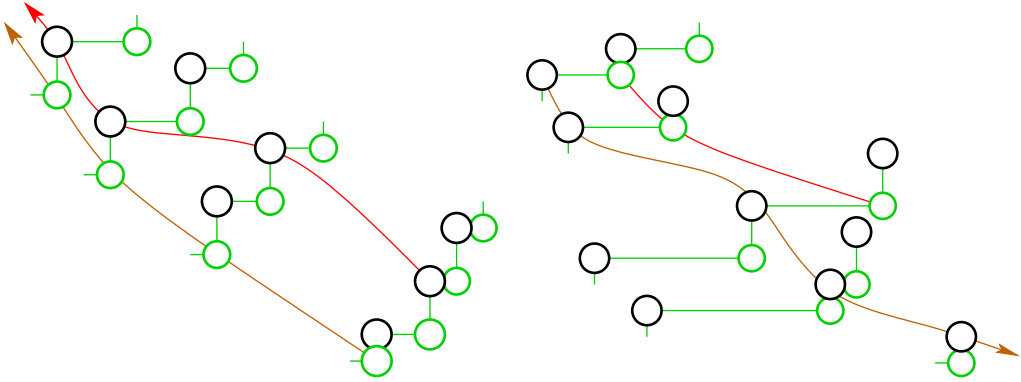

\centering
\resizebox{.4\textwidth}{!}{\input{figures/semi_bounded_walk_ng.pspdftex}}\quad\resizebox{.4\textwidth}{!}{\input{figures/semi_bounded_walk_rev_ng.pspdftex}}
\caption{A semi-bounded walk and a reverse semi-bounded walk. The latter is, in this case, forced to stop after two steps.}
\label{fig:semi-bounded walk}
\end{figure}

\subsection{Applications to channels}
\label{sec:apps to channels}

In this section we apply the above constructions to prove some basic results about channels. The first several results describe the relative position of channels of a partial permutation $w$ and the partial permutations $\fwC{d}{w}$ and $\bk{S}{w}$.

\begin{lemma}
\label{lem:width decreases}
Suppose $w$ is a partial permutation and $d$ is a proper numbering. Then the width of $P_w$ is at least as large as the width of $P_{\fwC{d}{w}}$.  
\end{lemma}
\begin{proof}
Let $\{Z_i\}_{i\in\Z}$ be the relevant collection of zig-zags. Consider a channel $D$ of $\fwC{d}{w}$. If the width of $P_{\fwC{d}{w}}$ were larger, then by the pigeonhole principle, for $b\in D$, between $b$ and $b+(n,n)$ at least two balls of $D$ would be part of the same zig-zag. This is impossible.
\end{proof}

Now it is easy to see that AMBC produces a valid element of $\Omega$.
\begin{proof}[Proof of Theorem \ref{thm:AMBC produces valid element}]\label{pf:AMBC produces valid element}  It is clear that two tabloids have the same shape and are filled with the distinct residue classes. By the above lemma, the row sizes decrease. 
\end{proof}

\begin{prop}
\label{prop: between channels a channel}
Let $w$ be a partial permutation, and $d$ a proper numbering. 
\begin{enumerate}
\item\label{first} For any pair of channels $C_1,C_2\in\C_w$ with $C_1$ southwest of $C_2$, there exists a channel $D\in\C_{\fwC{d}{w}}$ such that each ball of $D$ is located between the ball of $C_1$ and the ball of $C_2$ on its zig-zag. 
\item\label{second} If $P_{\fwC{d}{w}}$ and $P_w$ have the same width, then for any $D\in\C_{\fwC{d}{w}}$ there exists $C_1\in\C_w$ southwest of it and $C_2\in\C_w$ northeast of it.
\end{enumerate}
\end{prop}

\begin{proof}
(\ref{first}) For each $i$, let $b_i\in C_1$ (resp. $c_i\in C_2$) be the ball of $C_1$ (resp. $C_2$) in the $Z_i$. Choose $a_0\in\B_{\fwC{d}{w}}\cap Z_0$ which lies between $b_0$ and $c_0$. Let $(a_0, a_1,\dots)$ be the corresponding bounded reverse walk.

By the pigeonhole principle, this reverse walk contains two $(n,n)$-translates of the same ball. Up to reindexing, we can assume that for some $l$, $a_0$ and $a_{l-1}$ are $(n,n)$-translates. Consider the partial permutation $w'$ whose balls are $\{a_0,\dots, a_{l-1}\}$ as well as all their $(n,n)$-translates. Let $d'$ be the numbering of $w'$ where a ball is numbered $i$ if it is in $Z_i$. Then $d'$ is obviously a proper numbering, hence by Proposition \ref{prop: period of proper numbering} the widths $P_w$ and $P_{w'}$ must be the same. Now by Lemma \ref{lem:width decreases}, the width of $P_{\fwC{d}{w}}$ cannot be greater than that of $w$, so $P_{w'}$ must contain a maximal antichain of $P_{\fwC{d}{w}}$. The corresponding channel is the sought channel $D$. 

(\ref{second}) Start with a ball southwest of $D$ and a ball northeast of $D$. The semi-bounded (by $D$) walks starting at these balls arrive at channels $C_1$ and $C_2$ by the same argument as before. 
\end{proof}

\begin{cor}
\label{cor:interlacing channels}
Suppose $w$ is a partial permutation and $P_w$ has $k\geqslant 2$ disjoint maximal antichains; let $d$ be a proper numbering. Let $\{C_1,\dots, C_k\}$ be the largest disjoint collection of channels such that $C_i$ is southwest of $C_{i+1}$ for all $i$. For each $i$ choose a channel $D_i$ of $\fwC{d}{w}$ between $C_i$ and $C_{i+1}$ as done in Proposition \ref{prop: between channels a channel}. Then the channels $D_1,\dots, D_{k-1}$ form a largest disjoint collection of channels of $\fwC{d}{w}$.
\end{cor}
\begin{proof}
It is clear that the channels $D_1,\dots, D_{k-1}$ are pairwise disjoint. If there were more than $k-1$ pairwise disjoint channels of $\fwC{d}{w}$, then by Proposition \ref{prop: between channels a channel}, there would be more than $k$ disjoint channels for $w$.
\end{proof}

The next result is a similar statement for a step of the backward algorithm.

\begin{prop}
Suppose $w$ is a partial permutation and $S$ is a compatible stream such that the width of $P_w$ is equal to the flow of $S$. Choose a maximal disjoint collection  $\{D_1,\dots, D_{k-1}\}$ of channels of $w$ such that for any $i$, $D_i$ is southwest of $D_{i+1}$. Then there exists a channel $C_1$ of $\bk{S}{w}$ whose cells are southwest of the cells of $D_1$, a channel $C_k$ whose cells are northeast of the cells of $D_{k-1}$, and for each $2\leqslant i\leqslant k-1$ a channel $C_i$ between $D_{i-1}$ and $D_i$. The channels $C_1,\dots, C_k$ form a maximal disjoint collection of channels of $\bk{S}{w}$.
\end{prop}

\begin{proof}
There exists a ball $b$ of $\bk{S}{w}$ directly west of a ball of $D_1$. By Lemma \ref{lem:backward is proper}, the numbering of $\bk{S}{w}$ induced from the backward step is proper, so we may consider a semi-bounded walk starting at $b$ with respect to $D_1$. By the same argument as in Proposition \ref{prop: between channels a channel}, this walk arrives at a channel which we will call $C_1$ (by construction, it satisfies the desired property). By the same argument we can construct $C_k$; for $2\leqslant i\leqslant k-1$, $C_i$ is obtained by considering a bounded walk between $D_{i-1}$ and $D_i$. It is clear that these form a maximal collection of channels of $\bk{S}{w}$ since between any two channels of $\bk{S}{w}$ there is a channel of $w$. 
\end{proof}

\begin{rmk}
\label{rmk: can get to specific channel}
Suppose $w$ is a partial permutation and $P_w$ has $k\geqslant 2$ disjoint maximal antichains. Let $(C_1,\dots, C_k)$ and $(D_1,\dots,D_{k-1})$ be interlacing collections of channels as in Corollary \ref{cor:interlacing channels}. We know that for any $i$, a bounded walk $p$ between $D_i$ and $D_{i+1}$ necessarily intersects some channel (which lies completely between $D_i$ and $D_{i+1}$). It turns out that $p$ must intersect specifically our chosen channel $C_{i+1}$.

Indeed, since $p$ is an infinite path, two elements of $p$ must be translates. However $p$ contains an element in a consecutive collection of zig-zags, so by Lemma \ref{lem:high flow implies channel}, a subcollection of elements of $p$ together with their translates must form a channel $C$. Of course, $C$ must intersect $C_{i+1}$ since otherwise there would be too many disjoint channels of $w$. 

Exactly the same story holds for bounded reverse walks. Similarly, starting with a ball of $\fw{w}$ southwest (resp. northeast) of $D_1$ (resp. $D_{k-1}$), a semi-bounded walk intersects $C_1$ (resp. $C_k$). 
\end{rmk}

\subsection{Dominance in the image of $\Phi$}

 The main goal of this section is to prove that the image of $\Phi$ is  subset of $\Omega_{dom}$.  Before getting to the proof, we study two monotone numberings of the balls of $\fw{w}$. One is its own southwest channel numbering $d_{\fw{w}}^{SW}$. The second one is the numbering afforded by the zig-zags from $d_w^{SW}$ as follows. 

\begin{defn}
Suppose $d$ is a proper numbering of $w$. Define $\fw{d}:\B_{\fwC{d}{w}}\to\Z$ by $\left(\fw{d}\right)(b) = i$ if $b$ is part of the $i$-th zig-zag of $d$. 
\end{defn}

\begin{rmk}
\label{rmk:induced is monotone but not proper}
The two numberings $d_{\fw{w}}^{SW}$ and $\fw{d_w^{SW}}$ do not always coincide; moreover the numbering $\fw{d_w^{SW}}$ is not always proper. However, it is not difficult to see that $\fw{d}$ is monotone for every proper numbering $d$.
\end{rmk}

The next result demonstrates that, although the two numberings do not coincide everywhere, they do coincide on channels.
 
\begin{prop}
\label{prop:induced and southwest coincide on channels}
Consider a partial permutation $w$ such that $P_w$ has at least two disjoint maximal antichains. Fix some shift of $d_w^{SW}$ and choose the shift of $d_{\fw{w}}^{SW}$ to coincide with $\fw{d_w^{SW}}$ on the southwest channel of $\fw{w}$. Then the two numberings coincide on all the channels of $\fw{w}$. Moreover, for any $b\in\B_{\fw{w}}$ we have $\left(\fw{d_w^{SW}}\right)(b) \geqslant d_{\fw{w}}^{SW}(b)$.
\end{prop}

\begin{proof}
Consider interlacing maximal collections of channels $\{C_1,\dots, C_k\}$ of $w$ and $\{D_1,\dots, D_{k-1}\}$ of $\fw{w}$ as in Corollary \ref{cor:interlacing channels}. Since any channel of $\fw{w}$ intersects some $D_i$, and both numberings number channels by consecutive integers, it is sufficient to show that $\left(\fw{d_w^{SW}}\right)(b) = d_{\fw{w}}^{SW}(b)$ for every $b\in D_1\cup\dots\cup D_{k-1}$.

As mentioned in Remark \ref{rmk:induced is monotone but not proper}, $\fw{d_w^{SW}}$ is monotone. By construction, $D_1$ intersects the southwest channel of $\fw{w}$. By Remark \ref{rmk:monotone is sometimes channel}, to check whether $\fw{d_w^{SW}}$ coincides with $d_{\fw{w}}^{SW}$ on a certain ball $b$, it is sufficient to check that there is a path $(b=b_0, b_1,\dots, b_l)$ to $D_1$ such that $\fw{d_w^{SW}}(b_{i+1}) = \fw{d_w^{SW}}(b_i) - 1$. It remains to construct such a path from any $b\in D_1\cup\dots\cup D_{k-1}$. For every $i$, we will construct such a path from $D_i$ to $D_{i-1}$. We can then get from $D_i$ to $D_1$ as follows. Take the path from $D_i$ to $D_{i-1}$, then follow $D_{i-1}$ to the start of the closest translate of the path to $D_{i-2}$, and repeat until arriving at $D_1$. 

It is easier to construct the path in reverse; i.e. from $D_{i-1}$ to $D_i$. The construction of the reverse path will proceed in two parts. The first part will use a reverse funnel walk whose purpose is to get a reverse path to some point above $C_{i-1}$ (it is illustrated in Figure \ref{fig:path from di to di-1}). The second part is a bounded reverse walk from a ball of $\fw{w}$ between $C_{i-1}$ and $C_i$ to $D_i$.

By (the transpose of) Lemma \ref{lem:channel numbering determined by closest channel}, $d_w^{SW}$ and $d_w^{C_{i-2}}$ coincide on $C_{i-1}$ (where the shift of $d_w^{C_{i-2}}$ is chosen to coincide with $d_w^{SW}$ on $C_{i-2}$). Consider a path $p = (c_0,\ldots, c_h)$ from some $c_0\in C_{i-1}$ to some $c_h\in C_{i-2}$ such that $d_w^{C_{i-2}}$ decreases by $1$ at each step; since $d_w^{SW}$ and $d_w^{C_{i-2}}$ coincide on both $c_0$ and $c_h$, $d_w^{SW}$ must also decrease by $1$ with each step of $p$.

Consider the collection $\{Z_i\}_{i\in\Z}$ of zig-zags corresponding to $d_w^{SW}$. Since the overall shift of all numberings is irrelevant, we may for notational convenience assume that $c_h\in Z_0$, i.e. that $d_w^{SW}(c_h) = 0$. Start with the ball $b'_0\in D_{i-1}\cap Z_0$ which is located between $C_{i-1}$ and $C_{i-2}$. Thus $\fw{d_w^{SW}}(b'_0) = 0$. Let $(b'_0,\dots,b'_{h-1})$ be the reverse funnel walk between $p$ and $C_{i-1}$. Since the walk visits consecutive zig-zags, $\fw{d_w^{SW}}$ increases by $1$ with every step. Let $b'_h\in Z_h$ be the ball of $\fw{w}$ directly east of $c_0$. Let $(b'_h,b'_{h+1},\dots, b'_{h'})$ be a reverse bounded walk between $C_{i-1}$ and $C_i$ from $b'_h$  to $D_i$ (it exists by Remark \ref{rmk: can get to specific channel}). Since this walk visit consecutive zig-zags, $\fw{d_w^{SW}}$ again increases by $1$ at every step. Thus we have constructed a path with the desired properties, finishing the proof that the two numberings coincide on channels.

Observe that at each step of a path of maximal worth from any ball of $\fw{w}$ to the southwest channel, $d_{\fw{w}}^{SW}$ increases by 1 while $\fw{d_w^{SW}}$ increases by at least 1. This demonstrates the last claim.
\end{proof}

\begin{figure}
\centering
\resizebox{.6\textwidth}{!}{\input{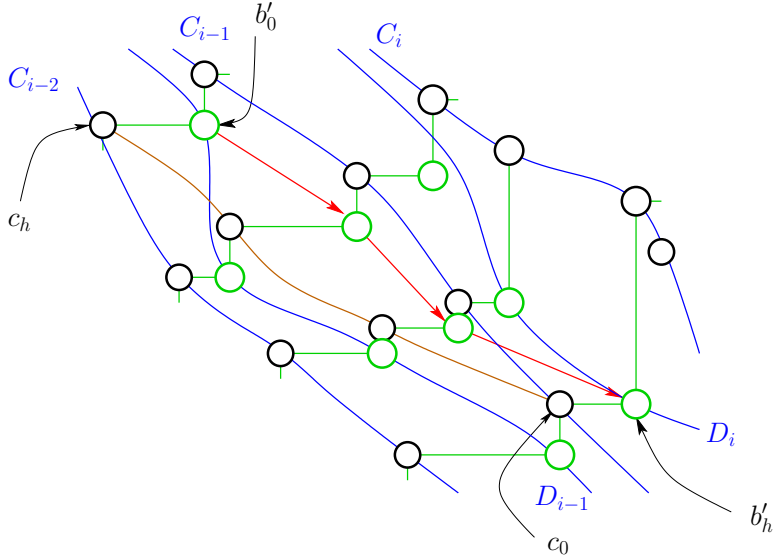}}
\caption{A reverse path from $D_{i-1}$ to $D_{i}$. The balls of $w$ are in black, while the balls of $\fw{w}$ are in green. The channels are indicated by thin blue lines. The path from $C_{i-1}$ to $C_{i-2}$ is indicated by the brown line.}
\label{fig:path from di to di-1}
\end{figure}

We proceed to a technical lemma, which handles the crux of the proof that the image of $\Phi$ is a subset of $\Omega_{dom}$. Recall that whenever we have a stream $S$ and a proper numbering of it, then $S^{(i)}$ refers to the cell of the stream numbered $i$. 

\begin{lemma}
\label{lem: plus above next star}
Let $w$ be a partial permutation such that $P_w$ has $k\geqslant 2$ maximal antichains.  Let $S = \st(w)$, and let $T = \st(\fw{w})$. Fix a shift of $d_w^{SW}$ and consider the numbering $d_{\fw{w}}^{SW}$ which coincides with $\fw{d_w^{SW}}$ on the channels of $\fw{w}$. The stream $S$ (resp. $T$) has a natural numbering $d_S:S\to\Z$ (resp. $d_T:T\to\Z$) coming from $d_w^{SW}$ (resp. $d_\fw{w}^{SW}$). In this case for some $i$, we have $T^{(i)}$ lies north of $S^{(i+1)}$.
\end{lemma}
\begin{proof}
Pick interlacing collections of channels $\{C_1,\dots, C_k\}$ of $w$ and $\{D_1,\dots, D_{k-1}\}$ of $\fw{w}$ as before.
Consider a path $p=(c_0,\dots, c_l)$ of $w$ from $C_k$ to $C_1$ such that $d_w^{SW}$ decreases by $1$ at each step. Start with $b_0\in D_1$ on the same zig-zag as $c_l$. Construct a reverse funnel walk $(b_0,\dots, b_{l-1})$ in $\fw{w}$ between $p$ and $C_k$. Since $b_0$ belongs to a channel of $\fw{w}$, $\left(\fw{d_w^{SW}}\right)(b_0) = d_\fw{w}^{SW}(b_0)$.  Now $d_\fw{w}^{SW}$ increases by $1$ with each step of $p$, and, by Proposition \ref{prop:induced and southwest coincide on channels}, so $\left(\fw{d_w^{SW}}\right)(b_i) = d_\fw{w}^{SW}(b_i)$ for $0\leqslant i\leqslant l-1$. 

Consider the collection $\{Z_j\}_{j\in\Z}$ of zig-zags corresponding to $d_w^{SW}$. Let $i= d_\fw{w}^{SW}(b_{l-1})$. Suppose $Z_{i+1}$does not extend north of $b_{l-1}$. In this case we are done since $T^{(i)}$ is north of $b_{l-1}$ while $S^{(i+1)}$ is in the northern row of $Z_{i+1}$, which is south of $b_{l-1}$. Instead, assume the $Z_{i+1}$ does extend above $b_{l-1}$. Choose $b_l$ to be the ball of $\fw{w}$ directly east of $c_0$ (recall $c_0\in Z_{i+1}$ and is south of $b_{l-1}$; hence it must not be the northeastern cell of $Z_{i+1}$). Now consider the semi-bounded reverse walk $(b_l,\dots, b_{l+l'})$ with respect to $C_k$. It must terminate since, by Corollary \ref{cor:interlacing channels}, there is no channel of $\fw{w}$ northwest of $C_k$. By the same argument as in the end of the previous paragraph, $\left(\fw{d_w^{SW}}\right)(b_{l+l'})=d_\fw{w}^{SW}(b_{l+l'})=i+l'$. The fact that the walk terminated implies that the $Z_{i+l'+1}$ lies south of $b_{l+l'}$ (as described in Remark \ref{rmk:semi-bounded walk}). Then $T^{(i+l')}$ lies north of $S^{(i+l'+1)}$.
\end{proof}

Now we are ready to prove that $\Phi(\widetilde{W}) \subseteq\Omega_{dom}$.
\begin{proof}[Proof of Theorem \ref{thm:image1}]
\label{pf:image1}
Of course it is sufficient to prove that dominance is preserved by every step of AMBC. Let $w$ be a partial permutation. We can assume that $P_w$ has at least two disjoint maximal antichains; otherwise dominance does not impose any restriction. Let $S=\st(w)$ and $T= \st(\fw{w})$; we are trying to show that $T$ is no lower than $S$. Consider the numberings $\fw{d_w^{SW}}$ and $d_\fw{w}^{SW}$ of $\fw{w}$ (chosen to coincide on the channels). The cells of $S$ (resp. $T$) have a natural numbering $d_S:S\to\Z$ (resp. $d_T:T\to\Z$) on coming from $d_w^{SW}$ (resp. $d_\fw{w}^{SW}$). Let $\{Z_i\}_{i\in\Z}$ be the collection of zig-zags corresponding to $d_w^{SW}$.

By Proposition \ref{prop:induced and southwest coincide on channels}, for any $b\in\B_w$, we have $\left(\fw{d_w^{SW}}\right)(b)\geqslant d_\fw{w}^{SW}(b)$. This implies that the southwest (resp. northeast) ball of $\fw{w}$ labeled $i$ by $d_\fw{w}^{SW}$ is part of $Z_j$ for $j\geqslant i$, and hence lies east (resp. south) of $S^{(i)}$. Hence for any $i$, $T^{(i)}$ lies southeast of $S^{(i)}$.

Let $t$ be the partial permutation corresponding to $T$ (as in Section \ref{subsec:two streams}); we now show that $d_T$, thought of as a numbering of $t$, is precisely $d_t^{\jbk,S}$. Since $d_T$ is monotone and $T^{(i)}$ lies southeast of $S^{(i)}$, we know that $d_T(b)\leqslant d_t^{\jbk,S}(b)$ (recall Remark \ref{rmk:lower bound for backward numbering}). By Lemma \ref{lem: plus above next star}, there exists $i$ such that $b_0 = T^{(i)}$ lies north of $S^{(i+1)}$. Hence $d_t^{\jbk,S}(b_0) \leqslant i = d_T(b_0)$; hence $d_t^{\jbk,S} = d_T(b_0)$. However they both number the balls by consecutive integers, so they coincide everywhere.

Now $b_0$ lies north of $S^{(i+1)}$. The ball of the southwest channel of $\bk{S}{t}$ (recall Lemma \ref{lem:pre-concurrent} and its proof on page \pageref{pf:pre-concurrent}) numbered $i$ by the induced numbering lies directly west of $b_0$, while the ball of the northeast channel numbered $i+1$ lies directly east of $S^{(i+1)}$. Thus there is a path from the northeast channel to the southwest channel such that the induced numbering decreases by $1$ with each step. Thus the induced numbering is the southwest channel numbering. By Lemma \ref{lem:sw iff higher}, $T$ is no lower than $S$, as desired.
\end{proof}

\subsection{Proof of bijectivity}

In the first part of this section we will develop the theory to show that for any $w$, $\Psi(\Phi(w)) = w$.

\begin{thm}
\label{thm:forward backward pretheorem}
Consider a partial permutation $w$ and $C\in\C_w$. Let $S = \str{C}{w}$. Then $\fw{d_w^C} = d_\fwC{C}{w}^{\jbk,S}$.
\end{thm}

Now we will work toward the proof of Theorem \ref{thm:forward backward pretheorem}.

\begin{defn}
\label{defn:terminal and terminating}
Consider a partial permutation $w$, a proper numbering $d$, and the corresponding collection of zig-zags $\{Z_i\}_{i\in\Z}$. A ball $b\in\B_{\fwC{d}{w}}$ on $Z_i$ is \emph{$Nr$-terminal} (resp. \emph{$Wr$-terminal}) if $Z_{i+1}$ is completely south (resp. east) of $b$. 

A ball $b\in\B_{\fwC{d}{w}}$ on $Z_i$ is \emph{$Nr$-terminating} if there exists a reverse path $(b=b_0, b_1, \dots, b_k)$ such that $b_l$ is on $Z_{i+l}$ and $b_k$ is $Nr$-terminal. So, an $Nr$-terminal ball is always $Nr$-terminating. Define \emph{$Wr$-terminating} analogously.  
\end{defn}

\begin{lemma}
\label{lem:all balls are terminating}
Consider a partial permutation $w$ and $C\in\C_w$. Then every ball of $\fwC{C}{w}$ is either $Nr$-terminating or $Wr$-terminating (or both).
\end{lemma}
\begin{proof}
Let $\{Z_i\}_{i\in\Z}$ be the collection of zig-zags associated with $d_w^C$. Choose a zig-zag; let $c$ be its southwestern cell and $c'$ be its northeastern cell. From both $c$ and $c'$ we can take paths (on the balls of $w$) of maximal worth to $C$. Extending the shorter of these along the channel, yields paths $p=(c_0=c, c_1, \dots, c_k)$ and $p'=(c'_0=c', c'_1, \dots, c'_k)$ with $d_w^C(c_{j+1}) = d_w^C(c_j)-1$, $d_w^C(c'_{j+1}) = d_w^C(c'_j)-1$, and $c_k = c'_k\in C$. Let $i = d_w^C(c_k)$.

Choose a ball $b\in\B_\fwC{C}{w}\cap Z_i$. Suppose $b$ is northeast of $c_k$. Consider a semi-bounded reverse walk $(b_0 = b,\dots, b_l)$ from $b$ above $p'$; it necessarily terminates since $c'_0$ is the northeast-most ball of its zig-zag. By Remark \ref{rmk:semi-bounded walk}, this means that $b_l$ is $Nr$-terminal. If instead $b$ is southwest of $c_k$ then the semi-bounded reverse walk leads to a $Wr$-terminal ball. Of course the paths $p, p'$ could have been further extended along $C$, so the argument works for any ball $b\in Z_j$ with $j<i$. However the property of being terminating is invariant under translation by $(n,n)$, finishing the proof.
\end{proof}

\begin{proof}[Proof of Theorem \ref{thm:forward backward pretheorem}]
Let $\{Z_i\}_{i\in\Z}$ be the collection of zig-zags associated with $d_w^C$. By construction, $\fw{d_w^C}$ is a monotone numbering and for each $b\in\B_{\fwC{C}{w}}\cap Z_i$, $S^{(i)}$ lies northwest of $b$. By Remark \ref{rmk:lower bound for backward numbering}, we know that for any ball $b$, $\fw{d^C}(b)\leqslant d_\fwC{C}{w}^{\jbk,S}(b)$.

For a terminal ball $b$, we manifestly have $d_\fwC{C}{w}^{\jbk,S}\leqslant \fw{d^C}(b)$, hence the two numberings must coincide on terminal balls. Now choose any $b$; by Lemma \ref{lem:all balls are terminating} it is terminating. Choose a reverse path from $b$ to a terminal ball such that $\fw{d^C}$ increases by $1$ on each step. The backward numbering must also increase by at least $1$ with each step. Hence $d_\fwC{C}{w}^{\jbk,S}(b)\leqslant\fw{d^C}(b)$. This finishes the proof.
\end{proof}

\begin{proof}[Proof of Proposition \ref{prop:fw bk is id}]
\label{pf:fw bk is id}
The proposition states that that taking a forward step with respect to a channel numbering, followed by a backward step does not alter a partial permutation. This is an immediate corollary of Theorem \ref{thm:forward backward pretheorem}.
\end{proof}
Thus we have now proven that for any $w\in\widetilde{W}$, $\Psi(\Phi(w)) = w$. In the remainder of the section we will show that if $(P,Q,\rho)\in\Omega_{dom}$, then $\Phi(\Psi(P,Q,\rho)) = (P,Q,\rho)$. Before getting to the proof we need to develop the theory of  backward numberings slightly further.

The next major result shows that the numbering induced from a backward step is, in fact, a channel numbering. We start with a definition and preliminary lemmas.

\begin{defn}
\label{defn:back terminating}
Consider a partial permutation $w$ and a compatible stream $S$ with some fixed proper numbering. Then $b\in\B_w$ whose backward numbering is $i$ is \emph{$N$-terminal} (resp. \emph{$W$-terminal}) if it lies north (resp. west) of $S^{(i+1)}$.

A ball $b$ is \emph{$N$-terminating} if there exists a reverse path $(b=b_0,\dots, b_k)$ such that the backward numbering increases by $1$ at each step and $b_k$ is $N$-terminal. Similarly for \emph{$W$-terminating}.
\end{defn}

Note that balls of $w$ which are $N$- or $W$-terminal are precisely the ones whose backward numbering is equal to the stream numbering.

\begin{lemma}
Consider a partial permutation $w$ and a compatible stream $S$ with some fixed proper numbering. Then any $b\in B_w$ is either $N$-terminating or $W$-terminating (or both).
\end{lemma}
\begin{proof}
Let $(d_0, d_1,\ldots, d_k)$ be the sequence of numberings in the backward numbering algorithm; so $d_0$ is the stream numbering and $d_k = d_w^{\jbk,S}$. Pick $b\in\B_w$. If $d_0(b) = d_w^{\jbk,S}$ then $b$ is terminal and we are done. If not, choose the smallest $l_1$ such that $d_{l_1}(b) = d_w^{\jbk,S}(b)$ (so $d_{l_1-1}(b) = d_{l_1}(b) + 1$). By construction, there exists a ball $b_1$ southeast of $b$ with $d_{l_1-1}(b_1) = d_{l_1-1}(b)$. Note that the backward numbering is monotone, so 
\[d_w^{\jbk,S}(b)<d_w^{\jbk,S}(b_1)\leqslant d_{l_1-1}(b_1) = d_{l_1-1}(b) = d_{l_1}(b) + 1 = d_w^{\jbk,S}(b) + 1.\]
So, $d_w^{\jbk,S}(b_1) = d_{l_1-1}(b_1) = d_w^{\jbk,S}(b)+1$. If $b_1$ is terminal, we are done. If not, choose the smallest $l_2$ such that $d_{l_2}(b_1) = d_{l_1-1}(b_1)$. There exists a ball $b_2$ southeast of $b_1$ with $d_{l_2-1}(b_2) = d_{l_2-1}(b_1)$. By the same argument as before, $d_w^{\jbk,S}(b_2) = d_{l_2-1}(b_2)= d_w^{\jbk,S}(b_1)+1$. Repeat this process; we claim that for some $r$, $b_r$ is terminal. Indeed, if this does not happen earlier, then for some $r$, $l_{r+1} = 0$. So $d_0(b_r) = d_{l_r-1}(b_r) = d_w^{\jbk,S}(b_r)$, i.e. $b_r$ is terminal.
\end{proof} 

\begin{lemma}
\label{lem:north of nterminating is nterminating}
Consider a partial permutation $w$ and a compatible stream $S$ with some fixed proper numbering. Let $d$ be the induced (proper) numbering on $\bk{S}{w}$. Suppose two balls $b$ and $b'$ of $w$ are part of the same zig-zag corresponding to $d$, $b$ is southwest of $b'$, and $b$ is $N$-terminating. Then $b'$ is also $N$-terminating.
\end{lemma}
\begin{proof}
Suppose $p=(b_0=b,\dots, b_k)$ is a reverse walk visiting consecutive zig-zags, and $b_k$ is $N$-terminal. Consider the semi-bounded (by $p$) reverse walk starting at $b'$. If the walk terminates in less than $k$ steps, then its last ball is $N$-terminal as in Remark \ref{rmk:semi-bounded walk}. Otherwise we get a sequence of balls $(b_0',\dots, b_k')$ and $b_k'$ is north of $b_k$ on the same zig-zag. Hence $b_k'$ would also be $N$-terminal.
\end{proof}

Of course, in the notation of the lemma, any ball west of a $W$-terminating ball in the same zig-zag is also $W$-terminating.

\begin{prop}
\label{prop:backward numbering is channel}
Consider a partial permutation $w$ and a compatible stream $S$ with some fixed proper numbering. Let $d$ be the induced (proper) numbering on $\bk{S}{w}$. Then $d$ is a channel numbering.  
\end{prop}
\begin{proof}
First note that if the flow of $S$ is strictly greater than the width of $P_w$, then $\bk{S}{w}$ cannot have two disjoint channels (if it did, then $w$ would have a channel of the same density as demonstrated in the first part of Proposition \ref{prop: between channels a channel}); so any proper numbering is necessarily a channel numbering. Hence assume the flow of $S$ is equal to the width of the Shi poset of $w$.

We will find a ball $c\in\B_{\bk{S}{w}}$ such that for any $c'$ with $d(c')$ sufficiently large, there exists a path $(c_0=c',\dots, c_k=c)$ with the value of $d$ decreasing by $1$ at each step. Let us first see that this will finish the proof. We will know that there exists such a path from some $(n,n)$-translate $c'$ of $c$. By Lemma \ref{lem:high flow implies channel}, $c'$ is part of some channel $C$. Then by Remark \ref{rmk:monotone is sometimes channel}, $d = d_{\bk{S}{w}}^C$ as desired.

Let $\{Z_i\}_{i\in\Z}$ be the collection of zig-zags associated to $d$. Choose $c\in\B_{\bk{S}{w}}$ such that in its zig-zag, all the balls of $w$ weakly north of it are $N$-terminating and all the balls of $w$ weakly west of it are $W$-terminating. Suppose $c\in Z_i$. Let $b\in\B_w$ be the ball directly south of $c$ and $b'\in\B_w$ be the ball directly east of $c$. Choose a reverse path $(b_0=b,b_1,\dots,b_k)$ which visits consecutive zig-zags and has $b_k$ $W$-terminal. Similarly, choose a reverse path $(b_0'=b',b_1',\dots,b_l')$ which visits consecutive zig-zags and has $b_l'$ $N$-terminal. We know that $b$ is strictly southwest of $b'$; let us first assume that $b_i$ is strictly southwest of $b_i'$ whenever both exist.

Without loss of generality, $l\geqslant k$. We start by showing that one can find the desired path when $c'\in Z_{i+l+1}$. First suppose $l>k$. Pick  $c'\in\B_{\bk{S}{w}}\cap Z_{i+l+1}$. Since $b_l'$ is $N$-terminal, we know that $b_l'$ is north of $c'$. If $b_l'$ is west of $c'$, let $c_1$ be the ball of $\bk{S}{w}$ directly west of $b_l'$. Otherwise, let $c_1$ be any ball northwest of $c'$ with $d(c_1) = i+l$ (it exists since $d$ is continuous). Thus $c_1 \in Z_{i+1}$ is northwest of $c_0$ and strictly west of $b_l'$. Let $(c_1,\dots,c_{l-k})$ be the (beginning of) a semi-bounded walk from $c_1$ with respect to $(b_l',b_{l-1}',\dots, b_{k+1}')$. Let $(c_{l-k},c_{l-k+1}, \dots, c_{l+1})$ be the bounded walk from $c_{l-k}$ between $(b_k',\dots, b_0')$ and $(b_k,\dots, b_0)$. But between $b_0$ and $b_0'$ there is only one ball of $\bk{S}{w}$, namely $c$. So $c_{l+1} = c$. If $l=k$ then the same argument works except in the construction of the path $(c_0=c',\dots,c_{l+1})$ one does not need to do the semi-bounded walk. Of course from any zig-zag southwest of $(i+l+1)$-st one we can find a path back to the $(i+l+1)$-st one which satisfies the required properties, finishing this part of the proof.

It remains to handle the case when $b_i$ is not strictly southwest of $b_i'$ for some $i$. In this case, there exists a ball $b''$  which is both $N$- and $W$-terminating (for example, $b'' = b$). Choose reverse paths  $p = (b_0=b'',b_1,\dots,b_k)$ and $p' =(b_0'=b'',b_1',\dots,b_l')$ which visit consecutive zig-zags and end at $W$- and $N$-terminal balls, respectively. We may assume that the two paths do not intersect (since otherwise we could rename their last intersection $b''$). Moreover, we may assume that $b_i$ is strictly southwest of $b_i'$. Indeed, if this were not the case and $i$ would be the first positive integer that $b_i$ is northeast of $b_i'$, then $(b_0',\ldots, b_{i-1}', b_i,\ldots, b_k)$ and $(b_0,\ldots,b_{i-1}, b_i',\ldots, b_l')$ would be acceptable paths with fewer crossings. A similar argument to the one used in the previous paragraph will finish the proof with $c$ being the ball directly west of $b''$.
\end{proof}

%

In the next lemma we prove that under appropriate dominance assumptions the channel numbering in the previous lemma is actually the southwest channel numbering.

\begin{lemma}
Consider a partial permutation $w$ and a compatible stream $S$ (with some fixed proper numbering). Let $T = \st(w)$. Suppose that $S$ and $T$ have the same flow and $T$ is no lower than $S$ (as in Definition \ref{def:no lower}). Let $d$ be the numbering of $\B_{\bk{S}{w}}$ induced by the backward step. Then $d = d_{\bk{S}{w}}^{SW}$.
\end{lemma}
\begin{proof}
By Proposition \ref{prop:backward numbering is channel}, $d$ is a channel numbering; consider the balls of $w$, the balls of $\bk{S}{w}$, and the zig-zags defined by $d$. We need to show that for any ball $c$ there is a path from $c$ to the southwest channel of $\bk{S}{w}$ such that $d$ decreases by 1 at each step. Since $d$ is proper, it is sufficient to show that such a path exists for any $c$ in some fixed zig-zag.

Number $T$ via its backward numbering with respect to $S$. Fixing a proper numbering of $T$ is equivalent to fixing a particular shift of $d_w^{SW}$; more precisely we fix $d_w^{SW}$ so that the column of $T^{(i)}$ is the column of the southwest ball $b$ with $d_w^{SW}(b) = i$. Notice that $d_w^{SW}$ is a monotone numbering which is bounded above by the stream numbering. Hence, by Remark \ref{rmk:lower bound for backward numbering}, $d_w^{SW}$ is a lower bound for the backward numbering of $w$.

Since $T$ is no lower than $S$, there exists $i$ for which $T^{(i)}$ is north of $S^{(i+1)}$. Consider the ball $b\in\B_w$ which is directly east of $T^{(i)}$. By definition, $d_w^{SW}(b) = i$. Since $S^{(i+1)}$ is south of $b$, the backward numbering of $b$ is at most $i$. Using the previous paragraph, we conclude that the backward numbering of $b$ is $i$. 

Construct a path $(b_0=b, b_1, \dots,b_k)$ to the southwest channel of $w$ such that $d_w^{SW}$ decreases by $1$ at each step. The backward numbering must decrease at each step; since $d_w^{SW}$ is a lower bound for it, the two must coincide on the whole path. Thus the path visits consecutive zig-zags. Continue this path by consecutive elements $(b_{k+1},b_{k+2}, \dots)$ of the southwest channel. 
Now consider $c\in\B_{\bk{S}{w}}\cap Z_{i+1}$. There exists a ball $c_1\in \B_{\bk{S}{w}}\cap Z_i$ northwest of $c$ and west of $b$ (if $c$ is east of $b$ then $c_1$ may be chosen to be the ball directly west of $b$; otherwise $c_1$ may be chosen to be any ball $Z_i$ which is northwest of $c$). Let $(c_1, c_2, \dots)$ be a semi-bounded walk from $c_1$ with respect to $(b_0=b, b_1, \dots)$. This walk is necessarily infinite by Proposition \ref{prop:semibounded}. Moreover, infinitely many cells of the path lie west of the corresponding entries of the southwest channel of $w$. Hence this path must intersect the southwest channel of $\bk{S}{w}$.
\end{proof}

Now it is easy to show that $\Phi\circ\Psi$ is the identity map, provided we start with a dominant weight. 

\begin{proof}[Proof of Theorem \ref{thm:image2}]
\label{pf:image2}
After each step of the backward algorithm we end up with a new partial permutation $\bk{S}{w}$. This partial permutation has two numberings of interest: the numbering induced by the backward step, and its own southwest channel numbering. It turns out that if $\rho$ is dominant, then these two numberings coincide. If the flow of $S$ is strictly greater than the width of $P_w$, the two coincide since there is only one proper numbering up to shift (see the proof of Proposition \ref{prop:backward numbering is channel}). If the flow of $S$ is equal to the width of $P_w$, then the two coincide by the previous lemma. This means that a step of the backward algorithm followed by a step of AMBC preserves the partial permutation. A simple induction finishes the proof.
\end{proof}

We have shown on page \pageref{pf:image1} that $\Phi(w)\in\Omega_{dom}$. Combined with the results of this section, this proves that AMBC provides a bijection between $\widetilde{W}$ and $\Omega_{dom}$.

\section{Distances and altitudes}

The main point of this section is to interpret distances between channels of $w$ in terms of weights and in terms of distances between channels of $\fw{w}$. Namely, we will prove Theorem \ref{thm:dist and alt}. 

Let $w$ be a partial permutation such that $P_w$ has at least $2$ disjoint maximal antichains; let $k$ be the size of the larges collection of such antichains. Consider interlacing maximal collections of channels $\{C_1,\dots, C_k\}$ and $\{D_1,\dots, D_{k-1}\}$ with respect to $d_w^{SW}$ (as described in Corollary \ref{cor:interlacing channels}).

\begin{thm}
\label{thm:dist nonfirst channels}
If $k\geqslant 3$, then for any $1\leqslant i \leqslant k-2$, $\h(D_i,D_{i+1}) = \h(C_{i+1},C_{i+2})$.  
\end{thm}

\begin{proof}
Fix $d_w^{SW}$ and consider the corresponding collection of zig-zags. Fix  $1\leqslant i \leqslant k-2$. All further numberings of $\B_w$ will be chosen to coincide with $d_w^{SW}$ on $C_{i+2}$ and all further numberings of $\B_\fw{w}$ will be chosen to coincide with $\fw{d_w^{SW}}$ on $D_{i+1}$. This fixes the numbering of $C_{i+2}$ and $D_{i+1}$ for the rest of the proof.

By Proposition \ref{prop:induced and southwest coincide on channels} we know that $\fw{d_w^{SW}}$ and $d_\fw{w}^{SW}$ coincide on all channels of $\fw{w}$. Since $D_i$ is southwest of $D_{i+1}$, $d_\fw{w}^{SW}$ and $d_\fw{w}^{D_i}$ coincide on $D_{i+1}$. Hence $\fw{d_w^{SW}}$ and $d_\fw{w}^{D_i}$ coincide on $D_{i+1}$. Suppose for some $c\in C_{i+1}$, we have $d_w^{C_{i+2}}(c) = l$. We will show that in this case, for the element $b\in D_i$ in the zig-zag of $c$ we have $d_\fw{w}^{D_{i+1}}(b) = l$. This will finish the proof.

First we will prove that $d_\fw{w}^{D_{i+1}}(b) \geqslant l$. By assumption, there is a path $p = (c_0=c,\dots, c_h)$ from $c$ to $C_{i+2}$ of worth $l$ (so the fixed numbering of $c_h$ must be $l-h$). We will build a reverse path from $D_{i+1}$ to $D_i$ (see left side of Figure \ref{fig:distance_preserved_1}). Start with a ball $b_0\in D_{i+1}$ which is on the same zig-zag as $c_h$. Let $(b_0,b_1,\dots, b_{h-1})$ be the funnel reverse walk between $p$ and $C_{i+1}$. Let $b_h\in\B_\fw{w}$ be the ball of $\fw{w}$ directly south of $c$. Let $(b_h,b_{h+1},\dots, b_{h+h'})$ be a reverse bounded walk between $C_i$ and $C_{i+1}$ to $D_i$. Thus we have a path with $h+h'$ steps from $b_{h+h'} \in D_i$ to $b_0\in D_{i+1}$. The fixed numbering of $b_0$ is $l-h$ (it is the same as the numbering of $c_h$), so the worth of the path is $l+h'$. Thus $d_\fw{w}^{D_{i+1}}(b_{h+h'})\geqslant l+h'$. Now $b_{h+h'}$ is $h'$ zig-zags ahead of $b$. So $d_\fw{w}^{D_{i+1}}(b) \geqslant l$.

Now we will prove that $d_\fw{w}^{D_{i+1}}(b) \leqslant l$; the argument is nearly symmetric. Suppose there is a path $q = (b_0=b,\dots, b_g)$ from $b$ to $D_{i+1}$ of worth $s$ (so the fixed numbering of $b_g$ must be $s-g$). We will build a path from $c$ to $C_{i+2}$ (see right side of Figure \ref{fig:distance_preserved_1}). Let $(c_0=c,\dots, c_{g-1})$ be the funnel walk between $D_{i+1}$ and $q$. Let $c_g\in\B_w$ be the ball of $w$ directly north of $b_g$. Let $(c_g,\dots, c_{g+g'})$ be the walk from $c_g$ to $C_{i+2}$ (if $D_{i+2}$ exists, we can find a bounded walk between $D_{i+1}$ and $D_{i+2}$; otherwise we can use a semi-bounded walk). The fixed numbering of $c_{g+g'}$ is $s-g-g'$ and the worth of the path $(c_0,\dots, c_{g+g'})$ is $s$. Hence $s\leqslant d_w^{C_{i+2}}(c) = l$. So $d_\fw{w}^{D_{i+1}}(b)$, which is the maximal value of $s$ over the set of paths, is also no larger than $l$.
\end{proof}

\begin{figure}
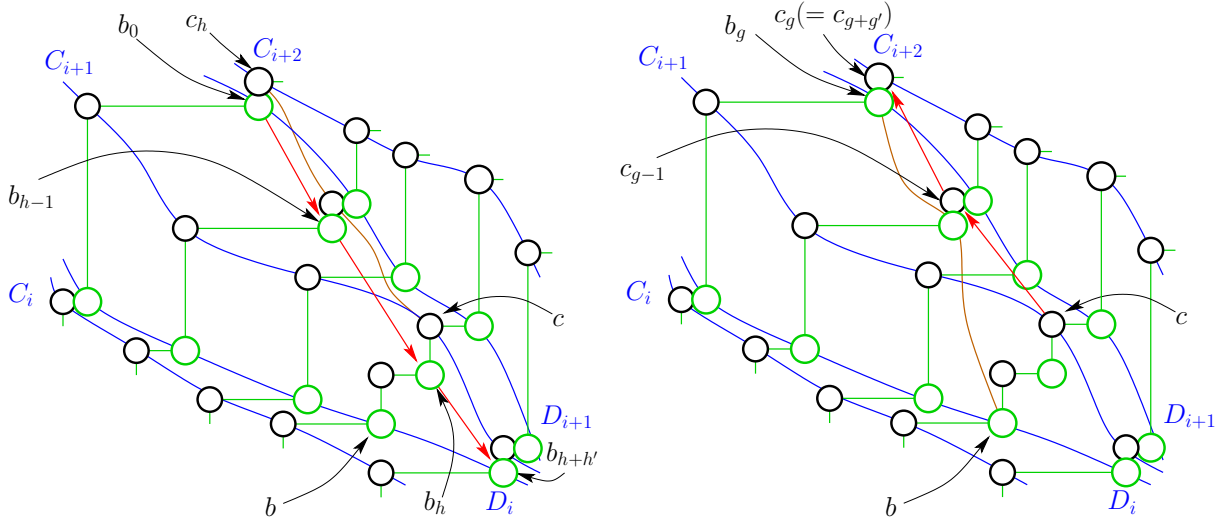

\centering
\resizebox{.45\textwidth}{!}{\input{figures/dist_preserved_1_ng.pspdftex}}\qquad\resizebox{.45\textwidth}{!}{\input{figures/dist_preserved_2_ng.pspdftex}}
\caption{The two paths constructed in the proof of Theorem \ref{thm:dist nonfirst channels}. In the right half the ball $c_g$ already lies on the channel $C_{i+2}$, so the last walk is not necessary in the example.}
\label{fig:distance_preserved_1}
\end{figure}

Before proving the statement for the southwest two channels $C_1, C_2$, we give an alternative description of the pseudometric $\h$ on the set of channels:
\begin{lemma}
Suppose $w$ is a partial permutation and $C, C'$ are channels. 
$$\h(C, C') = \min_{(b_0,\dots,b_f)} d^C(b_0) - (d^{C}(b_f) + f)$$
where the minimum is taken over all paths from $C$ to $C'$.
\end{lemma}
\begin{proof} Note that if we fix $b_0$ then the right hand side becomes $d^C(b_0) - d^{C'}(b_0)$, where the shift of $d^{C'}$ is chosen so that $d^C$ and $d^{C'}$ coincide on $C'$. This is independent of $b_0$ and is, by definition, $h(C,C')$.
\end{proof}

\begin{figure}
\centering\resizebox{.6\textwidth}{!}{\input{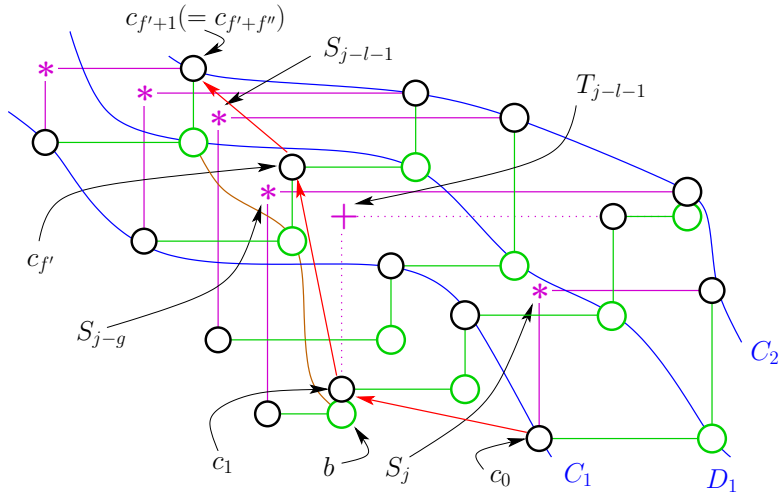}}
\caption{The construction of the first path in the proof of Theorem \ref{thm:dist first channels}. The cells of $S$ are labeled by $*$'s while the cell of $T$ is labeled by $+$ (other cells of $T$ are not relevant hence are omitted). }
\label{fig:distance_preserved_2}
\end{figure}

\begin{thm}
\label{thm:dist first channels}
Let $S$ be the stream resulting from a step of AMBC on $w$, and let $T$ be the stream resulting from a step of AMBC on $\fw{w}$. Choose $a_0$ to be the altitude of the stream from the class of $T$ (recall Definitions \ref{defn:class} and \ref{def:altitude}) which is concurrent to $S$. Then $a(T) = a_0 + \h(C_1, C_2)$ 
\end{thm}
\begin{proof}
We have already proved (Theorem \ref{thm:image1}) that $a(T) \geqslant a_0$. Suppose for some $l$, $a(T) = a_0 + l$. We will show that $\h(C_1,C_2)\leqslant l$ and that $\h(C_1,C_2)\geqslant l$. To prove the first statement we need to find a path $(b_0,\dots,b_f)$ from $C_1$ to $C_2$ such that $d^{SW}(b_0) - (d^{C}(b_f) + f)\leqslant l$.

Fix some proper numbering of $S$. If $t$ is the partial permutation whose balls are the cells of $T$, fix the numbering $T$ which gives the backward numbering of $t$ with respect to $S$. Recall from the proof of Theorem \ref{thm:image1} on page \pageref{pf:image1} that choosing the relative matching of the numberings of $T$ and $S$ in this way ensures that $\fw{d_w^{SW}}$ and $d_{\fw{w}}^{SW}$ coincide on channels.

One can see that $a(T) = a_0 + l$ means precisely that for some $j$, $S_j$ is east of $T_{j-l-1}$. Now $T_{j-l-1}$ lies directly north of a ball $b\in\B_{\fw{w}}$ (see Figure \ref{fig:distance_preserved_2}). So $d_\fw{w}^{SW}(b) = j-l-1$ and hence $\left(\fw{d_w^{SW}}\right)(b)\geqslant j-l-1$ (Proposition \ref{prop:induced and southwest coincide on channels}). Let $g = j - \left(\fw{d_w^{SW}}\right)(b)$; this way $b$ is in the $(j-g)$-th zig-zag. We know that $0< g\leqslant l+1$.

Consider a path $p=(b_0=b,\dots, b_{f'})$ of maximal worth from $b$ to $D_1$. We will now construct the required path from $C_1$ to $C_2$. Choose $c_0\in C_1$ such that $d_w^{SW}(c_0) = j$. Let $c_1\in\B_w$ be a ball northwest of $c_0$ and between $b$ and $D_1$ (it exists by Lemma \ref{lem: key lemma}). Let $(c_1,\dots, c_{f'})$ be the funnel walk between $p$ and $D_1$. Let $c_{{f'}+1}\in\B_w$ be the ball directly north of $b_{f'}$. Let $(c_{f'+1},\dots, c_{f'+f''})$ be the walk from $c_{f'+1}$ to $C_2$ (it is either bounded between $D_1$ and $D_2$ or semi-bounded above $D_1$, depending on whether $k>2$ or $k=2$).

The resulting path $(c_0,\dots, c_{f'+f''})$ has $f'+f''$ steps. By construction $c_0$ had label $j$ (in the SW numbering) and $c_1$ had label $j-g$. Now the label of $b_{f'}$ must be $j-g-f'$, and so the label of $c_{f'+1}$ is also $j-g-f'$. Finally the label of $c_{f'+f''}$ must be $j-g-f'-f''+1\geqslant j-l-f'-f''$. The worth of the path $(c_0,\dots, c_{f'+f''})$ with respect to $C_2$ is thus $\geqslant j-l$. Thus $\h(C_1,C_2)\leqslant l$ as desired.

Now we will show that $\h(C_1, C_2)\geqslant l$. Suppose $\h(C_1,C_2) = l'$. Choose a path $q=(c_0,c_2,\dots, c_f)$ from $C_1$ to $C_2$ such that $d_w^{SW}(c_0) - (d_w^{SW}(c_f) + f) = l'$. We want to show that for some $j$, $S^{(j)}$ is east of $T^{(j-l'-1)}$. Let $b_0$ be the ball of $D_1$ on the zig-zag of $c_f$; we will construct a reverse path $(b_0 = b,\dots, b_{f'})$. First consider a reverse funnel walk $(b_0,\dots, b_{f-1})$ between $q$ and $C_1$. If $c_0$ is the SW-most ball of its zig-zag, then let $f' = f-1$. Otherwise let $b_f$ be the ball of $\fw{w}$ directly south of $c_0$. Then continue the reverse path to a (necessarily finite) semi-bounded reverse walk $(b_f,b_{f+1},\dots, b_{f'})$ with respect to $C_1$. 

Suppose the zig-zag containing $b_{f'}$ is labeled $j-1$. Notice that $j - 1 = d_w^{SW}(c_0) + f'$ since the walk visits each zig-zag. Since the walk stopped, we must have $S_j$ lying east of $b_{f'}$. We know that $d_{\fw{w}}^{SW}(b_0) = d_w^{SW}(c_f)$ since $b_0$ is on a channel. Hence $d_{\fw{w}}^{SW}(b_{f-1})\geqslant d_w^{SW}(c_f) + f-1 = d_w^{SW}(c_0) - l' - 1$. Hence $d_{\fw{w}}^{SW}(b_{f-1})\geqslant d_w^{SW}(c_0) - l' + f' = j-l'-1$. Thus $T^{(j-l'-1)}$ is west of $b_{f'}$. This finishes the proof.
\end{proof}

\section{Proof of Weyl symmetry}

The purpose of this section is to prove Theorem \ref{thm:weyl} and Corollary \ref{cor:weyl}, namely to show that $\Phi\circ\Psi$ preserves the tabloids and maps the weight to the dominant chamber.

Suppose $P$ and $Q$ are two tabloids with the same shape $\lambda$, and $\rho$ is an integer vector of size $\ell(\lambda)$. Let $r_1 = 0$. For each $2\leqslant i\leqslant\ell(\lambda)$; if $\lambda_i < \lambda_{i-1}$ then let $r_i = 0$, otherwise let $r_i$ be the unique integer such that $\mathfrak{st}_{r_i}(P_i, Q_i)$ is concurrent to $\mathfrak{st}_{0}(P_{i-1}, Q_{i-1})$.

Choose $i$ such that $\lambda_i = \lambda_{i-1}$. Since any permutation is a product of transpositions, to prove Theorem \ref{thm:weyl} it suffices to prove that $\Psi(P,Q,\rho) = \Psi(P,Q,\rho')$, where 
\[\rho' = (\rho_1,\dots,\rho_{i-2}, \rho_i-r_i+r_{i-1},\rho_{i-1}-r_{i-1}+r_i, \rho_{i+1},\dots, \rho_{\ell(\lambda)}).\] 
We will in fact prove that after the two backward steps which remove rows $i$ and $i-1$, the partial permutations already coincide. 

The notions of shifting streams, channels, and rivers will be central to the rest of this section. Recall the for a stream $S$ we write $S\sh{i}$ for the shift of $S$ by $i$, and analogously for partial permutations consisting of a single channel.
\begin{defn}
For a partial permutation $w$, a channel $C$, and a river $R$ we write 
\begin{itemize}
\item $w\sh{i}_C$ for the permutation obtained from $w$ by replacing $C$ with $C\sh{i}$, and
\item $w\sh{i}_R$ for the permutation obtained from $w$ by shifting the northeast channel of $R$ by $i$ if $i > 0$ or shifting the southwest channel of $R$ by $i$ if $i < 0$.
\end{itemize}
Th first of these is referred to as \emph{shifting the channel by $1$} and the second as \emph{shifting the river by $1$}.
\end{defn}

In Section \ref{sec:moving streams proofs} we study what happens to the backward step of a permutation if we change the altitude of the stream by 1. In Section \ref{sec:moving channels proofs} we study what happens when change the altitude of one stream and do two steps back (preserving the altitude of the second stream). In Section \ref{sec:weyl symmetry proofs} we combine the two theories to say that first decreasing the altitude of one stream appropriately and then increasing the altitude of the other preserves the partial permutation resulting from the two backward steps.

\subsection{Shifting streams}
\label{sec:moving streams proofs}
Consider a partial permutation $w$ and a stream $S$ compatible with it; let $S'= S\sh{1}$. Assume that the streams have the same flow as the width of the Shi poset of $w$. In this section we describe the difference between the partial permutations $\bk{S}{w}$ and $\bk{S'}{w}$; it turns out that only one channel is different between them. 

We will have to consider proper numberings of the two streams $S$ and $S'$. The overall shift is unimportant, so we fix some proper numbering of $S$. There are two natural numberings of $S'$: one matches that of $S$ on the rows (call it the \emph{row-matching numbering}) while the other one matches on the columns (call it the \emph{column-matching numbering}). 
\begin{rmk}
\label{rmk:column one greater than row}
Since $S$ is numbered by consecutive integers, for any cell of $S'$, the column-matching numbering is always $1$ greater than the row-matching numbering.
\end{rmk}

\begin{lemma}
Let $d$ be the backward numbering of $w$ with respect to $S$, and $d'$ the backward numbering with respect to the row-matching numbering of $S'$. Then for any $b\in\B_w$ we have $d(b) = d'(b)$ or $d(b) = d'(b) + 1$. 
\end{lemma}
\begin{proof}
Let $d''$ be the backward numbering with respect to the column-matching numbering of $S'$. Note that $d'$ is a monotone numbering which is no larger than the stream numbering for $S$, so by Remark \ref{rmk:lower bound for backward numbering}, $d'(b)\leqslant d(b)$. For the same reason, it is clear that $d(b)\leqslant d''(b)$. By Remark \ref{rmk:column one greater than row}, $d''(b)-d'(b) = 1$. This finishes the proof.
\end{proof}

Next we describe precisely on which balls the numberings disagree.

\begin{lemma}
\label{lem:moving stream numbering}
Let $d$ be the backward numbering of $w$ with respect to $S$, and $d'$ the backward numbering with respect to the row-matching numbering of $S'$. Then $d(b) = d'(b) + 1$ if and only if $b$ is $W$-terminating with respect to $S$. Moreover, any ball $W$-terminating with respect to $S$ is $W$-terminating with respect to $S'$.
\end{lemma}
\begin{proof}
Suppose $b$ is $W$-terminating with respect to $S$. Consider a reverse path $p=(b_0 = b, b_1, \dots, b_k)$ such that $d$ increases by $1$ at each step and $b_k$ is $W$-terminal. Let $i = d(b_k)$. Now $S^{(i+1)}$ is east of $b_k$; hence $S'^{(i)}$ is east of $b_k$, and so $d'(b_k) < d(b_k)$. So $d'(b) < d(b)$. Hence $d'(b) = d(b) - 1$, as desired. Moreover, $d'(b_k) = d(b_k) - 1 = i-1$ and hence $b_k$ is $W$-terminal with respect to $S'$. Since $d'$ increases with each step of $p$, it must increase by $1$; thus $b$ is $W$-terminating with respect to $S'$.

Conversely, suppose $d'(b) < d(b)$. Consider a reverse path $(b_0=b,\dots, b_k)$ such that $d'$ increases by one at each step and $b_k$ is terminal with respect to $S'$. Of course, $d$ must also increase at each step, hence $d'(b_k) < d(b_k)$. By the previous lemma, $d'(b_k) = d(b_k) - 1$ and $d$ increases by $1$ at each step. To finish the proof we will show that $b_k$ is $W$-terminal with respect to $S$.

Let $i = d(b_k)$; we know that $S^{(i)}$ is northwest of $b_k$. However since $b_k$ was terminal with respect to $S'$ and $d'(b_k) = i-1$ , $S'^{(i)}$ is not northwest of $b_k$. Since $S'^{(i)}$ is in the same row as $S^{(i)}$, is is north of $b_k$. Hence $S'^{(i)}$ is east of $b_k$. Hence $S^{(i+1)}$ is east of $b_k$. Thus $b_k$ is $W$-terminal with respect to $S$ as desired. 
\end{proof}

Recall that the balls of $w$ in each zig-zag southwest of a certain point are $W$-terminating and northeast of that point are not $W$-terminating.
\begin{lemma}
\label{lem:it_s_a_channel}
Consider the permutations $w$ and $\bk{S}{w}$, and the corresponding zig-zags $\{Z_i\}_{i\in\Z}$. For each $i$, let $c_i$ be the ball of $\B_{\bk{S}{w}}\cap Z_i$ directly north of the northeastern $W$-terminating ball of $\B_w\cap Z_i$ (or, if $Z_i$ has no $W$-terminating balls of $w$, let $c_i$ be the southwest ball of $\B_{\bk{S}{w}}\cap Z_i$). Then $\{c_i\}_{i\in\Z}$ is a channel of $\bk{S}{w}$.
\end{lemma}

\begin{proof}
By Corollary \ref{cor:with of backward Shi poset}, it is sufficient show that $c_{i+1}$ is always southeast of $c_i$. Suppose first that $c_{i+1}$ is north of $c_i$. Since the induced numbering $d$ of $\bk{S}{w}$ is proper, there must be a ball $\B_{\bk{S}{w}}\cap Z_i$ which is strictly north of $c_{i+1}$. Hence, there exists a ball $b$ of $w$ directly east of $c_i$. Now by Lemma \ref{lem:zigzags dont intersect}, $c_{i+1}$ must be east of $b$. By construction, $b$ was not $W$-terminating and in particular not $W$-terminal. Hence the southwest ball of $\bk{S}{w}$ in the $(i+1)$-st zig-zag must be west of $b$. It cannot lie northwest of $b$, so it must be south of $b$. By the reflection of Lemma \ref{lem: key lemma} in an anti-diagonal, we know that there is a ball $b'$ of $w$ in the $(i+1)$-st zig-zag, southwest of $c_{i+1}$, and southeast of $b$. Now $b'$ must be $W$-terminating since it lies below $c_{i+1}$ in its zig-zag. Hence $b$ is also $W$-terminating. This is a contradiction. Thus $c_{i+1}$ is south of $c_i$.

Now suppose $c_{i+1}$ is strictly west of $c_i$. Let $b\in\B_w$ be the ball directly south of $c_i$ (it must exist since $d$ is proper). By construction of $c_i$, $b$ is $W$-terminating, but it is not $W$-terminal. Hence there is a ball in $\B_w\cap Z_{i+1}$ which is east of $c_i$ and is $W$-terminating. So there is a ball of $w$ directly east of $c_{i+1}$, and by the reflection of Lemma \ref{lem:north of nterminating is nterminating}, it is $W$-terminating. This contradicts the definition of $c_{i+1}$ and finishes the proof.   
\end{proof}

\begin{defn}
Suppose $w$ is a partial permutation and $S$ is a compatible stream. Then a step of the backward algorithm induces a channel numbering $d$ on $\bk{S}{w}$. Define the \emph{indexing river} corresponding to $S$ and $w$ to be 
\[\bigcup_{\substack{C\in\C_{\bk{S}{w}}\\ d^C_{\bk{S}{w}} = d}} C.\]
\end{defn}

\begin{rmk}
\label{rmk:indexing between green and blue}
Suppose $b\in\B_{\bk{S}{w}}$ and the ball of $w$ directly south of $b$ is $W$-terminating while the ball of $w$ directly east of it is $N$-terminating. In the proof of Proposition \ref{prop:backward numbering is channel} we have shown that starting at any ball of $\bk{S}{w}$ with sufficiently large value of $d$, there is a path to $b$ such that $d$ decreases by $1$ with each step. In particular, there exists such a path from a translate of $b$. By Lemma \ref{lem:high flow implies channel}, $b$ is part of a channel. Moreover, this shows that $b$ is in the indexing river. 
\end{rmk}

\begin{lemma}
\label{lem:channel_is_northeast_indexing}
Consider the permutations $w$ and $\bk{S}{w}$. Then the channel $C=\{c_i\}_{i\in\Z}$ described in the previous lemma is the northeast channel of the indexing river.
\end{lemma}
\begin{proof} As observed in Remark \ref{rmk:indexing between green and blue}, $C$ is part of the indexing river. Consider a channel $C'$ weakly northeast of $C$. We would like to show that $d(C,C')>0$. Choose a ball $c\in C'\cap Z_i$ strictly northeast of $c_i$. 

Suppose that this was not the case. Choose a path $p = (a_0,a_1,\dots, a_k)$ which visits consecutive zig-zag and starts at $a_0\in C\cap Z_i$. Consider the ball $b$ of $w$ which is directly south of $c$. Since $c$ is strictly northeast of $c_i$, $b$ is not $W$-terminating. Consider a reverse funnel walk $(b_0=b,b_1,\dots, b_{k-1})$ between $C$ and $p$. We will show that $b_{k-1}$ is $W$-terminating, thus arriving at the desired contradiction.

Either $a_k$ is the southwest cell of its zig-zag, or there exists a ball $b_k$ of $w$ directly south of it. In the first case, $b_{k-1}$ is, in fact, $W$-terminal. In the second, let $b_k$ be the ball directly south of $a_k$. Then by the definition of $C$, $b_k$ is $W$-terminating and is southeast of $b_{k-1}$; thus $b_{k-1}$ is also $W$-terminating. This finishes the proof.  
\end{proof}

\begin{thm}
\label{thm:moving streams}
Consider a partial permutation $w$ and a compatible stream $S$; let $S'=S\sh{1}$. Assume that the streams have the same flow as the width of the Shi poset of $w$. Let $R$ (resp. $R'$) be the indexing river of $\bk{S}{w}$ (resp. $\bk{S'}{w}$). Then $\bk{S'}{w} = \bk{S}{w}\sh{1}_R$, and $\bk{S}{w} = \bk{S'}{w}\sh{-1}_{R'}$.
\end{thm}
\begin{proof}
It is easy to see that the change in backward numbering described in Lemma \ref{lem:moving stream numbering} results, after the backward step, in shifting by 1 the channel $C = \{c_i\}$ from Lemma \ref{lem:it_s_a_channel}. This channel is the northeast channel of the indexing river by Lemma \ref{lem:channel_is_northeast_indexing}, finishing the proof. The last statement follows by transposing with respect to the main diagonal.
\end{proof}

\begin{rmk} 
\label{rmk:moving streams}
As noted in Lemma \ref{lem:moving stream numbering}, a ball which is $W$-terminating with respect to $S$ remains $W$-terminating with respect to $S'$. It is also clear that any ball which is not $W$-terminating with respect to $S$ remains $N$-terminating with respect to $S'$ (the same path to an $N$-terminal ball works in both cases). By Remark \ref{rmk:indexing between green and blue}, $C\sh{1}$  is part of the indexing river of $\bk{S'}{w}$ (in fact, it is the southwest channel of this river). Thus if we were to shift $S'$ by $1$, the result of the backward step would be to shift by $1$ the river containing $C\sh{1}$.
\end{rmk}

\subsection{Shifting channels}
\label{sec:moving channels proofs}

Consider a permutation $w$ and a compatible stream $S$ such that the flow of $S$ is equal to the width of the Shi poset of $w$. Choose a channel $C$ of $w$ which is the northeast channel of its corresponding river. Let $C' = C\sh{1}$ and $w' = w\sh{1}_C$. In this section we are interested in studying the difference between the partial permutations $\bk{S}{w}$ and $\bk{S}{w'}$. It turns out that, just as in the case of shifting a stream, these permutations are related by shifting a channel.

\subsubsection{Preliminary results on shifting rivers}

We start with a few results about the relationship between $w$ and $w'$ themselves.

\begin{defn}
\label{defn:new sw numbering}
Construct a numbering of $w'$ as follows. For each $b'\in\B_{w'}$, define
\[d(b')=\begin{cases}
d_w^C (b')     & \text{if $b'$ is strictly northeast of $C'$}\\
d_w^C (b)      & \text{if $b'\in C'$}\\
d_w^C (b') - 1 & \text{if $b'$ is strictly southwest of $C'$}
\end{cases}.\] 
\end{defn}

\begin{lemma}
\label{lem:d is the sw numbering}
With the above notation, $d = d_{w'}^{c'}$.
\end{lemma}
\begin{proof}
There is a path from any ball of $w'$ to $C'$ with $d$ decreasing by $1$ at each step (obtained from a path to $C$ in $w$ by substituting the last ball with the relevant ball of $C'$). Thus it remains to show that $d$ is proper.

First we show that there are no balls of $w$ or $w'$ is certain positions. Let $(\ldots, c_{-1},c_0,c_1,\ldots)$ be the list of balls of $C$ from northwest to southeast. Consider the collection $\mathcal M$ of cells $c$ such that the exists $i$ with $c_i$ strictly north and strictly west of $b$ and  $c_{i+1}$ is strictly south and strictly east of $b$ ($\mathcal M$ is lightly shaded in the example if Figure \ref{fig:shifting channel is ok}). Note that $\mathcal M$ cannot contain any balls of $w$ or $w'$. Let $\mathcal N$ be an analogous collection for $C\sh{1}$; it is shaded dark in the example in the figure. There are also no balls of $w$ or $w'$ in $\mathcal N$ since the existence of such a ball would contradict the fact that $C$ is the northeast channel of its river (just replace one ball of $C$ with a ball of $\mathcal N$ northeast of it). Thus a ball of $w$ is strictly southwest (resp. northeast) of $C$ if and only if it is present is $w'$ and is strictly southwest (northeast) of $C'$.

It is not difficult to check continuity of $d$. For monotonicity, there are $9$ cases, depending on where each of the two balls lies relative to $C'$. Most of them are easy; we will only do the two most difficult ones. First, suppose $b_1'$ is northeast of $C'$, $b_2'\in C'$, and $b_1'$ is northwest of $b_2'$. We can check that $d(b_1') < d(b_2')$. Suppose not. Consider a maximal worth path $(c_0 = b_1',c_1,\dots c_k)$ from $b_1'$ to $C$ in $w$. If $c$ is the ball directly south of $b_2'$ then $(c, c_0, \dots, c_k)$ is a path in $w$ from $C$ to itself which visits every zig-zag (with respect to $d_w^C$). Moreover this path intersects $C$ and contains balls strictly northeast of $C$. This contradicts the assumption that $C$ was the northeast channel of its river. The second case is similar except $b_2'$ is southwest of $C'$. In this case consider the ball $b_2''$ of $C'$ directly east of the ball of $C$ on the zig-zag of $b_2'$. Replacing $b_2'$ with $b_2''$ reduces the proof to the previous case.
\end{proof}

\begin{figure}
\centering\resizebox{.3\textwidth}{!}{\input{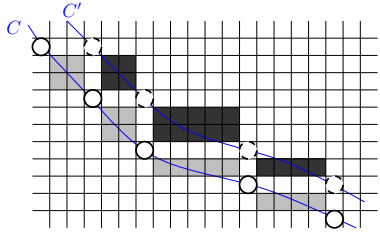}}
\caption{Two areas that cannot have any balls of $w$ or $w'$.}
\label{fig:shifting channel is ok}
\end{figure}

Since $d$ is proper and has the same period as $d_w^C$, Proposition \ref{prop: period of proper numbering} implies:
\begin{cor}
\label{cor:shi posets have same width}
The posets $P_w$ and $P_{w'}$ have the same width.
\end{cor}

\begin{lemma}
$C'$ is the southwest channel of its river.
\end{lemma}
\begin{proof}
Consider the zig-zags corresponding to $d_{w'}^{C'}$. Suppose there exists a ball $c$ of $w'$ southwest of $C'$ and part of the same river. Then there should be a path $(c_0,c_1,\dots, c_k = c)$ from $c_0\in C'$ to $c$ which visits consecutive zig-zags. We may assume that $c_1$ is southwest of $C'$ (if not, find the smallest $i$ such that $c_j$ is southwest of $C'$ for every $j\geqslant i$, let $c_0'$ be the ball of $C'$ in the zig-zag of $c_{i-1}$, and consider the path $(c_0',c_{i},\dots, c_k = c)$). Let $b$ be the element of $C$ directly west of $c_0$. In this case, $c_1$ is northwest of $b$. However, 
$$d_{w}^{C}(b) = d_{w'}^{C'}(c_0) = 1 + d_{w'}^{C'}(c_1) = d_{w}^{C}(c_1).$$ 
This contradicts the fact that $d_{w}^{C}$ is proper.
\end{proof}

The next series of results develops the theory of $W$- and $N$-terminating balls.

\begin{lemma}
Consider a partial permutation $w$ and a compatible stream $S$. For any river $R$ of $w$ and any ball $b\in R$, if $b$ is $W$-terminating (resp. $N$-terminating) then every ball of $R$ is $W$-terminating (resp. $N$-terminating).
\end{lemma}
\begin{proof}
Let $d = d_w^{\jbk, S}$. Consider $C\in\C_w$ with $b\in C$ and another channel $C'\subseteq R$; let $d_w^R := d_w^C = d_w^{C'}$. There exists a path $p$ from $b$ to some $(n,n)$-translate of $b$ which visits $C'$ and at each step $d_w^R$ decreases by one. Since $d$ is monotone and has the same period as $d_w^R$, we know that $d$ also decreases by one at each step of $p$. So there is a reverse path from $C'$ to $b$ such that $d$ increases by one with each step. Since $b$ is $W$-terminating, there is a reverse path from $b$ to a $W$-terminal ball with $d$ increasing by one with each step. So some (equiv. every) ball of $C'$ is also $W$-terminating. The same argument works in case $b$ was $N$-terminating.
\end{proof}

\begin{lemma}
\label{lem:one purple river}
Consider a partial permutation $w$ and a compatible stream $S$. Suppose $C$ is a channel whose balls are both $W$-terminating and $N$-terminating. Fix the shift of $d_w^C$ to coincide with $d_w^{\jbk,S}$ on $C$. Then $d_w^C$ and $d_w^{\jbk,S}$ coincide on any channel of $w$. 
\end{lemma}
\begin{proof} The zig-zags considered will be associated to the backward step. Consider another channel $C'$. Pick $b\in C$ and $b'\in C'$ on the same zig-zag. We assume that $b'$ is northeast of $b$ (if not, just transpose everything, including compass directions). Consider a reverse path $(b_0=b, b_1,\dots, b_k)$ from $b$ to an $N$-terminal ball $b_k$ which visits consecutive zig-zags. Consider a reverse path $(b_0'=b',\dots, b_{k+1}')$ consisting of consecutive elements of the channel $C'$. 

There are three possible cases:
\begin{enumerate}
\item $b_i'$ is northeast of $b_i$ for $i = 1,\ldots,k$ (as in the left-hand side of Figure \ref{fig:one_purple_river}), or
\item the two paths cross without intersecting (as in the right-hand side of Figure \ref{fig:one_purple_river}), or
\item the two paths intersect.
\end{enumerate}
In the first case, $b_{k+1}'$ is east of $b_k'$ which is east of $b_k$. Also, since $b_k$ is $N$-terminal, $b_{k+1}'$ is part of a zig-zag which is entirely south of $b_k$. So $b_{k+1}'$ is southeast of $b_k$. Then the path $(b_{k+1}',b_k,\dots, b_0)$ is a path from $C'$ to $C$ such that the backward numbering decreases by 1 at each step. Thus, by Remark \ref{rmk:monotone is sometimes channel}, the backward numbering agrees with $d_w^C$ on $C'$. In the last two cases, let $h$ be the smallest number such that $b_h'$ is weakly southwest of $b_h$. By essentially the same argument as before, $(b_h',b_{h-1},\dots, b_0)$ is a path from $C'$ to $C$ such that the backward numbering decreases by 1 at each step. Thus the backward numbering again agrees with $d_w^C$ on $C'$.

\begin{figure}
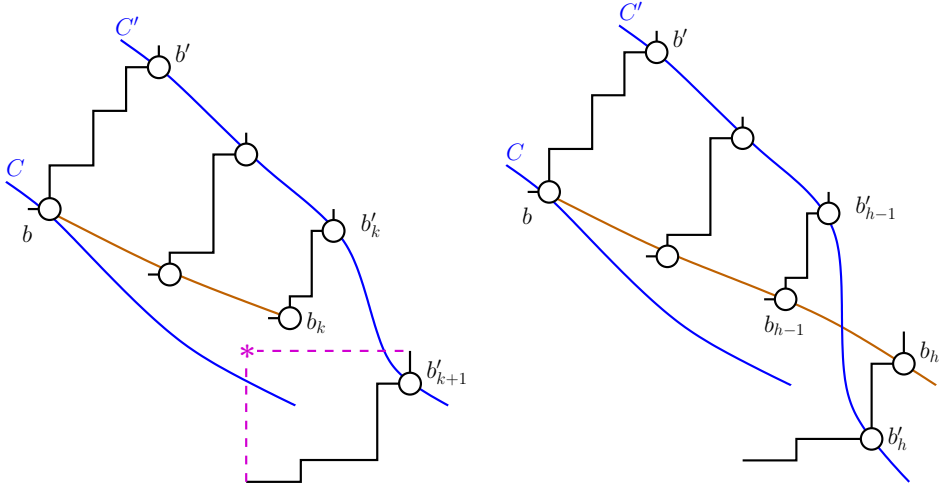

\centering\resizebox{.36\textwidth}{!}{\input{figures/one_purple_1.pspdftex}}\qquad \resizebox{.35\textwidth}{!}{\input{figures/one_purple_2.pspdftex}}
\caption{The first two cases considered in Lemma \ref{lem:one purple river}.}
\label{fig:one_purple_river}
\end{figure}

\end{proof}

\begin{cor}
\label{cor:W and N}
Consider a partial permutation $w$ and a compatible stream $S$. Then there exists at most one river of $w$ whose balls are both $W$-terminating and $N$-terminating with respect to $S$.
\end{cor}
\begin{proof}
This follows easily from Lemma \ref{lem:one purple river}.
\end{proof}

We will periodically need to keep careful track of how a (reverse) path can ``cross'' a channel. To this end we describe several types of interaction.

\begin{defn}
Suppose $p=(c_0,c_1,\dots, c_k)$ is a path or a reverse path, and $C$ is a channel. Then we say
\begin{itemize}
\item $p$ \emph{bridges} $C$ on step $i$ if $c_i$ is on one side (strictly northeast or strictly southwest) of $C$ while $c_{i+1}$ is on the other, 
\item $p$ \emph{fords} $C$ if for some $1\leqslant i<j<k$ we have $c_{i+1},\dots,c_j\in C$, while $c_i$ and  $c_{j+1}$ are on different sides of $C$,
\item $p$ \emph{skims} $C$ if for some $1\leqslant i<j<k$ we have $c_{i+1},\dots,c_j\in C$, while $c_i$ and  $c_{j+1}$ are on the same side of $C$, and
\item $p$ \emph{intersects} $C$ if $p$ fords or skims $C$.
\end{itemize}
\end{defn}

The next result is a non-algorithmic alternative definition for the backward numbering.	
\begin{lemma}
\label{lem:minimal r worth}
Consider a partial permutation $w$ and a compatible stream $S$. Let $d = d_w^{\jbk,S}$ and let $d_0$ be the stream numbering. Then for any ball $b$ of $w$,
$$d(b) = \min_{(b_0=b,\dots, b_k)} d_0(b_k) - k,$$
where the minimum is taken over all reverse paths starting at $b$.
\end{lemma}
\begin{proof}
Let $d'$ be the numbering defined by the right hand side of the above equation. It is clear that $d'$ is monotone and $d'(b) \leqslant d_0(b)$. So, by Remark \ref{rmk:lower bound for backward numbering}, $d'(b)\leqslant d(b)$. Now let $(b_0=b,\dots, b_k)$ be a reverse path which achieves the desired minimum, i.e. such that $d'(b) = d_0(b_k) - k$. So
\[d'(b) = d_0(b_k) - k \geqslant d(b_k) - k\geqslant d(b).\] 
\end{proof}

\begin{defn}
Consider a partial permutation $w$ and a compatible stream $S$. For a reverse path $(b_0,\dots, b_k)$, we refer to the number $d_0(b_k) - k$ as the \emph{$r$-worth} of the path.
\end{defn}

\begin{rmk}
\label{rmk:last ball is terminal}
The last ball on a reverse path of minimal $r$-worth is necessarily terminal since the stream numbering coincides with the backward numbering. 
\end{rmk}

We now analyze the backward numbering with respect to $S$ of $w$ and of $w'$.

\begin{thm}
\label{thm:how backward numbering changes}
Let $w$ be a partial permutation and $S$ be a compatible stream whose flow is equal to the width of $P_w$. Choose a non-$N$-terminating river of $w$ and let $C$ be its northeast channel. Let $w'=w\sh{1}_C$. Let $C'=C\sh{1}$. Let $\widetilde{C}$ be the northeast channel of $\bk{S}{w}$ which is southwest of $C$. Let $d$ (resp. $d'$) be the backward numbering of $w$ (resp. $w'$) with respect to $S$. Then
\begin{itemize}
\item for any $b\in C$, $d(b) = d'(b')$, where $b'$ is the ball of $C'$ directly north of $b$,
\item for any $b\in\B_w$ which is strictly northeast of $C$, $d'(b) = d(b)$,
\item for any $b\in\B_w$ which is southwest of $\widetilde{C}$, $d'(b) = d(b)$, and
\item for any $b\in\B_w$ which is northeast of $\widetilde{C}$ and strictly southwest of $C$, $d'(b) = d(b) + 1$. 
\end{itemize}
\end{thm}

The proof of the above theorem (p. \pageref{pf:how backward numbering changes}), toward which we will now start working, mostly consists of (tedious) careful analysis of the $r$-worth of reverse paths starting from the four regions.  Until the end of the proof, let $d, d', C, \widetilde{C}$ be as in the theorem, except we do not assume $C$ is not $N$-terminating.

\subsubsection{Normal form of paths}

First we narrow the class of reverse paths that need to be considered. In the next result we will show that the path does not need to cross $C$ multiple times. After that, we will show that considering paths that bridge $C$ is often unnecessary.

\begin{defn}
A reverse path $(b_0,\dots, b_l)$ for $w$ is in \emph{normal form} if for some $0\leqslant g\leqslant h\leqslant l+1$ we have
\begin{itemize}
\item $b_0,\dots, b_{g-1}$ all lie strictly to one side (southwest or northeast) of $C$,
\item $b_{g},\dots, b_{h-1}$ are consecutive elements of $C$, and
\item $b_{h},\dots, b_l$ all lie strictly to one side (southwest or northeast) of $C$.
\end{itemize}
\end{defn}

\begin{lemma}
\label{lem:to nf}
For every reverse path $p=(b_0,\dots, b_k)$ for $w$, there exists a reverse path $p'=(b'_0=b_0,\dots, b'_l=b_k)$ in normal form of at most the same $r$-worth.
\end{lemma}
\begin{proof}
If the reverse path never intersects or bridges the channel, then we can take $p' = p$, $g = h = 0$. Thus we only need to show that if $p$ intersects or bridges the channel twice, then the path obtained by following the channel between these points will have at most the same $r$-worth. There are four cases to consider depending on whether there are bridges or intersections each of the two times. The cases are similar, so we only consider one in detail. Suppose that the reverse path first intersects the channel and then bridges the channel (see Figure \ref{fig:not cross too much} for an example). 

Without loss of generality, we can consider $b_0$ to be the point of intersection and we may assume that the path stays northeast of $C$ between $b_0$ and the bridge. Let $j$ be such that $b_j$ is the first ball after the bridge. Consider the collection of zig-zags induced by the backward step. For each $1\leqslant i<j$, let $b_i'$ be the ball in $C$ in the same zig-zag as $b_i$. Now, $b_j$ is southwest of $C$, hence south of the element of $C$ in its zig-zag, and thus south of $b_{j-1}'$. Also, $b_j$ is east of $b_{j-1}$, which is east of $b_{j-1}'$. So $b_j$ is southeast of $b_{j-1}'$. Thus the path $(b_0,b_1',\ldots, b_{j-1}',b_j,\ldots,b_k)$ has the same $r$-worth. If the $b_i'$s are not in consecutive zig-zags (as required by the definition of normal form), then including the elements of $C$ between them will only decrease the $r$-worth of the alternate path as compared to the $r$-worth of $p$. 

\begin{figure}
\centering\resizebox{.33\textwidth}{!}{\input{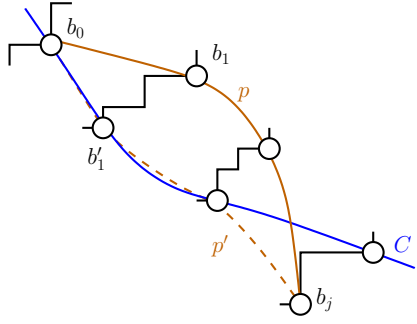}} 
\caption{Getting a path into normal form, as in the proof of Lemma \ref{lem:to nf}.}
\label{fig:not cross too much}
\end{figure}
\end{proof}

\begin{lemma}
\label{lem:jump off the bridge}
Suppose we have a reverse path $p=(b_0,\dots, b_k)$ for $w$ in normal form; moreover suppose that $p$ bridges $C$. Then there exists a reverse path $p'=(b'_0=b_0,\dots, b'_l)$ in normal form of the same $r$-worth that fords $C$.
\end{lemma}
\begin{proof}
Choose $i$ such that $b_i$ is on one side of $C$ while $b_{i+1}$ is on the other. Without loss of generality, $b_i$ is northeast of $C$ and $b_{i+1}$ is southwest of $C$. Let $b_i'$ be the southeast-most ball of $C$ which is west of $b_{i+1}$. Recall that $m$ is our notation for the width of $P_w$; let $b_{i+1}', b_{i+3}',\dots, b_{i+m}' = b_i'+(n,n)$ be consecutive elements of $C$ after $b_i'$. 

Notice that by definition of $b_i'$, $b_{i+1}'$ is east of $b_{i+1}$ and hence east of $b_i$. Moreover, $b_{i+1}'$ is south of $b_i'$which is south of $b_i$ (else we would contradict the assumption that $b_i$ is northeast of $C$). So $b_{i+1}'$ is southeast of $b_i$. Also, note that $b_i'$ is north (and, by definition, west) of $b_{i+1}$ (since $b_{i+1}$ is southwest of $C$). So $b_{i+m}' = b_i' + (n,n)$ is northwest of $b_{i+1}+(n,n)$. Now the reverse path $(b_0,\dots,b_i,b_{i+1}',\dots, b_{i+m}',b_{i+1}+(n,n),\dots, b_k+(n,n))$ fords $C$ and has the same $r$-worth as $p$. 
\end{proof}

\begin{lemma}
\label{lem:crossing paths not interesting 1}
Suppose $b\in\B_w$ is on one side of $C$ and $\tilde{b}\in\B_w$ is on the other. If there is a reverse path from $b$ to $\tilde{b}$ of $r$-worth $l$ then there exists a reverse path for 
$w'$ from $b$ to $\tilde{b}$ of $r$-worth $\leqslant l$.

Similarly suppose $b'\in\B_{w'}$ is on one side of $C'$ and $\tilde{b'}\in\B_{w'}$ is on the other. If there is a reverse path from $b'$ to $\tilde{b'}$ of $r$-worth $l$ then there exists a reverse path for $w$ from $b'$ to $\tilde{b'}$ of $r$-worth $\leqslant l$.

\end{lemma}
\begin{proof}
Suppose first that $b$ is northeast of $C$ and $\tilde{b}$ is southwest of $C$. Consider a path $p=(b_0=b,\dots, b_k=\tilde{b})$ in normal form with $r$-worth $\leqslant l$; it necessarily fords or bridges $C$. In case of bridging, the same path can be used in $w'$. Otherwise, consider $g$ and $h$ such that $b_{g},\dots, b_{h-1}$ are the elements of the path on $C$. Let $b_{g}',\dots, b_{h-1}'$ be the balls of $w'$ directly north of $b_{g},\dots, b_{h-1}$. Now $b_{g-1}$ is clearly west of $b_g'$ and it must be northeast of some element of $C$, so it is north of $b_{g-1}$. Thus $b_{g-1}$ is northwest of $b_g'$. So $(b_0,\dots, b_{g-1},b_g',\dots, b_{h-1}',b_h,\dots, b_k)$ is a path in $w'$ of the same $r$-worth $l$. 

If $b$ is southwest of $C$ and $\tilde{b}$ is northeast of $C$, the argument is the same, but we need to take $b_g',\dots, b_{h-1}'$ directly east of $b_g,\dots, b_{h-1}$. The argument for the statement in the second paragraph is exactly the same.
\end{proof}

\subsubsection{Effect of shifting a channel on backward numbering}

In the next two lemmas, we address the first part of Theorem \ref{thm:how backward numbering changes}, i.e. we describe the backward numbering on $C'$.

\begin{lemma}
\label{lem: backward numbering pm1}
Choose a ball $b\in C$. Let $b'\in C'$ be the ball directly north of $b$, and let $b''\in C$ be the ball directly west of $b'$. Then either $d'(b') = d(b)$ or $d'(b') = d(b'')$.
\end{lemma}
\begin{proof}
We will show that $d'(b')\leqslant d(b)$ and that $d(b'')\leqslant d'(b')$. Since $d(b'') = d(b) - 1$, this will imply the desired conclusion. We will only do the first inequality since the proof of the second one is exactly the same. 

Consider a reverse path in normal form $p=(b_0=b,b_1,\dots, b_k)$ starting at $b$ and having minimal $r$-worth. For each $i$ with $b_i\in C$, let $b_i'$ be the ball of $C'$ directly north of $b_i$. Either $b_k\in C$ or not. If $b_k\in C$, then $(b_0'=b',b_1', \dots, b_k')$ is a reverse path starting at $b'$ and it has at most the $r$-worth of $p$ (it has the same number of steps and $d_0(b_k')\leqslant d_0(b_k)$. Now suppose $b_k\notin C$. Let $h$ be such that $b_h$ is the last ball of $p$ on $C$. Then $(b_0'=b',b_1', \dots, b_h', b_{h+1}, \dots, b_k)$ is a reverse path starting at $b'$ and it has the same $r$-worth as $p$ (since it has the same number of steps and the same ending).
\end{proof} 

\begin{lemma} 
\label{lem: correct backward numbering}
Choose a ball $b\in C$. Let $b'\in C'$ be the ball directly north of $b$, and let $b''\in C$ be the ball directly west of $b'$. The channel $C$ is $N$-terminating if and only if $d'(b') = d(b'')$.
\end{lemma} 
\begin{proof}
Suppose $d'(b') = d(b'')$. Consider a reverse path $p'=(b_0'=b', b_1',\dots, b_k')$ in normal form (with respect to the channel $C'$) of minimal $r$-worth. For each $i$ with $b_i'\in C'$, let $b_i$ be the ball of $C$ directly south of $b_i'$. There are three cases depending on whether $b_k'$ is on $C'$, southwest of $C'$ or northeast of $C'$.  

Suppose $b_k'$ is on $C'$. Consider the reverse path $p=(b_0=b, b_1,\dots, b_k)$. Let $d_0$ be the stream numbering; we don't need to make a distinction between $w$ and $w'$ since for balls present in both partial permutations the stream numbering is the same. By Remark \ref{rmk:last ball is terminal}, $b_k'$ is terminal, so we have $d_0(b_k') = d'(b_k') = d(b_k'') = d(b_{k-1})$. So $S^{(d_0(b_k') + 1)}$ is not northwest of $b_k'$ and hence not northwest of $b_{k-1}$. The path $p$ has the same number of steps as the path $p'$, so by the assumption that $d(b) > d'(b')$, we have $d_0(b_k) > d_0(b_k')$. Hence the element of the stream $S^{(d_0(b_{k})+1)}$ must be northwest of $b_{k}$. So $S^{(d_0(b_{k})+1)}$ is south of $b_{k-1}$, and hence $b_{k-1}$ is $N$-terminal.

Suppose $b_k'$ is southwest of $C'$. Let $h$ be the index of the last ball $b_h'$ of $p'$ on $C'$; so the balls $b_{h+1}',\ldots, b_k'$ are strictly southwest of $C'$. We know that $b_{h+1}'$ is southeast of $b_h'$ and southwest of $C'$. As shown in the proof of Lemma \ref{lem:d is the sw numbering}, this implies $b_{h+1}'$ is southeast of $b_h$. Then $(b_0=b,b_1,\dots, b_h, b_{h+1}',\dots, b_k')$ is a reverse path for $w$ starting at $b$ and having the same $r$-worth as $p'$. This contradicts the assumption that $d(b) (= d(b'') + 1) = d'(b')+1$.

Suppose $b_k'$ is northeast of $C'$. Let $h$ be the index of the last ball $b_h'$ of $p'$ on $C'$. Then $p=(b_0=b,b_1,\dots, b_{h-1}, b_{h+1}',\dots, b_k')$ is a reverse path for $w$ starting at $b$ and having $r$-worth one greater than than of $p'$ (we have no guarantee that $b_{h+1}'$ is south of $b_h$, so we could not include $b_h$ in the path). Since $d(b) = d'(b')+1$, Lemma \ref{lem:minimal r worth} guarantees that $p$ is a minimal $r$-worth reverse path from $b$, and thus $b_k'$ is terminal for $w$ (and of course it was terminal for $w'$). If it were $W$-terminal, then the element of $C'$ in its zig-zag would also be $W$-terminal. Thus we could have chosen $p'$ to stay on $C'$, reducing the problem to a previous case. This proves the first half of the statement.

Now we prove that if $C$ is $N$-terminating then $d'(b') = d(b'')$. Consider a minimal $r$-worth reverse path $p=(b_0 = b'',b_1,\dots,b_k)$ from $b''$ in normal form such that $b_k$ is $N$-terminal. For each $i$ with $b_i\in C$, let $b_i'''$ be the ball directly east of $b_i$. If $b_k\in C$ then, since $S^{(d(b_k))}$ is south of $b_k$, we have $d'(b_k''')\leqslant d(b_k)$ and Lemma \ref{lem: backward numbering pm1} finishes the proof. If $b_k$ is strictly southwest of $C$, then there is a $N$-terminal ball on $C$, so one can replace $p$ by the reverse path along the channel and reduce the problem to the previous case.  Finally, if $b_k$ is strictly northeast of $C$, let $i$ be the largest index such that $b_i\in C$. Then $(b_0'''=b',b_1',\dots,b_i''',b_{i+1},\dots, b_k)$ is a reverse path in $w'$ starting at $b'$ and having $r$-worth equal to $d(b'')$. Hence $d'(b')\leqslant d(b'')$ and Lemma \ref{lem: backward numbering pm1} again finishes the proof.
\end{proof}

Combining the lemma with its reflection in the main diagonal yields the following corollary.
\begin{cor}
\label{cor: terminating sides}
The channel $C$ is $N$-terminating if and only if $C'$ is not $W$-terminating.
\end{cor}

In the next lemma we describe the backward numbering strictly northeast of $C$; it turns out to be the same for $w$ and $w'$. 

\begin{lemma}
\label{lem:ne of C}
Suppose $C$ is not $N$-terminating. Then for any ball $b\in\B_w$ strictly northeast of $C$, $b$ is also a ball of $w'$ and $d'(b) = d(b)$.
\end{lemma}
\begin{proof}
Let $b\in\B_w$ be strictly northeast of $C$; of course $b$ is a ball of $w'$ since $w'$ is obtained from $w$ by shifting $C$. We want to show that for any reverse path (in normal form) in $w$ starting at $b$, there exists a reverse path in $w'$ of at most the same $r$-worth, and vice versa. We omit the second part of the proof since it is nearly identical to the first one.

Consider a reverse path $p$ (in normal form) starting at $b$ in $w$. If the path does not intersect $C$, then the same path has the same $r$-worth in $w'$. Thus we now assume $p$ intersects $C$.

Suppose $p = (b_0=b,b_1,\dots, b_k)$. Let $h$ be the first index such that $b_h\in C$. Let $b_h'$ be the ball of $C'$ directly north of $b_h$. Note that $b_{h-1}$ is strictly northeast of $C$, and every ball which is northeast of $C$ and northwest of $b_h$ is also northwest of $b_h'$. Choose a path $(b_h',\dots, b_l')$ of smallest $r$-worth in $w'$. Since $d(b_h) = d'(b_h')$, the $r$-worth of the reverse path $(b_0,\dots, b_{h-1}, b_h',\dots, b_l')$ is at most the same as the $r$-worth of $p$.
\end{proof}

In the next three statements we describe the backward numbering strictly southwest of $\widetilde{C}$ (recall the definition in Theorem \ref{thm:how backward numbering changes}); it turns out to be the same for $w$ and $w'$. 

\begin{lemma}
\label{lem: southwest of Ctilde}
Suppose $C$ is not $N$-terminating and $b$ is some ball southwest of $\widetilde{C}$. Then there exists a path of minimal $r$-worth such that every element of the path is southwest of $\widetilde{C}$.
\end{lemma}
\begin{proof}
Suppose $b$ is a ball of $w$ southwest of $\widetilde{C}$. Consider the semi-bounded reverse walk starting at $b$ and staying southwest of $\widetilde{C}$. If the walk is forced to stop, then we are done since it must stop at a $W$-terminal element. Otherwise, it necessarily reaches some channel $B$ of $w$; let $p=(b_0=b,b_1, \dots,b_k)$ be the initial part of the walk such that $b_k$ is the first element in $B$. Now $b_k$ is southwest of $C$ (in fact, southwest of $\widetilde{C}$), hence $b_k$ is $W$-terminating. Consider a reverse path $q$ (in normal form) starting at $b_k$ and having minimal $r$-worth. The path $q$ may be chosen to stay (weakly) southwest of $B$, since if it ended northeast of $B$, the ball of $B$ in the same zig-zag as the ending would be $W$-terminal, hence one could just follow $B$ to the terminal ball. The path formed by following $p$ and then following $q$ satisfies the desired properties.
\end{proof}

\begin{cor}
Suppose $C$ is not $N$-terminating and $b$ is some ball southwest of $\widetilde{C}$. Then $d'(b)\leqslant d(b)$.
\end{cor}
\begin{proof}
By Lemma \ref{lem: southwest of Ctilde}, there is a path of minimal $r$-worth (in $w$) southwest of $\widetilde{C}$. Such path is also a path in $w'$.
\end{proof}

\begin{lemma}
\label{lem:sw of Ctilde}
Suppose $C$ is not $N$-terminating and $b$ is some ball southwest of $\widetilde{C}$. Then $d'(b) = d(b)$.
\end{lemma}
\begin{proof}
Consider a reverse path $p'$(in normal form) for $w'$ which starts at $b$. We will show that there exists a reverse path $p$ for $w$ which starts at $b$ and has at most the same $r$-worth as $p'$. This implies $d(b)\leqslant d'(b)$, and an application of the previous corollary will finish the proof.

If $p'$ did not intersect $C'$ then there is nothing to prove; hence we assume that it did. Let $(b_0'=b, b_1',\dots, b_k')$ be the initial segment of $p'$ such that $b_k'$ is the first element of the path on $C'$ (hence $b_{k-1}'$ is southwest of $C'$). Let $b_k$ be the element of $C$ directly west of $b_k'$; then $d(b_k) < d'(b_k')$ by Lemmas \ref{lem: backward numbering pm1} and \ref{lem: correct backward numbering}. Let $(b_k, b_{k+1},\dots, b_l)$ be a reverse path of minimal $r$-worth starting at $b_k$. Then the reverse path $(b_0'=b, b_1',\dots, b_{k-1}',b_k,b_{k+1},\dots, b_l)$ is a reverse path in $w$ of smaller $r$-worth than $p'$.
\end{proof}

In fact, the proof of the above lemma shows that a minimal $r$-worth path in $w'$ must not intersect $C'$. Finally, in the next two results, we finish the description of the backward numbering of $w'$.

\begin{lemma}
Consider a ball $b\in\B_w$ which is northeast of $\widetilde{C}$ and strictly southwest of $C$. If $d(b) = k$, then $b$ is east of the element $c$ of $\widetilde{C}$ in $Z_{k+1}$. 
\end{lemma}
\begin{proof}
Since $b$ is northeast of $\widetilde{C}$, it is northeast of the ball of $\widetilde{C}$ in $Z_k$, which is north of $c$. Thus $b$ is north of $c$. Suppose $b$ is west of $c$. Let $c'$ be the ball of $\bk{S}{w}$ directly north of $b$; $c'$ is northeast of $\widetilde{C}$ and (strictly) southwest of $C$. If there exists a channel of $w$ strictly southwest of $C$, let $(c_0=c', c_1, \dots, c_l)$ be the walk bounded by it and $C$ from $c'$ to $\widetilde{C}$. Otherwise, let $(c_0=c', c_1, \dots, c_l)$ be the walk semi-bounded by $C$ from $c'$ to $\widetilde{C}$. Then the path $(c, c_0, c_1,\dots, c_l)$ is a path from $\widetilde {C}$ to itself which visits consecutive zig-zags, stays southwest of $C$, and contains elements strictly northeast of $\widetilde{C}$. This contradicts the fact that $\widetilde{C}$ was the northwest channel of $\bk{S}{w}$ southwest of $C$.
\end{proof}

\begin{rmk}
\label{rmk:touching channels}
Notice an interesting consequence of the previous lemma. First, there cannot be $W$-terminal balls northeast of $\widetilde{C}$ and strictly southwest of $C$. If in every zig-zag there are balls in the region described, then of course there cannot be $W$-terminal balls northeast of $C$. Moreover, no reverse path which visits consecutive zig-zags can cross $\widetilde{C}$ in this case. So in this case $C$ could not have been $W$-terminating. Thus if $C$ is $W$-terminating then there exists a ball of $C$ directly east of a ball of $\widetilde{C}$.
\end{rmk}

\begin{lemma}
\label{lem:ne of Ctilde and sw of C}
Suppose $C$ is not $N$-terminating. Consider a ball $b\in\B_w$ which is northeast of $\widetilde{C}$ and strictly southwest of $C$. Then $d'(b) = d(b) + 1$.
\end{lemma}
\begin{proof}
We want to show that $d'(b)\leqslant d(b)+1$ and $d(b)\leqslant d'(b)-1$. In terms of reverse paths, we want to show that for every reverse path in $w$ from $b$ there exists a path in $w'$ whose $r$-worth is at most one greater, and for every reverse path in $w'$ from $b$ there exists a path in $w$ whose $r$-worth is at least one less. We only do the first of these since the second one is exactly the same. Consider a reverse path $p=(b_0 = b, b_1, \dots, b_k)$ in $w$ of minimal $r$-worth in normal form. 

As a path of minimal $r$-worth, $p$ visits consecutive zig-zags. Moreover, by the previous lemma, if $b_i$ is northeast of $\widetilde{C}$ and strictly southwest of $C$ then $b_{i+1}$ is northeast of $\widetilde{C}$ (provided $i<k$) and $b_i$ is not $W$-terminal. Hence some $b_i$ must be northeast of $C$. By Lemma \ref{lem:jump off the bridge} we only need to consider the case when $p$ intersects $C$. Suppose $b_i$ is the first ball of $p$ on $C$. Let $b_i'$ be the ball of $C'$ directly east of $b_i$. Let $(b_i', b_{i+1}',\dots, b_l')$ be a reverse path of minimal $r$-worth starting at $b_i'$. Then, since $d(b_i) = d'(b_i') - 1$ the path $(b_0,\dots, b_{i-1},b_i',\dots, b_l')$ has $r$-worth at most one greater than that of $p$.
\end{proof}

\begin{proof}[Proof of Theorem \ref{thm:how backward numbering changes}]
\label{pf:how backward numbering changes}
The theorem follows by Lemmas \ref{lem: correct backward numbering}, \ref{lem:ne of C}, \ref{lem:sw of Ctilde} and \ref{lem:ne of Ctilde and sw of C}.
\end{proof}

\subsubsection{Effect of moving a channel on the result of the backward step}

\begin{lemma}
Let $w$ be a partial permutation and $S$ be a compatible stream whose flow is equal to the width of the Shi poset of $w$. Choose a non-$N$-terminating river of $w$ and let $C$ be its northeast channel. Let $\widetilde{C}$ be the northeast channel of $\bk{S}{w}$ which is southwest of $C$. Then $\widetilde{C}$ is the northeast channel in its river.
\end{lemma}
\begin{proof}
This follows from an easy adaptation of Theorem \ref{thm:dist nonfirst channels} to the case of arbitrary channel numberings, noting that the indexing river here is northeast of $C$.
\end{proof}

\begin{prop}
\label{prop:shifting channel right}
Let $w$ be a partial permutation and $S$ be a compatible stream whose flow is equal to the width of the Shi poset of $w$. Choose a non-$N$-terminating river of $w$ and let $C$ be its northeast channel. Let $w'=w\sh{1}_C$. Let $C' = C\sh{1}$. Let $\widetilde{C}$ be the northeast channel of $\bk{S}{w}$ which is southwest of $C$. Then $\bk{S}{w'} = \bk{S}{w}\sh{1}_{\widetilde{C}}$.
\end{prop}
\begin{proof}
Notice a couple of consequences of Theorem \ref{thm:how backward numbering changes}. Suppose $c$ and $c'$ are two consecutive elements of $\widetilde{C}$. Then all the elements on the zig-zag of $c$ strictly between $C$ and $\widetilde{C}$ lie east of $c'$ (otherwise the numbering described in the theorem would not be monotone). Similarly, any ball of $w$ in the zig-zag of $b\in C$ which is northeast of $b$ must lie north of the element of $C$ preceding $b$.

Now using Theorem \ref{thm:how backward numbering changes} we can analyze the zig-zags for $\bk{S}{w'}$, which can be seen in Figure 
\ref{fig:moving channel}. The zig-zags corresponding to $\bk{S}{w}$ are shown in black. The zig-zags corresponding to $\bk{S}{w'}$ are shown in green. The black balls are balls of $w$, the green balls are balls of $C'\subseteq \B_{w'}$, and the magenta balls are balls of $\widetilde{C}\subseteq\B_{\bk{S}{w}}$. As can be seen from the picture, northwest corners of the black zig-zags and green zig-zags almost always coincide, except the collection of black corners giving $\widetilde{C}$ gets replaced by a collection of green corners corresponding to a shift of $\widetilde{C}$ by 1. Hence the two partial permutations differ by a shift of $\widetilde{C}$, as desired.
\end{proof}

\begin{figure}
\centering\resizebox{.4\textwidth}{!}{\input{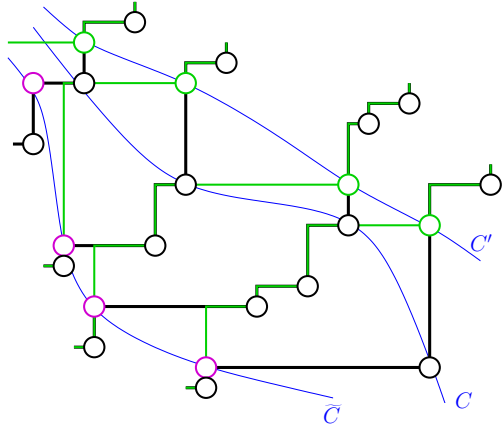}} 
\caption{The zig-zags of $\bk{S}{w}$ (in black) and of $\bk{S}{w'}$ (in green).}
\label{fig:moving channel}
\end{figure}

We have described what happens when we shift a non-$N$-terminating river by $1$. By reflecting in the main diagonal, we can thus describe what happens when we shift a non-$W$-terminating river by $-1$. To finish the description, we need to understand what happens when we shift an $N$-terminating river by $1$, or equivalently when we shift an $W$-terminating river by $-1$.

\begin{prop}
\label{prop:shifting channel by -1}
Let $w$ be a partial permutation and $S$ be a compatible stream whose flow is equal to the width of the Shi poset of $w$. Choose a $W$-terminating river of $w$ and let $C$ be its southwest channel. Let $w'=w\sh{-1}_C$. Let $C'=C\sh{-1}$. 
Let $B$ be the northeast channel of $\bk{S}{w}$ southwest of $C$; let $\widetilde{C}$ be the southwest channel of the river of $B$. Then $\bk{S}{w'}=\bk{S}{w}\sh{-1}_{\widetilde{C}}$.
\end{prop}
\begin{proof}
By Corollary \ref{cor: terminating sides} we know that $C'$ is not $N$-terminating. Thus, given Proposition \ref{prop:shifting channel right}, we need to show that if $B'$ is the northwest channel of $\bk{S}{w'}$ southwest of $C'$, then $\widetilde{C} = B'\sh{1}$. We already know that $B'\sh{1}$ is the southwest channel of its river and that $\bk{S}{w} = \left(\bk{S}{w'}\right)\sh{1}_{B'}$. If the channel $B$ did not intersect $B'\sh{1}$ then we would have a contradiction with the definition of $B'$. So $B'$ is part of the river of $B$, as desired.
\end{proof}

We can also understand what happens when we shift a channel multiple times.
\begin{prop}
\label{prop:moving channels}
Suppose $w$ is a partial permutation, $S$ is a compatible stream whose flow is equal to the width of the Shi poset of $w$, and $C$ is the northeast channel of a river which is not $N$-terminating. Let $\widetilde{C}$ be the northeast channel of $\bk{S}{w}$ southwest of $C$. Let $C' = C\sh{1}$, $\widetilde{C}' = \widetilde{C}\sh{1}$, and $w' = w\sh{1}_C$. Let $C''$ be the northeast channel of the river of $w'$ containing $C'$, and let $\widetilde{C}''$ be the northwest channel of the river of $\bk{S}{w'}$ containing $\widetilde{C}'$. Let $w'' = w'\sh{1}_{C''}$. 

Then $$\bk{S}{w''} = \bk{S}{w'}\sh{1}_{\widetilde{C}''}.$$
\end{prop}

\begin{proof} 
Suppose first that $C'$ is also not $N$-terminating; let $\widetilde{C}'$ be the result of shifting $\widetilde{C}$ by 1. Suppose we shift again the river containing $C'$; let $C''$ be the northeast channel of that river. One only needs to prove that $\widetilde{C}'$ is part of the same river as the northeast channel $\widetilde{C}''$ of $\bk{S}{w'}$ southwest of $C''$.

Since the indexing river is located northeast of $\widetilde{C}''$, it is sufficient to show that there exists a path from $\widetilde{C}''$ to $\widetilde{C}'$ which visits consecutive zig-zags. Now $C'$ and $C''$ are part of the same river, there exists a path from $C''$ to itself which intersects $C'$ such that the corresponding river numbering drops by 1 at each step. Since the backward numbering and the river numbering have the same period, it must be that the backward numbering also drops down by one at each step. Hence there is a path from $C''$ to $C'$ such that the backward numbering drops down by 1 at each step. Continue this path along $C'$. Then a (bounded or semi-bounded) walk leads from $\widetilde{C}''$ to $\widetilde{C}'$ and visits consecutive zig-zags. 

Now suppose $C'$ is $N$-terminating (and of course, still, $W$-terminating). Since $C$ contains a ball directly east of $\widetilde{C}$, $C'$ contains a ball directly east of a ball of $\widetilde{C}'$. Now this ball must be $N$- and $W$-terminating, so $\widetilde{C}'$ is part of the indexing river. Now any ball directly north of a ball of $C''$ is also part of the indexing river. Hence, using the reflection of Proposition \ref{prop:shifting channel by -1} in the main diagonal, we see that the channel to be shifted to get $\bk{S}{w''}$ is exactly the northwest channel of the indexing river, i.e. $\widetilde{C}''$.
\end{proof}

\subsection{Weyl symmetry}
\label{sec:weyl symmetry proofs}

We are now ready to tackle the proof of our last main result, namely Theorem \ref{thm:weyl}.

\begin{lemma}
Suppose $w$ is a partial permutation, $T$ is a compatible stream, and $S$ is a stream compatible with $w' := \bk{T}{w}$. Moreover suppose that $T$ is concurrent to $S$ (in particular, $T$ and $S$ have the same flow). Then the indexing river of $w'$ is both $N$-terminating and $W$-terminating with respect to $S$.
\end{lemma}
\begin{proof}
Choose a proper numbering $d_S$ of $S$; let $d_T$ be the backward numbering of $T$ with respect to $S$. Let $d$ be the induced numbering of $w'$ from the backward step with respect to $T$. Let $d'$ be the backward numbering of $w'$ with respect to $S$. By Lemma \ref{lem:backward is proper}, $d$ is a monotone numbering, and by construction, for each $i$, $S^{(i)}$ is northwest of all balls labeled $i$. So, by Remark \ref{rmk:lower bound for backward numbering} we know that $d(b)\leqslant d'(b)$ for all $b\in\B_{w'}$. 

Since $T$ is concurrent to $S$, we can choose $i$ such that $T^{(i)}$ is north of $S^{(i+1)}$. Consider $b\in\B_{w'}$ directly east of $T^{(i)}$. By definition of $d'$, $d'(b)\leqslant i = d(b)$, so $d'(b)= d(b)$. Thus $b$ is $N$-terminal with respect to $S$. Now choose a path $(b_0=b,b_1, \dots, b_k)$ such that $b_k$ is on the indexing river of $w'$ and $d$ decreases by $1$ at each step. Since $d'$ must also decrease at each step, we have $d(b_i) = d'(b_i)$ for all $i$. So the indexing river of $w'$ is $N$-terminating. By reflecting the arguments in the main diagonal, we see that the indexing river of $w'$ is $W$-terminating.
\end{proof}

\begin{lemma}
Suppose $w$ is a partial permutation and $S$ is a compatible stream whose flow is equal to the width of the Shi poset of $w$. Suppose $w$ has a river $R$ which is both $W$-terminating and $N$-terminating. Let $w'=w\sh{-1}_R$ and $w''=w\sh{1}_R$. Then $\bk{S}{w'}$ differs from $\bk{S}{w}$ by shifting the indexing river by $-1$, and $\bk{S}{w''}$ differs from $\bk{S}{w}$ by shifting the indexing river by $1$.
\end{lemma}
\begin{proof}
It is sufficient to prove the statement for $w'$; the one for $w''$ will follow by reflection in the main diagonal. Let $C$ be the southwest channel of $R$. We only need to show that the channel $\widetilde{C}$ of $\bk{S}{w}$ from Proposition \ref{prop:shifting channel by -1} is part of the indexing river; this is true because the ball directly west of a ball of $C$ is clearly part of the indexing river (recall the characterization of the indexing river in the proof of Proposition \ref{prop:backward numbering is channel}). 
\end{proof}

\begin{thm}
Suppose $w$ is a partial permutation, $T=\str{r}{A,B}$ is a compatible stream and $S=\str{r'}{A',B'}$ is compatible with $\bk{T}{w}$. Let $z$ be the unique integer such that $\str{z}{A,B}$ is concurrent to $\str{0}{A',B'}$. Let $\widetilde{T} = \str{r'+z}{A,B}$ and $\widetilde{S} = \str{r-z}{A',B'}$. Then $\bk{S}{\bk{T}{w}} = \bk{\widetilde{S}}{\bk{\widetilde{T}}{w}}$.
\end{thm}
\begin{proof}
By Proposition \ref{prop:unique concurrent} we know that $\widetilde{T}$ is concurrent to $S$. By the previous two lemmas, Theorem \ref{thm:moving streams}, and Remark \ref{rmk:moving streams} and Proposition \ref{prop:moving channels}, we know that for any $k$, $\bk{S}{\bk{\widetilde{T}\sh{k}}{w}} = \bk{S\sh{k}}{\bk{\widetilde{T}}{w}}$. In particular, for $k = r-z-r'$ we have $\widetilde{T}\sh{k} = T$ and $S\sh{k} = \widetilde{S}$, finishing 	the proof.
\end{proof}

\begin{proof}[Proof of Theorem \ref{thm:weyl}]
\label{pf:weyl}
Since every permutation is a product of transpositions of adjacent elements, repeated application of the previous theorem shows that 
\[\Psi(P,Q,\rho) = \Psi(P,Q,\rho').\] 
Now because $\rho'$ is dominant, $\Phi(\Psi(P,Q,\rho')) = (P, Q,\rho')$. 
\end{proof}

\bibliographystyle{plain}
\bibliography{affine}

\end{document}